\documentclass[11pt,a4paper]{amsart}

\usepackage{comment}
\usepackage{amsmath}
\usepackage{amssymb}
\usepackage{enumitem}
\usepackage{hyperref}
\usepackage[capitalise]{cleveref}
\usepackage{tikz}
\usepackage{subcaption}
\usepackage{mathtools}

\usetikzlibrary{calc}
\usetikzlibrary{intersections}

\newtheorem{theorem}{Theorem} [section]
\newtheorem{lemma}[theorem]{Lemma}

\newtheorem{corollary}[theorem]{Corollary}

\DeclareMathOperator{\cl}{cl}
\DeclareMathOperator{\si}{si}
\DeclareMathOperator{\co}{co}
\DeclarePairedDelimiter\ceil{\lceil}{\rceil}
\DeclarePairedDelimiter\floor{\lfloor}{\rfloor}

\setenumerate{label=\rm(\roman*),midpenalty=1000}

\title[Detachable pairs in $3$-connected matroids]{Detachable pairs in $3$-connected matroids and simple $3$-connected graphs}

\author[N.\ Brettell]{Nick Brettell}
\address{School of Mathematics and Statistics, Victoria University of Wellington, New Zealand}
\email{nick.brettell@vuw.ac.nz}
\author[C.\ Semple]{Charles Semple}
\address{School of Mathematics and Statistics, University of Canterbury, New Zealand}
\email{charles.semple@canterbury.ac.nz}
\author[G.\ Toft]{Gerry Toft}
\address{School of Mathematics and Statistics, University of Canterbury, New Zealand}
\email{gerrytoft3@gmail.com}

\thanks{The first two authors were supported by the New Zealand Marsden Fund.  The third author was supported by a University of Canterbury Doctoral Scholarship.}

\date{\today}


\begin{document}
\begin{abstract}
    Let $M$ be a $3$-connected matroid.
    A pair $\{e,f\}$ in $M$ is \emph{detachable} if $M \backslash e \backslash f$ or $M / e / f$ is $3$-connected.
    Williams (2015) proved that if $M$ has at least 13 elements, then at least one of the following holds: $M$ has a detachable pair, $M$ has a $3$-element circuit or cocircuit, or $M$ is a spike.
    We address the case where $M$ has a $3$-element circuit or cocircuit, to obtain a characterisation of when a matroid with at least 13 elements has a detachable pair.
    As a consequence, we characterise when a simple $3$-connected graph $G$ with $|E(G)| \ge 13$ has a pair of edges $\{e,f\}$ such that $G/e/f$ or $G \backslash e\backslash f$ is simple and $3$-connected.
\end{abstract}

\maketitle

\section{Introduction}
Tutte's Wheels-and-Whirls Theorem~\cite{wheelsandwhirls} and Seymour's Splitter Theorem~\cite{seymoursplitter,tansplitter} are fundamental tools in matroid theory.
They have been used to prove several important results, including Seymour's decomposition theorem for regular matroids~\cite{seymoursplitter}, and the excluded-minor characterisations for GF$(4)$-representable matroids~\cite{quaternary} and near-regular matroids~\cite{nearreg}.
Tutte's Wheels-and-Whirls Theorem states that a $3$-connected matroid $M$ has an element $e$ such that either $M \backslash e$ or $M/e$ is $3$-connected, unless $M$ is a wheel or a whirl.
Such a result, which ensures the existence of an element, or elements, that can be removed while preserving a connectivity condition, is known as a \emph{chain theorem}.
A \emph{splitter theorem} additionally ensures that, given a minor~$N$, removing the element or elements also preserves the existence of a minor isomorphic to $N$.
In this paper, the focus is a chain theorem that preserves the property of being $3$-connected after deleting or contracting a pair of elements.

Let $M$ be a $3$-connected matroid.
A pair $\{e,f\} \subseteq E(M)$ is called a \emph{detachable pair} if either $M \backslash e \backslash f$ or $M / e / f$ is $3$-connected. Note that, since $M$ is $3$-connected, if $M \backslash e \backslash f$ is $3$-connected, then $M \backslash e$ and $M \backslash f$ are both $3$-connected. Similarly, if $M / e / f$ is $3$-connected, then $M / e$ and $M / f$ are both $3$-connected.
A \emph{triangle} is a circuit of size three, a \emph{triad} is a cocircuit of size three, and a \emph{spike} is a matroid with a partition into pairs such that the union of any two of these pairs is both a circuit and a cocircuit.
Williams~\cite{williamsdetpairs} (see also \cite[Section~7]{detpairs3}) proved the following:
\begin{theorem}[Williams 2015]\label{detpairs_notriangles}
	Let $M$ be a $3$-connected matroid with $|E(M)| \geq 13$.
    Then at least one of the following holds:
	\begin{enumerate}
		\item $M$ has a detachable pair,
        \item $M$ has a triangle or a triad, or
		\item $M$ is a spike.
	\end{enumerate}
\end{theorem}

For a $3$-connected matroid to have an element~$e$ such that $M \backslash e$ or $M/e$ is $3$-connected, one potential obstacle is the presence of triangles or triads: after contracting an element in a triangle, or deleting an element in a triad, the resulting matroid is not $3$-connected.
It is for this reason that wheels and whirls are exceptional in Tutte's Wheels-and-Whirls Theorem: for a wheel or whirl, there is a cyclic ordering on the ground set such that the sets formed by three consecutive elements alternate between triangles and triads, and so every element is in both a triangle and a triad.
Similarly, for a $3$-connected matroid to have a detachable pair, triangles and triads can again be problematic.
This issue is bypassed by case~(ii) of \cref{detpairs_notriangles}.
In this paper, we describe precisely the matroids with at least 13 elements that have no detachable pairs, including those with triangles or triads.
In particular, we prove the following:
\begin{theorem}\label{detachable_main}
	Let $M$ be a $3$-connected matroid with $|E(M)| \geq 13$. Then precisely one of the following holds:
	\begin{enumerate}
		\item $M$ has a detachable pair,
		\item $M$ is a wheel or a whirl,
		\item $M$ is an accordion,
        \item $M$ is an even-fan-spike, or an even-fan-spike with tip and cotip,
        \item $M$ or $M^*$ is an even-fan-paddle,
        \item $M$ or $M^*$ is a triad-paddle or a hinged triad-paddle,
        \item $M$ is a tri-paddle-copaddle, or
        \item $M$ or $M^*$ is a quasi-triad-paddle with
            \begin{enumerate}[label=\rm(\alph*)]
                \item an augmented-fan petal,
                \item a co-augmented-fan petal,
                \item a quad petal, or
                \item a near-quad petal.
            \end{enumerate}
    \end{enumerate}
\end{theorem}

The matroids in this theorem are illustrated as geometric representations in Figures \ref{efs_figs}--\ref{accordion_figs}. While formal definitions of these matroids are deferred until \Cref{exceptionalmatroids}, we make some initial observations.
Each family described in one of (ii)--(viii) contains only matroids that have no detachable pairs, and these matroids can be arbitrarily large.
Fans feature prominently in many of these families (a fan is a subset with an ordering such that the sets formed by three consecutive elements alternate between triangles and triads).
A reader may wonder why spikes do not explicitly appear in \cref{detachable_main}; we view a spike as an example of an even-fan-spike (where each even fan has size two).

The notion of a flower can be used to describe matroids with crossing $3$-separations~\cite{flowers}.
Let $\Phi = (P_1,P_2,\ldots,P_m)$ be a partition of the ground set of a $3$-connected matroid $M$.
Then $\Phi$ is a \emph{flower} in $M$ if each $P_i$ consists of at least two elements and is $3$-separating, and each $P_i \cup P_{i+1}$ is $3$-separating, where all subscripts are interpreted modulo $m$.
A flower $\Phi$ is an \emph{anemone} if $\bigcup_{s\in S} P_s$ is $3$-separating for every subset $S$ of $\{1,2,\dotsc,m\}$.
The matroids described in (iv)--(viii) can be viewed as anemones where, with a few particular exceptions, each petal is either a triad, or a fan with even length. 
The matroids described in (iii) have \emph{path-width three}; that is, there is an ordering $(e_1,e_2,\dotsc,e_n)$ of $E(M)$ such that $\{e_1,\dotsc,e_i\}$ is $3$-separating for each positive integer $i \le n$.

Now let $G$ be a simple $3$-connected graph.
A pair $\{e,f\} \subseteq E(G)$ is called a \emph{detachable pair} if either $G \backslash e \backslash f$ or $G / e / f$ is simple and $3$-connected.
As a consequence of \cref{detachable_main}, we obtain the following chain theorem for simple $3$-connected graphs:
\begin{theorem}\label{detachable_graph}
	Let $G$ be a simple $3$-connected graph with $|E(G)| \geq 13$. Then precisely one of the following holds:
	\begin{enumerate}
		\item $G$ has a detachable pair,
        \item $G$ is a wheel,
        \item $G$ is a mutant wheel, 
        \item $G$ is a twisted wheel or a warped wheel, 
        \item $G$ is a multi-wheel,
        \item $G$ is a stretched wheel,
        \item $G$ is isomorphic to $K_{3,m}$, for some $m \geq 5$, or 
        \item $G$ is isomorphic to $K^a_{3,m}$ or $K^b_{3,m}$, for some $m \geq 3$.
    \end{enumerate}
\end{theorem}

\noindent
These graphs are illustrated in \cref{exceptionalgraphsfig}; definitions are given in \cref{exceptionalgraphs}.

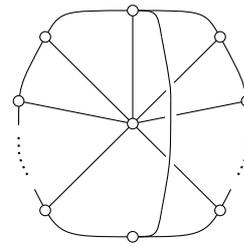
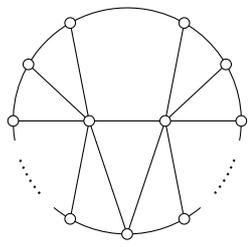
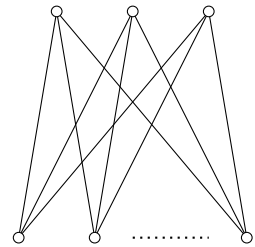
\begin{figure}
	\centering
	\begin{subfigure} {0.45\textwidth}
		\centering
		\begin{tikzpicture}
		\draw (160:1.5) arc (160:-145:1.5);
		\draw[thick, dotted] (-155:1.5) arc (-155:-190:1.5);
		\draw (0:0) -- (90:1.5);
		\draw (0:0) -- (35:1.5);
		\draw (0:0) -- (-20:1.5);
		\draw (0:0) -- (-75:1.5);
		\draw (0:0) -- (-130:1.5);
		\draw (0:0) -- (145:1.5);

		\draw[fill=white] (0:0) circle (2pt);
		\draw[fill=white] (90:1.5) circle (2pt);
		\draw[fill=white] (35:1.5) circle (2pt);
		\draw[fill=white] (-20:1.5) circle (2pt);
		\draw[fill=white] (-75:1.5) circle (2pt);
		\draw[fill=white] (-130:1.5) circle (2pt);
		\draw[fill=white] (145:1.5) circle (2pt);
		\end{tikzpicture}
	\subcaption{A wheel.}
	\end{subfigure}
	\begin{subfigure} {0.45\textwidth}
		\centering
		\begin{tikzpicture}
		\draw (160:1.5) arc (160:-145:1.5);
		\draw[thick, dotted] (-155:1.5) arc (-155:-190:1.5);
		\draw (0:0) -- (90:1.5);
		\draw (0:0) -- (35:1.5);
		\draw (0:0) -- (-20:1.5);
		\draw (0:0) -- (-75:1.5);
		\draw (0:0) -- (-130:1.5);
		\draw (0:0) -- (145:1.5);
		
		\draw (90:0.75) -- (35:1.5);
		\draw (35:0.75) -- (-20:1.5);
		
		\draw[fill=white] (0:0) circle (2pt);
		\draw[fill=white] (90:1.5) circle (2pt);
		\draw[fill=white] (35:1.5) circle (2pt);
		\draw[fill=white] (-20:1.5) circle (2pt);
		\draw[fill=white] (-75:1.5) circle (2pt);
		\draw[fill=white] (-130:1.5) circle (2pt);
		\draw[fill=white] (145:1.5) circle (2pt);
		
		\draw[fill=white] (90:0.75) circle (2pt);
		\draw[fill=white] (35:0.75) circle (2pt);
		\end{tikzpicture}
	\subcaption{A mutant wheel.}
	\end{subfigure}
	\begin{subfigure} {0.45\textwidth}
		\centering
		\begin{tikzpicture}
		\coordinate (a) at (-1.5,0);
		\coordinate (b) at (1.5,0);
		\coordinate (c) at (0,3);
		\coordinate (d) at (0,0.5);
		
		\coordinate (ac1) at ($(c)!0.2!(d)$);
		\coordinate (ac2) at ($(c)!0.4!(d)$);
		\coordinate (ac3) at ($(c)!0.8!(d)$);
		
		\coordinate (bd1) at ($(a)!0.2!(b)$);
		\coordinate (bd2) at ($(a)!0.4!(b)$);
		\coordinate (bd3) at ($(a)!0.8!(b)$);
		
		\draw (c) -- ($(ac2)+(0,-0.2)$);
		\draw ($(ac3)+(0,0.2)$) -- (d);
		
		\draw (d) -- (a);
		\draw (d) -- (bd1);
		\draw (d) -- (bd2);
		\draw (d) -- (bd3);
		\draw (d) -- (b);
		
		\draw (a) -- ($(bd2)+(0.2,0)$);
		\draw ($(bd3)+(-0.2,0)$) -- (b);
		
		\draw (a) -- (c);
		\draw (a) -- (ac1);
		\draw (a) -- (ac2);
		\draw (a) -- (ac3);
		\draw (c) -- (b);
		
		\draw[thick, dotted] ($(bd2)+(0.325,0)$) -- ($(bd3)+(-0.325,0)$);
		\draw[thick, dotted] ($(ac2)+(0,-0.325)$) -- ($(ac3)+(0,0.325)$);
		
		\draw[fill=white] (a) circle (2pt);
		\draw[fill=white] (b) circle (2pt);
		\draw[fill=white] (c) circle (2pt);
		\draw[fill=white] (d) circle (2pt);
		
		\draw[fill=white] (ac1) circle (2pt);
		\draw[fill=white] (ac2) circle (2pt);
		\draw[fill=white] (ac3) circle (2pt);
		
		\draw[fill=white] (bd1) circle (2pt);
		\draw[fill=white] (bd2) circle (2pt);
		\draw[fill=white] (bd3) circle (2pt);
		\end{tikzpicture}
	\subcaption{A twisted wheel.}
	\end{subfigure}
	\begin{subfigure} {0.45\textwidth}
		\centering
		\begin{tikzpicture}
		\coordinate (a) at (-1.5,0);
		\coordinate (b) at (0,0);
		\coordinate (c) at (1.5,0);
		
		\coordinate (d1) at (0,1.5);
		\coordinate (d2) at (0,1.25);
		\coordinate (d3) at (0,1);
		\coordinate (d4) at (0,0.25);
		
		\coordinate (e1) at (0,-1.5);
		\coordinate (e2) at (0,-1.25);
		\coordinate (e3) at (0,-1);
		\coordinate (e4) at (0,-0.25);
		
		\draw (d1) -- ($(d3)+(0,-0.15)$);
		\draw ($(d4)+(0,0.15)$) -- (b);
		\draw[thick,dotted] ($(d3)+(0,-0.25)$) -- ($(d4)+(0,0.25)$);
		\draw (e1) -- ($(e3)+(0,0.15)$);
		\draw ($(e4)+(0,-0.15)$) -- (b);
		\draw[thick,dotted] ($(e3)+(0,0.25)$) -- ($(e4)+(0,-0.25)$);
		
		\draw (a) -- (c);
		
		\draw (a) -- (d1);
		\draw (a) -- (d2);
		\draw (a) -- (d3);
		\draw (a) -- (d4);
		\draw (a) -- (e1);
		
		\draw (c) -- (d1);
		\draw (c) -- (e1);
		\draw (c) -- (e2);
		\draw (c) -- (e3);
		\draw (c) -- (e4);
		
		\draw[fill=white] (a) circle (2pt);
		\draw[fill=white] (b) circle (2pt);
		\draw[fill=white] (c) circle (2pt);
		\draw[fill=white] (d1) circle (2pt);
		\draw[fill=white] (d2) circle (2pt);
		\draw[fill=white] (d3) circle (2pt);
		\draw[fill=white] (d4) circle (2pt);
		\draw[fill=white] (e1) circle (2pt);
		\draw[fill=white] (e2) circle (2pt);
		\draw[fill=white] (e3) circle (2pt);
		\draw[fill=white] (e4) circle (2pt);
		\end{tikzpicture}
	\subcaption{A warped wheel.}
	\end{subfigure}
	\begin{subfigure} {0.45\textwidth}
		\centering
		\begin{tikzpicture}
		\draw (0,1.5) -- (0,0);
		
		\draw[thick,dotted] (-1.5,-0.18) .. controls (-1.5,-0.45) .. (-1.36,-0.73);
		\draw (-1.29,-0.87) .. controls (-0.975,-1.5) .. (0,-1.5);
		
		\draw (-1.5,-0.02) .. controls (-1.5,0.5) .. (-1.15,1.15);
		\draw (-1.15,1.15) .. controls (-0.95,1.5) .. (0,1.5);
		
		\draw[thick,dotted] (1.5,-0.18) .. controls (1.5,-0.45) .. (1.36,-0.73);
		\draw (1.29,-0.87) .. controls (0.975,-1.5) .. (0,-1.5);
		
		\draw (1.5,-0.02) .. controls (1.5,0.5) .. (1.15,1.15);
		\draw (1.15,1.15) .. controls (0.95,1.5) .. (0,1.5);
		
		\draw (0,0) -- (-1.15,1.15);
		\draw (0,0) -- (-1.5,0.3);
		\draw (0,0) -- (-1.15,-1.15);
		
		\draw (0,0) -- (1.15,1.15);
		\draw (0,0) -- (1.5,0.3);
		\draw (0,0) -- (1.15,-1.15);
		
		\draw (0,0) -- (0.38,1.15);
		\draw (0,0) -- (0.5,0.3);
		\draw (0,0) -- (0.38,-1.15);
		
		\filldraw[white] (0.48,0.5) circle (2pt);
		\filldraw[white] (0.5,0.1) circle (2pt);
		\filldraw[white] (0.48,-0.5) circle (2pt);
		\draw[thick,dotted] (0.5,-0.18) .. controls (0.5,-0.45) .. (0.45,-0.7);
		\draw (0.43,-0.855) .. controls (0.32,-1.5) .. (0,-1.5);
		
		\draw (0.5,-0.02) .. controls (0.5,0.5) .. (0.38,1.15);
		\draw (0.38,1.15) .. controls (0.32,1.5) .. (0,1.5);
		
		\draw[thick, dotted] (0.75,0) -- (1.3,0);
		
		\draw[fill=white] (-1.15,1.15) circle (2pt);
		\draw[fill=white] (-1.5,0.3) circle (2pt);
		\draw[fill=white] (-1.15,-1.15) circle (2pt);
		
		\draw[fill=white] (1.15,1.15) circle (2pt);
		\draw[fill=white] (1.5,0.3) circle (2pt);
		\draw[fill=white] (1.15,-1.15) circle (2pt);
		
		\draw[fill=white] (0.38,1.15) circle (2pt);
		\draw[fill=white] (0.5,0.3) circle (2pt);
		\draw[fill=white] (0.38,-1.15) circle (2pt);
		
		\draw[fill=white] (0,1.5) circle (2pt);
		\draw[fill=white] (0,0) circle (2pt);
		\draw[fill=white] (0,-1.5) circle (2pt);
		\end{tikzpicture}
	\subcaption{A multi-wheel.}
	\end{subfigure}
	\begin{subfigure} {0.45\textwidth}
		\centering
		\begin{tikzpicture}
		\draw (0,1.5) -- (0,0);
		
		\draw[thick,dotted] (-1.5,-0.18) .. controls (-1.5,-0.45) .. (-1.36,-0.73);
		\draw (-1.29,-0.87) .. controls (-0.975,-1.5) .. (0,-1.5);
		
		\draw (-1.5,-0.02) .. controls (-1.5,0.5) .. (-1.15,1.15);
		\draw (-1.15,1.15) .. controls (-0.95,1.5) .. (0,1.5);
		
		\draw[thick,dotted] (1.5,-0.18) .. controls (1.5,-0.45) .. (1.36,-0.73);
		\draw (1.29,-0.87) .. controls (0.975,-1.5) .. (0,-1.5);
		
		\draw (1.5,-0.02) .. controls (1.5,0.5) .. (1.15,1.15);
		\draw (1.15,1.15) .. controls (0.95,1.5) .. (0,1.5);
		
		\draw (0,0) -- (-1.15,1.15);
		\draw (0,0) -- (-1.5,0.3);
		\draw (0,0) -- (-1.15,-1.15);
		
		\draw (0,0) -- (1.15,1.15);
		\draw (0,0) -- (1.5,0.3);
		\draw (0,0) -- (1.15,-1.15);
		
		\filldraw[white] (0.48,0.5) circle (2pt);
		\filldraw[white] (0.5,0.1) circle (2pt);
		\filldraw[white] (0.48,-0.5) circle (2pt);
		\draw (0.5,0.3) .. controls (0.5,-0.5) .. (0.38,-1.15);
		\draw (0.38,-1.15) .. controls (0.32,-1.5) .. (0,-1.5);
		
		\draw (0.5,0.3) .. controls (0.5,0.5) .. (0.38,1.15);
		\draw (0.38,1.15) .. controls (0.32,1.5) .. (0,1.5);
		
		\draw[fill=white] (-1.15,1.15) circle (2pt);
		\draw[fill=white] (-1.5,0.3) circle (2pt);
		\draw[fill=white] (-1.15,-1.15) circle (2pt);
		
		\draw[fill=white] (1.15,1.15) circle (2pt);
		\draw[fill=white] (1.5,0.3) circle (2pt);
		\draw[fill=white] (1.15,-1.15) circle (2pt);
		
		\draw[fill=white] (0,1.5) circle (2pt);
		\draw[fill=white] (0,0) circle (2pt);
		\draw[fill=white] (0,-1.5) circle (2pt);
		\end{tikzpicture}
	\subcaption{A degenerate multi-wheel.}
	\end{subfigure}
	\begin{subfigure} {0.45\textwidth}
		\centering
		\begin{tikzpicture}
		\draw (-10:1.5) arc (-10:190:1.5);
		\draw (-130:1.5) arc (-130:-50:1.5);
		\draw[thick,dotted] (-20:1.5) arc (-20:-40:1.5);
		\draw[thick,dotted] (-140:1.5) arc (-140:-160:1.5);
		
		\draw (0:1.5) -- (0.5,0);
		\draw (30:1.5) -- (0.5,0);
		\draw (60:1.5) -- (0.5,0);
		\draw (-60:1.5) -- (0.5,0);
		\draw (-90:1.5) -- (0.5,0);
		
		\draw (120:1.5) -- (-0.5,0);
		\draw (150:1.5) -- (-0.5,0);
		\draw (180:1.5) -- (-0.5,0);
		\draw (-120:1.5) -- (-0.5,0);
		\draw (-90:1.5) -- (-0.5,0);
		
		\draw (-0.5,0) -- (0.5,0);
		
		\draw[fill=white] (0:1.5) circle (2pt);
		\draw[fill=white] (30:1.5) circle (2pt);
		\draw[fill=white] (60:1.5) circle (2pt);
		\draw[fill=white] (120:1.5) circle (2pt);
		\draw[fill=white] (150:1.5) circle (2pt);
		\draw[fill=white] (180:1.5) circle (2pt);
		
		\draw[fill=white] (-0.5,0) circle (2pt);
		\draw[fill=white] (0.5,0) circle (2pt);
		
		\draw[fill=white] (-60:1.5) circle (2pt);
		\draw[fill=white] (-120:1.5) circle (2pt);
		\draw[fill=white] (-90:1.5) circle (2pt);
		\end{tikzpicture}
		\subcaption{A stretched wheel.}
	\end{subfigure}
	\begin{subfigure} {0.45\textwidth}
		\centering
		\begin{tikzpicture}
		\coordinate (y) at (0,3);
		\coordinate (a1) at (-1,0);
		\coordinate (b1) at (0,0);
		\coordinate (c1) at (1,0);
		
		\coordinate (a3) at ($(-1.5,0)-(y)$);
		\coordinate (b3) at ($(-0.5,0)-(y)$);
		\coordinate (c3) at ($(1.5,0)-(y)$);
		
		\draw (a3) -- (a1);
		\draw (a3) -- (b1);
		\draw (a3) -- (c1);
		\draw (b3) -- (a1);
		\draw (b3) -- (b1);
		\draw (b3) -- (c1);
		\draw (c3) -- (a1);
		\draw (c3) -- (b1);
		\draw (c3) -- (c1);
		\draw[thick, dotted] ($(b3)+(0.5,0)$) -- ($(c3)+(-0.5,0)$);
		
		\draw[black,fill=white] (a1) circle (2pt);
		\draw[black,fill=white] (b1) circle (2pt);
		\draw[black,fill=white] (c1) circle (2pt);
		
		\draw[black,fill=white] (a3) circle (2pt);
		\draw[black,fill=white] (b3) circle (2pt);
		\draw[black,fill=white] (c3) circle (2pt);
		\end{tikzpicture}
	\subcaption{$K_{3,m}$.}
	\end{subfigure}
	\begin{subfigure} {0.45\textwidth}
		\centering
		\begin{tikzpicture}
		\coordinate (y) at (0,1.5);
		\coordinate (a1) at (-1,0);
		\coordinate (b1) at (0,0);
		\coordinate (c1) at (1,0);
		
		\coordinate (a2) at (0,0.75);
		\coordinate (b2) at (0,1.5);
		
		\coordinate (a3) at ($(-1.5,0)-(y)$);
		\coordinate (b3) at ($(-0.5,0)-(y)$);
		\coordinate (c3) at ($(1.5,0)-(y)$);
		
		\draw (a3) -- (a1);
		\draw (a3) -- (b1);
		\draw (a3) -- (c1);
		\draw (b3) -- (a1);
		\draw (b3) -- (b1);
		\draw (b3) -- (c1);
		\draw (c3) -- (a1);
		\draw (c3) -- (b1);
		\draw (c3) -- (c1);
		\draw[thick, dotted] ($(b3)+(0.5,0)$) -- ($(c3)+(-0.5,0)$);
		
		\draw (a2) -- (a1);
		\draw (a2) -- (b1);
		\draw (a2) -- (c1);
		\draw (b2) -- (a2);
		\draw (b2) -- (a1);
		\draw (b2) -- (c1);
		
		\draw[black,fill=white] (a1) circle (2pt);
		\draw[black,fill=white] (b1) circle (2pt);
		\draw[black,fill=white] (c1) circle (2pt);
		
		\draw[black,fill=white] (a2) circle (2pt);
		\draw[black,fill=white] (b2) circle (2pt);
		
		\draw[black,fill=white] (a3) circle (2pt);
		\draw[black,fill=white] (b3) circle (2pt);
		\draw[black,fill=white] (c3) circle (2pt);
		\end{tikzpicture}
	\subcaption{$K^a_{3,m}$}
	\end{subfigure}
	\begin{subfigure} {0.45\textwidth}
		\centering
		\begin{tikzpicture}
		\coordinate (y) at (0,1.5);
		\coordinate (a1) at (-1,0);
		\coordinate (b1) at (0,0);
		\coordinate (c1) at (1,0);
		
		\coordinate (a2) at ($(-0.5,0)+(y)$);
		\coordinate (b2) at ($(0.5,0)+(y)$);
		
		\coordinate (a3) at ($(-1.5,0)-(y)$);
		\coordinate (b3) at ($(-0.5,0)-(y)$);
		\coordinate (c3) at ($(1.5,0)-(y)$);
		
		\draw (a3) -- (a1);
		\draw (a3) -- (b1);
		\draw (a3) -- (c1);
		\draw (b3) -- (a1);
		\draw (b3) -- (b1);
		\draw (b3) -- (c1);
		\draw (c3) -- (a1);
		\draw (c3) -- (b1);
		\draw (c3) -- (c1);
		\draw[thick, dotted] ($(b3)+(0.5,0)$) -- ($(c3)+(-0.5,0)$);
		
		\draw (a2) -- (a1);
		\draw (a2) -- (b1);
		\draw (b2) -- (b1);
		\draw (b2) -- (c1);
		\draw (a1) .. controls ($(b1)+(0,0.6)$) .. (c1);
		\draw (a2) -- (b2);
		
		\draw[black,fill=white] (a1) circle (2pt);
		\draw[black,fill=white] (b1) circle (2pt);
		\draw[black,fill=white] (c1) circle (2pt);
		
		\draw[black,fill=white] (a2) circle (2pt);
		\draw[black,fill=white] (b2) circle (2pt);
		
		\draw[black,fill=white] (a3) circle (2pt);
		\draw[black,fill=white] (b3) circle (2pt);
		\draw[black,fill=white] (c3) circle (2pt);
		\end{tikzpicture}
	\subcaption{$K^b_{3,m}$}
	\end{subfigure}
\caption{Simple $3$-connected graphs with no detachable pairs.} \label{exceptionalgraphsfig}
\end{figure}

\Cref{detpairs_notriangles} was an important step towards a splitter theorem for detachable pairs in $3$-connected matroids having no triangles or triads, which was later obtained by Brettell, Whittle, and Williams \cite{detpairs1,detpairs2,detpairs3}.
The initial motivation for these results was as a tool towards proving excluded-minor characterisations for particular classes of representable matroids~\cite{almostfrag1,almostfrag2,2regexminors,BP23}.
For these classes, the excluded minors are closed under \emph{$\Delta$-$Y$ exchange}~\cite{deltay}, an operation that transforms a triangle into a triad.
In this setting, it suffices to be able to obtain a detachable pair after a $\Delta$-$Y$ or $Y$-$\Delta$ exchange,
so an analysis of when matroids with triangles or triads have detachable pairs was unnecessary.
However, we foresee \cref{detachable_main} as a tool towards proving excluded-minor characterisations for classes of matroids that are not closed under $\Delta$-$Y$ exchange.
It is also a step towards a splitter theorem for detachable pairs in $3$-connected matroids (that may have triangles or triads).

We note that \cref{detachable_main} resolves \cite[Conjecture~7.5]{detpairs3} which, although ``correct in spirit'', was missing the exceptional matroids given by cases (iii), (vii), and (viii), and part of (iv) and (vi).
Similarly, \cref{detachable_graph} resolves \cite[Conjecture~7.6]{detpairs3}, which was missing the exceptional graphs given in cases~(iii), (vi), and (viii), and part of case (iv).

This paper is structured as follows.  In \cref{exceptionalmatroids,exceptionalgraphs}, we describe the exceptional matroids and graphs that appear in \cref{detachable_main,detachable_graph}, respectively.
We present some preliminaries in \cref{det_prelims,lemmas}.
The remainder of the paper consists of a proof of \cref{detachable_main}.
In \cref{disjoint_fans,fans_intersecting} we address cases where the matroid $M$ has distinct maximal fans: one of length at least four, and the other of length at least three.
First, in \cref{disjoint_fans}, we assume the fans are disjoint and both start with triangles, or both start with triads.
Then, in \cref{fans_intersecting}, we assume the fans have non-empty intersection.
Next, in \cref{fans}, we consider the remaining cases where $M$ has a fan of length at least four.
In \cref{no_fans}, it remains only to consider the case where $M$ has no $4$-element fans.
Finally, in \cref{graph_proof}, we prove \cref{detachable_graph} by showing that the graphs in this theorem correspond to the matroids in \cref{detachable_main} that are graphic.

\section{Matroids with no detachable pairs} \label{exceptionalmatroids}

We now formally define the $3$-connected matroids with no detachable pairs, appearing in \cref{detachable_main}.
In order to do so, we first recall the notions of flowers and fans.
For a positive integer $m$, we let $[m]$ denote the set $\{1,2,\dotsc,m\}$, and let $[0]=\emptyset$.
Let $M$ be a matroid with ground set $E$.
The \emph{local connectivity} of subsets $X, Y \subseteq E$ is 
\[\sqcap(X,Y) = r(X) + r(Y) - r(X \cup Y).\] 
The \emph{connectivity} of $X$ in $M$ is
\[
    \lambda(X) = \sqcap(X,E-X) = r(X) + r(E-X) - r(M).
\]

Let $M$ be a $3$-connected matroid.
Recall that a partition $\Phi = (P_1,P_2,\ldots,P_m)$ of $E(M)$, for some $m \ge 2$, is a \emph{flower} if, for all $i \in [m]$, we have that $|P_i| \geq 2$, and $\lambda(P_i) \le 2$, and $\lambda(P_i \cup P_{i+1}) \le 2$, where subscripts are interpreted modulo $m$.
The sets $P_i$ are called \emph{petals} of $\Phi$.
The flower~$\Phi$ is an \emph{anemone} if, for all subsets $J$ of $[m]$, we have that $\lambda(\bigcup_{j \in J} P_j) \le 2$.
Furthermore, when $m \ge 3$, we say the anemone $\Phi$ is
\begin{enumerate}
	\item a \emph{paddle} if $\sqcap(P_i,P_j) = 2$ for all distinct $i,j \in [m]$,
    \item \emph{spike-like} if $\sqcap(P_i,P_j) = 1$ for all distinct $i,j \in [m]$, and
    \item a \emph{copaddle} if $\sqcap(P_i,P_j) = 0$ for all distinct $i,j \in [m]$.
\end{enumerate}
Note that if $\Phi$ is a paddle in $M^*$, then it is a copaddle in $M$; whereas if $\Phi$ is spike-like in $M^*$, then it is also spike-like in $M$~\cite[Proposition~4.2]{flowers}.

Let $F$ be a subset of $E(M)$.
If $|F| \ge 3$ and $F$ has an ordering $(e_1,e_2,\ldots,e_{|F|})$ such that
\begin{enumerate}
    \item $\{e_1,e_2,e_3\}$ is a triangle or a triad, and
    \item for all $i \in [|F|-3]$, if $\{e_i,e_{i+1},e_{i+2}\}$ is a triangle, then $\{e_{i+1},e_{i+2},e_{i+3}\}$ is a triad, and if $\{e_i,e_{i+1},e_{i+2}\}$ is a triad, then $\{e_{i+1},e_{i+2},e_{i+3}\}$ is a triangle,
\end{enumerate}
then $F$ is a \emph{fan} of $M$, and we call $(e_1,e_2,\ldots,e_{|F|})$ a \emph{fan ordering} of $F$ with \emph{ends} $e_1$ and $e_{|F|}$.
If $|F|=2$, then we also say that $F$ is a fan (where any ordering is a fan ordering of $F$).
The \emph{length} of a fan $F$ is $|F|$.
A fan is \emph{even} if it has even length, otherwise it is \emph{odd}.
For a fan $F$, we say that $e \in F$ is an \emph{end} of $F$ if there is a fan ordering of $F$ for which $e$ is an end.
Note that when a fan $F$ has length at least $4$, it has a unique pair of ends \cite{oxleywufans}.

The exceptional matroids in \cref{detachable_main} fall roughly into four categories:
firstly, spike-like anemones where each petal is an even fan (\cref{efs_figs});
secondly, paddles where each petal is an even fan (\cref{efp_figs});
thirdly, paddles that can be constructed by attaching particular matroids to $M(K_{3,m})$ for some $m \geq 2$ (\cref{mk3mpaddles_fig});
and finally, a family of matroids with path-width three that we call accordions 
(\cref{accordion_figs}).

Throughout the remainder of this section, $M$ is a $3$-connected matroid.

\subsection*{Even-fan-spikes}
We say that $M$ is a \emph{(tipless) non-degenerate even-fan-spike} with \emph{partition} $\Phi$ if $M$ has a spike-like anemone $\Phi = (P_1,P_2,\ldots,P_m)$, for $m \geq 3$, such that
\begin{enumerate}
    \item for every $i \in [m]$, the petal $P_i$ is an even fan with length at least two, and
    \item for all distinct $i,j \in [m]$, the fans $P_i$ and $P_j$ have orderings $(p_1,p_2,\ldots,p_{|P_i|})$ and $(q_1,q_2,\ldots,q_{|P_j|})$ respectively such that $\{p_1,p_2,p_3\}$ is a triad or $|P_i|=2$, and $\{q_1,q_2,q_3\}$ is a triad or $|Q_i|=2$, and $\{p_1,p_2,q_1,q_2\}$ is a circuit and $\{p_{|P_i|-1},p_{|P_i|},q_{|P_j|-1},q_{|P_j|}\}$ is a cocircuit.
\end{enumerate}
See \cref{efs_fig} for an example with $m  = 4$. Note that $\bigcap_{i \in [m]} \cl(P_i) = \emptyset$ and $\bigcap_{i \in [m]} \cl^*(P_i) = \emptyset$.
We call each $P_i$ a \emph{petal} of the non-degenerate even-fan-spike. If each petal has size two, then $M$ is a (tipless) spike.

\begin{figure}
	\centering
    \begin{subfigure}[t]{0.45\textwidth}
		\centering
		\begin{tikzpicture}
		\coordinate (tip) at (0,0);
		\coordinate (leg1) at (-28:2);
		\coordinate (leg2) at (-67:2);
		\coordinate (leg3) at (-113:2);
		\coordinate (leg4) at (-152:2);
		
		\coordinate (wheel1a) at ($(tip)+(47:1)$);
		\coordinate (wheel1b) at ($(leg2)+(47:1)$);
		
		\coordinate (wheel2a) at ($(tip)+(-257:0.8)$);
		\coordinate (wheel2c) at ($(leg4)+(-227:0.8)$);
		\coordinate (wheel2b) at ($($(wheel2a)!0.5!(wheel2c)$)+(-242:0.8)$);
		
		\draw[white, name path=wheel1] (wheel1a) -- (wheel1b);
		
		\draw[name path=backleg] (tip) -- (leg1);
		\draw (tip) -- (leg2);
		\draw (tip) -- (leg3);
		\draw (tip) -- (leg4);
		
		\draw (tip) -- (wheel1a);
		\draw (wheel1b) -- (leg2);
		
		\filldraw[white, name intersections={of=wheel1 and backleg}=-] (intersection-1) circle (5pt);
		
		\draw (wheel1a) -- (wheel1b);
		
		\draw (tip) -- (wheel2a) -- (wheel2b) -- (wheel2c) -- (leg4);
		
		\filldraw ($(tip)!0.5!(leg1)$) circle (2pt);
		\filldraw (leg1) circle (2pt);
		\filldraw ($(tip)!0.5!(leg3)$) circle (2pt);
		\filldraw (leg3) circle (2pt);
		
		\filldraw ($(tip)!0.5!(wheel1a)$) circle (2pt);
		\filldraw (wheel1a) circle (2pt);
		\filldraw ($(wheel1a)!0.5!(wheel1b)$) circle (2pt);
		\filldraw (wheel1b) circle (2pt);
		\filldraw ($(leg2)!0.5!(wheel1b)$) circle (2pt);
		\filldraw (leg2) circle (2pt);
		
		\filldraw ($(tip)!0.5!(wheel2a)$) circle (2pt);
		\filldraw (wheel2a) circle (2pt);
		\filldraw ($(wheel2a)!0.5!(wheel2b)$) circle (2pt);
		\filldraw (wheel2b) circle (2pt);
		\filldraw ($(wheel2b)!0.5!(wheel2c)$) circle (2pt);
		\filldraw (wheel2c) circle (2pt);
		\filldraw ($(wheel2c)!0.5!(leg4)$) circle (2pt);
		\filldraw (leg4) circle (2pt);
		
		\filldraw[white] ($(tip)+(0,-2.2)$) circle (2pt);
		\end{tikzpicture}
		\subcaption{A non-degenerate even-fan-spike.} \label{efs_fig}
	\end{subfigure}
    \begin{subfigure}[t]{0.45\textwidth}
		\centering
		\begin{tikzpicture}
		\coordinate (tip) at (0,0);
		\coordinate (leg1) at (-28:2);
		\coordinate (leg2) at (-67:2);
		\coordinate (leg3) at (-113:2);
		\coordinate (leg4) at (-152:2);
		
		\coordinate (wheel1a) at ($(tip)+(47:1)$);
		\coordinate (wheel1b) at ($(leg2)+(47:1)$);
		
		\coordinate (wheel2a) at ($(tip)+(-257:0.8)$);
		\coordinate (wheel2c) at ($(leg4)+(-227:0.8)$);
		\coordinate (wheel2b) at ($($(wheel2a)!0.5!(wheel2c)$)+(-242:0.8)$);
		
		\draw[white, name path=wheel1] (wheel1a) -- (wheel1b);
		
		\draw[name path=backleg] (tip) -- (leg1);
		\draw (tip) -- (leg2);
		\draw (tip) -- (leg3);
		\draw (tip) -- (leg4);
		
		\draw (tip) -- (wheel1a);
		\draw (wheel1b) -- (leg2);
		
		\filldraw[white, name intersections={of=wheel1 and backleg}=-] (intersection-1) circle (5pt);
		
		\draw (wheel1a) -- (wheel1b);
		
		\draw (tip) -- (wheel2a) -- (wheel2b) -- (wheel2c) -- (leg4);
		
		\filldraw ($(tip)!0.5!(leg1)$) circle (2pt);
		\filldraw (leg1) circle (2pt);
		\filldraw ($(tip)!0.5!(leg3)$) circle (2pt);
		\filldraw (leg3) circle (2pt);
		
		\filldraw ($(tip)!0.5!(wheel1a)$) circle (2pt);
		\filldraw (wheel1a) circle (2pt);
		\filldraw ($(wheel1a)!0.5!(wheel1b)$) circle (2pt);
		\filldraw (wheel1b) circle (2pt);
		\filldraw ($(leg2)!0.5!(wheel1b)$) circle (2pt);
		\filldraw (leg2) circle (2pt);
		
		\filldraw ($(tip)!0.5!(wheel2a)$) circle (2pt);
		\filldraw (wheel2a) circle (2pt);
		\filldraw ($(wheel2a)!0.5!(wheel2b)$) circle (2pt);
		\filldraw (wheel2b) circle (2pt);
		\filldraw ($(wheel2b)!0.5!(wheel2c)$) circle (2pt);
		\filldraw (wheel2c) circle (2pt);
		\filldraw ($(wheel2c)!0.5!(leg4)$) circle (2pt);
		\filldraw (leg4) circle (2pt);
		
		\filldraw (tip) circle (2pt) node[left] {$x$};
		\filldraw ($(tip)+(0,-2.2)$) circle (2pt) node[above] {$y$};
		\end{tikzpicture}
		\subcaption{A non-degenerate even-fan-spike with tip~$x$ and cotip~$y$.} \label{efs_tip_fig}
	\end{subfigure}
    \begin{subfigure}[t]{0.45\textwidth}
		\centering
		\begin{tikzpicture}
		\coordinate (wheel1a) at (-0.75,0);
		\coordinate (wheel1e) at (1.25,0);
		\coordinate (wheel2a) at (-1.25,0);
		\coordinate (wheel2d) at (1.25,0);
		
		\coordinate (wheel1b) at ($(wheel1a)+(105:0.8)$);
		\coordinate (wheel1d) at ($(wheel1e)+(75:0.8)$);
		\coordinate (wheel1c) at ($($(wheel1b)!0.5!(wheel1d)$)+(0,0.8)$);
		
		\coordinate (wheel2b) at ($(wheel2a)+(1,-0.7)$);
		\coordinate (wheel2c) at ($(wheel2d)+(1,-0.7)$);
		
		\draw (wheel1a) -- (wheel1b) -- (wheel1c) -- (wheel1d) -- (wheel1e);
		\draw (wheel2a) -- (wheel2b) -- (wheel2c) --(wheel2d);
		\draw ($(wheel2a)+(-0.2,0)$) -- ($(wheel2d)+(0.2,0)$);
		
		\filldraw (wheel1a) circle (2pt);
		\filldraw ($(wheel1a)!0.5!(wheel1b)$) circle (2pt);
		\filldraw (wheel1b) circle (2pt);
		\filldraw ($(wheel1b)!0.5!(wheel1c)$) circle (2pt);
		\filldraw (wheel1c) circle (2pt);
		\filldraw ($(wheel1c)!0.5!(wheel1d)$) circle (2pt);
		\filldraw (wheel1d) circle (2pt);
		\filldraw ($(wheel1d)!0.5!(wheel1e)$) circle (2pt);
		
		\filldraw (wheel2a) circle (2pt);
		\filldraw ($(wheel2a)!0.5!(wheel2b)$) circle (2pt);
		\filldraw (wheel2b) circle (2pt);
		\filldraw ($(wheel2b)!0.5!(wheel2c)$) circle (2pt);
		\filldraw (wheel2c) circle (2pt);
		\filldraw ($(wheel2c)!0.5!(wheel2d)$) circle (2pt);
		\end{tikzpicture}
		\subcaption{A degenerate even-fan-spike.} \label{deg_efs_fig}
	\end{subfigure}
    \begin{subfigure}[t]{0.45\textwidth}
		\centering
		\begin{tikzpicture}
		\coordinate (leg1) at (-48:2.1);
		\coordinate (leg2) at (-132:2.1);
		
		\coordinate (wheel1a) at ($(0,0)!0.8!(leg2)$);
		\coordinate (wheel1b) at ($(wheel1a)+(-207:0.8)$);
		\coordinate (wheel1d) at ($(0,0)+(-237:0.8)$); 
		\coordinate (wheel1c) at ($($(wheel1b)!0.5!(wheel1d)$)+(-222:0.8)$);
		
		\coordinate (wheel2a) at ($(0,0)!0.8!(leg1)$);
		\coordinate (wheel2b) at ($(wheel2a)+(66:1)$);
		\coordinate (wheel2c) at ($(0,0)+(66:1)$);
		
		\draw (0,0) -- (leg1) -- (leg2) -- (0,0);
		\draw (wheel1a) -- (wheel1b) -- (wheel1c) -- (wheel1d) -- (0,0);
		\draw (wheel2a) -- (wheel2b) -- (wheel2c) -- (0,0);
		
		\filldraw (wheel1a) circle (2pt);
		\filldraw ($(wheel1a)!0.5!(wheel1b)$) circle (2pt);
		\filldraw (wheel1b) circle (2pt);
		\filldraw ($(wheel1b)!0.5!(wheel1c)$) circle (2pt);
		\filldraw (wheel1c) circle (2pt);
		\filldraw ($(wheel1c)!0.5!(wheel1d)$) circle (2pt);
		\filldraw (wheel1d) circle (2pt);
		\filldraw ($(wheel1d)!0.5!(0,0)$) circle (2pt);
		\filldraw (0,0) circle (2pt) node[left] {$x$};
		
		\filldraw (wheel2a) circle (2pt);
		\filldraw ($(wheel2a)!0.5!(wheel2b)$) circle (2pt);
		\filldraw (wheel2b) circle (2pt);
		\filldraw ($(wheel2b)!0.5!(wheel2c)$) circle (2pt);
		\filldraw (wheel2c) circle (2pt);
		\filldraw ($(wheel2c)!0.5!(0,0)$) circle (2pt);
		
		\filldraw ($(leg1)!0.5!(leg2)$) circle (2pt) node[above] {$y$};
		\end{tikzpicture}
		\subcaption{A degenerate even-fan-spike with tip~$x$ and cotip~$y$.} \label{deg_efs_tip_fig}
	\end{subfigure}
	\caption{Examples of even-fan-spikes.} \label{efs_figs}
\end{figure}
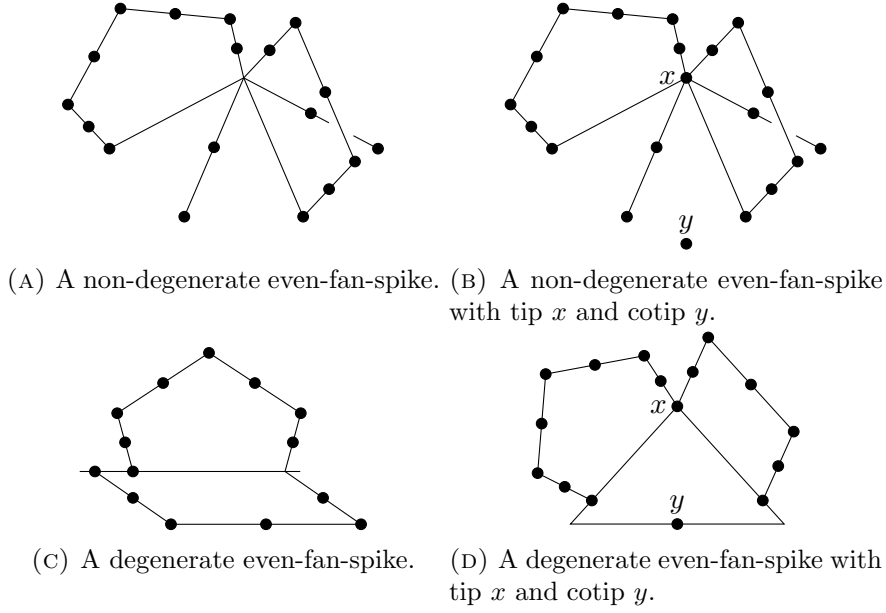

We say that $M$ is an \emph{non-degenerate even-fan-spike with tip and cotip} if
\begin{enumerate}
    \item $M$ has a spike-like anemone $\Phi = (P_1,P_2,\ldots,P_m)$ for $m \geq 3$, and
    \item there are distinct elements $x, y \in E(M)$ such that, for every $i \in [m]$, the petal $P_i \cup \{x,y\}$ is an even fan with length at least four, having ends $x$ and $y$.
\end{enumerate}
See \cref{efs_tip_fig} for an example with $m = 4$. We call $\Phi$ a \emph{partition} of the non-degenerate even-fan-spike with tip and cotip. Note that, up to swapping $x$ and $y$, we have $\bigcap_{i \in [m]} \cl(P_i) = \{x\}$ and $\bigcap_{i \in [m]} \cl^*(P_i) = \{y\}$; in this case, we call $x$ the \emph{tip}, and $y$ the \emph{cotip}.
If $|P_i \cup \{x,y\}| = 4$ for all $i \in [m]$, then $M$ is a spike with tip~$x$ and cotip~$y$.

We now consider the degenerate case, where 
$M$ has a flower $(P,Q)$ such that $P$ and $Q$ are even fans.  Note that we view these as ``even-fan-spikes'' even though $\sqcap(P,Q)=2$.
We say that $M$ is a \emph{(tipless) degenerate even-fan-spike} if $E(M)$ has a partition $(P,Q)$ such that
\begin{enumerate}
    \item $P$ and $Q$ are even fans with length at least four, and
    \item the fans $P$ and $Q$ have orderings $(p_1,p_2,\ldots,p_{|P|})$ and $(q_1,q_2,\ldots,q_{|Q|})$ respectively such that $\{p_1,p_2,p_3\}$ and $\{q_1,q_2,q_3\}$ are triads, $\{p_1,p_2,q_1,q_2\}$ is a circuit, and $\{p_{|P|-1},p_{|P|},q_{|Q|-1},q_{|Q|}\}$ is a cocircuit.
\end{enumerate}
An example is shown in \cref{deg_efs_fig}. We call $P$ and $Q$ the two \emph{petals}, and $(P,Q)$ the \emph{partition}, of the degenerate even-fan-spike.

Additionally, $M$ is a \emph{degenerate even-fan-spike with tip and cotip} if $E(M)$ has a partition $(P,Q,\{x,y\})$ such that $P \cup \{x,y\}$ and $Q \cup \{x,y\}$ are even fans of length at least four, with ends $x$ and $y$ (see \cref{deg_efs_tip_fig} for an example).
Note that, up to swapping $x$ and $y$, we have $\cl(P) \cap \cl(Q) = \{x\}$ and $\cl^*(P) \cap \cl^*(Q) = \{y\}$; in this case, we call $x$ the \emph{tip}, and $y$ the \emph{cotip}. We also call $(P,Q,\{x,y\})$ the \emph{partition} of the degenerate even-fan-spike with tip and cotip.

We say that $M$ is an \emph{even-fan-spike} (with tip and cotip) if $M$ is either a non-degenerate or degenerate even-fan-spike (with tip and cotip, respectively).
It is easily checked that even-fan-spikes and even-fan-spikes with tip and cotip have no detachable pairs.
We also note that if $M$ is an even-fan-spike (with tip and cotip) having partition $\Phi$, then $M$ is self-dual, and $M^*$ also has partition $\Phi$.

\subsection*{Even-fan-paddles}
The matroid $M$ is an \emph{even-fan-paddle} with \emph{partition} $(P_1,P_2,\ldots,P_m)$ if $(P_1,P_2,\ldots,P_m)$ is a paddle, for some $m \geq 3$, and there is an element $x \in P_m$, such that
\begin{enumerate}
    \item for all $i \in [m-1]$, the set $P_i \cup \{x\}$ is an even fan of length at least four with $x$ as an end;
    \item $P_m$ is an even fan of length at least two, and if $|P_m|=2$, then $m=3$; and
    \item for all distinct $i, j \in [m]$, there is a fan ordering $(p_1^i,p_2^i,\ldots,p^i_{|P_i|-1},x)$ of $P_i \cup \{x\}$ and a fan ordering $(p_1^j,p_2^j,\ldots,p_{|P_j|-1}^j,x)$ of $P_j \cup \{x\}$ such that the set $\{p^i_1,p^i_2,p^j_1,p^j_2\}$ is a circuit.
\end{enumerate}

An even-fan-paddle with partition $(P_1,P_2,\ldots,P_m)$ is \emph{degenerate} if $m=3$ and $|P_m|=2$, otherwise it is \emph{non-degenerate}. An example of a degenerate even-fan-paddle is shown in \cref{deg_efp_fig}, and examples of non-degenerate even-fan-paddles are shown in \cref{efp_fig,efp_four}.
For a non-degenerate even-fan-paddle with partition $(P_1,P_2,\ldots,P_m)$, we have $\bigcap_{i \in [m]} \cl(P_i) = \{x\}$ and $\bigcap_{i \in [m]} \cl^*(P_i) = \emptyset$; whereas for a degenerate even-fan-paddle with partition $(P_1,P_2,P_3)$, where $P_3=\{x,y\}$, we have $\bigcap_{i \in [m]} \cl(P_i) = \{x,y\}$ and $\bigcap_{i \in [m]} \cl^*(P_i) = \emptyset$.

\begin{figure}
	\centering
\begin{subfigure}[t]{0.45\textwidth}
\centering
\begin{tikzpicture}
\coordinate (wheel1a) at (-1,0);
\coordinate (wheel1e) at (1,0);
\coordinate (wheel1b) at ($(wheel1a)+(105:0.8)$);
\coordinate (wheel1d) at ($(wheel1e)+(75:0.8)$);
\coordinate (wheel1c) at ($($(wheel1b)!0.5!(wheel1d)$)+(0,0.8)$);

\draw ($(wheel1a)+(-0.2,0)$) -- ($(wheel1e)+(0.2,0)$);
\draw (wheel1a) -- (wheel1b) -- (wheel1c) -- (wheel1d) -- (wheel1e);

\coordinate (wheel2a) at (wheel1a);
\coordinate (wheel2d) at (wheel1e);
\coordinate (wheel2b) at ($(wheel2a)+(1,-0.7)$);
\coordinate (wheel2c) at ($(wheel2d)+(1,-0.7)$);

\draw (wheel2a) -- (wheel2b);
\draw[white, name path=line1] (wheel2b) -- (wheel2c);
\draw (wheel2c) -- (wheel2d);

\coordinate (wheel3a) at (wheel1a);
\coordinate (wheel3d) at (wheel1e);
\coordinate (wheel3b) at ($(wheel3a)+(-0.5,-1)$);
\coordinate (wheel3c) at ($(wheel3d)+(-0.5,-1)$);

\draw (wheel3a) -- (wheel3b);
\draw (wheel3b) -- (wheel3c);
\draw[name path=line2] (wheel3c) -- (wheel3d);

\filldraw[white, name intersections={of=line1 and line2}=-] (intersection-1) circle (5pt);

\draw (wheel2b) -- (wheel2c);

\filldraw (wheel1a) circle (2pt) node[below, xshift=1] {$x$};
\filldraw ($(wheel1a)!0.5!(wheel1b)$) circle (2pt);
\filldraw (wheel1b) circle (2pt);
\filldraw ($(wheel1b)!0.5!(wheel1c)$) circle (2pt);
\filldraw (wheel1c) circle (2pt);
\filldraw ($(wheel1c)!0.5!(wheel1d)$) circle (2pt);
\filldraw (wheel1d) circle (2pt);
\filldraw ($(wheel1d)!0.5!(wheel1e)$) circle (2pt);

\filldraw ($(wheel2a)!0.5!(wheel2b)$) circle (2pt);
\filldraw (wheel2b) circle (2pt);
\filldraw ($(wheel2b)!0.5!(wheel2c)$) circle (2pt);
\filldraw (wheel2c) circle (2pt);
\filldraw ($(wheel2c)!0.5!(wheel2d)$) circle (2pt);

\filldraw ($(wheel3a)!0.5!(wheel3b)$) circle (2pt);
\filldraw (wheel3b) circle (2pt);
\filldraw ($(wheel3b)!0.5!(wheel3c)$) circle (2pt);
\filldraw (wheel3c) circle (2pt);
\filldraw ($(wheel3c)!0.5!(wheel3d)$) circle (2pt);
\end{tikzpicture}
\subcaption{A non-degenerate even-fan-paddle.} \label{efp_fig}
\end{subfigure}
\begin{subfigure}[t]{0.45\textwidth}
\centering
\begin{tikzpicture}
\coordinate (wheel1a) at (-1,0);
\coordinate (wheel1e) at (1,0);
\coordinate (wheel1b) at ($(wheel1a)+(105:0.8)$);
\coordinate (wheel1d) at ($(wheel1e)+(75:0.8)$);
\coordinate (wheel1c) at ($($(wheel1b)!0.5!(wheel1d)$)+(0,0.8)$);

\draw ($(wheel1a)+(-0.2,0)$) -- ($(wheel1e)+(0.2,0)$);
\draw (wheel1a) -- (wheel1b) -- (wheel1c) -- (wheel1d) -- (wheel1e);

\coordinate (wheel2a) at (wheel1a);
\coordinate (wheel2d) at (wheel1e);
\coordinate (wheel2b) at ($(wheel2a)+(1,-0.7)$);
\coordinate (wheel2c) at ($(wheel2d)+(1,-0.7)$);

\draw (wheel2a) -- (wheel2b);
\draw (wheel2b) -- (wheel2c);
\draw (wheel2c) -- (wheel2d);

\filldraw (wheel1a) circle (2pt) node[below] {$x$};
\filldraw ($(wheel1a)!0.5!(wheel1b)$) circle (2pt);
\filldraw (wheel1b) circle (2pt);
\filldraw ($(wheel1b)!0.5!(wheel1c)$) circle (2pt);
\filldraw (wheel1c) circle (2pt);
\filldraw ($(wheel1c)!0.5!(wheel1d)$) circle (2pt);
\filldraw (wheel1d) circle (2pt);
\filldraw ($(wheel1d)!0.5!(wheel1e)$) circle (2pt);

\filldraw ($(wheel2a)!0.5!(wheel2b)$) circle (2pt);
\filldraw (wheel2b) circle (2pt);
\filldraw ($(wheel2b)!0.5!(wheel2c)$) circle (2pt);
\filldraw (wheel2c) circle (2pt);
\filldraw ($(wheel2c)!0.5!(wheel2d)$) circle (2pt);

\filldraw (0,0) circle (2pt) node[above] {$y$};
\filldraw[white] (0,-1) circle (2pt);
\end{tikzpicture}
\subcaption{A degenerate even-fan-paddle.} \label{deg_efp_fig}
\end{subfigure}
\begin{subfigure}[t]{0.45\textwidth}
	\centering
	\begin{tikzpicture}
	\coordinate (end1) at (-1.2,0);
	\coordinate (middle) at (0,0);
	\coordinate (end2) at (1.5,0);
	\coordinate (end3) at (-2.5,0);
	
	\draw ($(end3)+(-0.2,0)$) -- ($(end2)+(0.2,0)$);
	
	\coordinate (middle2) at (1.5,0.8);
	\coordinate (left2) at ($(end1)!0.6!(middle2)$);

	\coordinate (middle3) at (-1.5,1.5);
	\coordinate (left3) at ($(end3)!0.6!(middle3)$);
	
	\draw (end3) -- (middle3);
	\draw[name path=back1] (end1) -- (middle3);
	\draw[name path=back2] (left3) -- (end2);
	
	\coordinate (middle1) at ($(middle)+(0.5,-1)$);
	\coordinate (left1) at ($(end3)!0.6!(middle1)$);
	
	\draw[white, name path=topfront1] (end3) -- (middle1);
	\draw[name path=topfront2] (left1) -- (end2);
	\draw[name path=topfront3] (middle1) -- (middle);
	
	\coordinate (middle4) at ($(middle)+(-2.3,-1.5)$);
	\coordinate (left4) at ($(end3)!0.6!(middle4)$);
	
	\draw (end3) -- (middle4);
	\draw[name path=topback1] (middle4) -- (end1);
	\draw[name path=topback2] (left4) -- (middle);
	
	\filldraw[white, name intersections={of=topfront1 and topback1}=-] (intersection-1) circle (5pt);
	\filldraw[white, name intersections={of=topfront1 and topback2}=-] (intersection-1) circle (5pt);
	
	\draw (end3) -- (middle1);
	
	\filldraw (left3) circle (2pt);
	\filldraw (middle3) circle (2pt);
	\filldraw[name intersections={of=back1 and back2}=-] (intersection-1) circle (2pt);
	
	\filldraw (end3) circle (2pt) node[above] {$x$};
	\filldraw (middle1) circle (2pt);
	\filldraw (left1) circle (2pt);
	\filldraw[name intersections={of=topfront2 and topfront3}=-] (intersection-1) circle (2pt);
	
	\filldraw (middle4) circle (2pt);
	\filldraw (left4) circle (2pt);
	\filldraw[name intersections={of=topback1 and topback2}=-] (intersection-1) circle (2pt);
	
	\filldraw[white] (0,1.8) circle (2pt);
	\end{tikzpicture}
	\subcaption{A non-degenerate even-fan-paddle in which every fan has length four.} \label{efp_four}
\end{subfigure}
\caption{Examples of even-fan-paddles.} \label{efp_figs}
\end{figure}
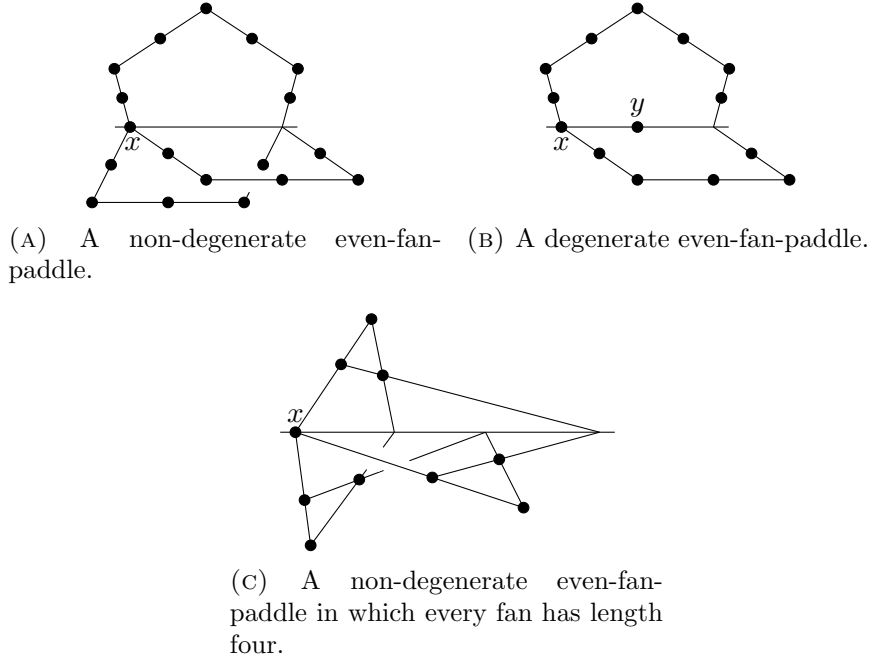

\Cref{efp_four} shows a non-degenerate even-fan-paddle in which $|P_i \cup \{x\}| = 4$ for all $i \in [m]$. Note that, in this instance, the fan ordering of $P_i \cup \{x\}$ in (iii) differs depending on the choice of $j$.  
It is easily checked that an even-fan-paddle has no detachable pairs.

\subsection*{Triad-paddles and related structures}

We say that $M$ is a \emph{triad-paddle} if $M \cong M(K_{3,m})$ for some $m \ge 2$ (see \cref{mk33_fig} for when $m=3$).
Note that $M(K_{3,m})$ has a paddle $(P_1,P_2,\ldots,P_m)$ such that $P_i$ is a triad for all $i \in [m]$, in which case $(P_1,P_2,\ldots,P_m)$ is the \emph{partition} of the triad-paddle.
A triad-paddle has no detachable pairs.

\input{mk3m_figure}

A matroid $M$ is a \emph{quasi-triad-paddle with partition $(P_1,P_2,\ldots,P_m)$} if $M$ has a paddle $(P_1,P_2,\ldots,P_m)$ for some $m \ge 3$ such that $M \backslash P_m$ is a triad-paddle with partition $(P_1,P_2,\ldots,P_{m-1})$.
Next, we define several families of quasi-triad-paddles with no detachable pairs (see \cref{mk3mpaddles_fig}).

First, if $M$ has a partition $(P_1,P_2,\ldots,P_s,Q_1,Q_2,\dotsc,Q_t)$, for $s,t \ge 2$, such that
\begin{enumerate}
    \item $M$ is a quasi-triad-paddle with partition $$(P_1,P_2,\ldots,P_s,Q_1 \cup Q_2 \cup \dotsm \cup Q_t),$$ and
    \item $M^*$ is a quasi-triad-paddle with partition $$(Q_1,Q_2,\ldots,Q_t,P_1 \cup P_2 \cup \dotsm \cup P_t),$$
\end{enumerate}
then we say $M$ is a \emph{tri-paddle-copaddle} with \emph{partition} $$(P_1,P_2,\ldots,P_s, Q_1,Q_2,\dotsc,Q_t)$$ (see \cref{hybrid} for an example with $s=t=2$).
Note that $P_i$ is a triad for each $i \in [s]$, whereas $Q_i$ is a triangle for each $i \in [t]$, and $M$ is both a paddle $(P_1,P_2,\ldots,P_s,Q_1 \cup Q_2 \cup \dotsm \cup Q_t)$ and a copaddle $(P_1 \cup P_2 \cup \dotsm \cup P_s,Q_1,Q_2,\dotsc,Q_t)$.

In what follows, let $X$ and $T^*$ be disjoint subsets of $E(M)$, where $T^*$ is a triad.
	
The set $X$ is a \emph{$4$-element-fan affixed to $T^*$} if
\begin{enumerate}
	\item the set $X$ is a fan of length four with ordering $(x_0,x_1,x_2,x)$ where $\{x_0,x_1,x_2\}$ is a triad,
	\item $x \in \cl(T^*)$, and
	\item for each $i \in \{1,2\}$, there is a $4$-element circuit $C_i$ such that $\{x_0,x_i\} \subseteq C_i \subseteq (\{x_0,x_i\} \cup T^*)$.
\end{enumerate}
The matroid $M$ is a \emph{hinged triad-paddle} with partition $(P_1,P_2,\ldots,P_m,\{x\})$, for some $m \ge 3$, if
\begin{enumerate}
    \item $(P_1,P_2,\ldots,P_m \cup \{x\})$ is a paddle, with $x \notin P_m$,
	\item $P_i$ is a triad for each $i \in [m]$,
	\item $x \in \cl(P_m)$ but $P_m \cup \{x\}$ is not a $4$-element fan, and
    \item for each $i \in [m-1]$, either $P_i \cup \{x\}$ is a $4$-element-fan affixed to $P_m$, or $M | (P_i \cup P_m) \cong M(K_{2,3})$.
\end{enumerate}
\cref{hinged_fig} shows an example of a hinged triad-paddle with $m = 4$, and \cref{hinged_no_fan_fig} shows a hinged triad-paddle with $m = 3$ in which there is no $4$-element-fan affixed to $P_m$. We note that every hinged triad-paddle can be constructed as follows.
Start with $U_{2,4}$ on ground set $\{x,y,z,w\}$.
Repeatedly attach copies of $M(K_4)$ along subsets of $\{x,y,z,w\}$ of size three using generalised parallel connection.
Delete $y$, $z$, and $w$.
If every copy of $M(K_4)$ was attached along a subset of $\{x,y,z,w\}$ containing $x$, then the resulting matroid is an even-fan-paddle.
Otherwise, we see that the matroid is a hinged triad-paddle by taking a partition $(P_1,P_2,\ldots,P_m,\{x\})$ where each $P_i$ consists of the remaining elements from a copy of $M(K_4)$ and $P_m$ has the elements of a copy of $M(K_4)$ that was attached along $\{y,z,w\}$.
It is easily checked that a hinged triad-paddle has no detachable pairs.

Now, suppose $M$ is a quasi-triad-paddle with partition $(P_1,P_2,\ldots,P_m)$.
We describe the other possibilities for the petal $P_m$, when $M$ has no detachable pairs.

We say $X$ is an \emph{augmented fan affixed to $T^*$} if there is some $x \in X$ such that
\begin{enumerate}
    \item $X-\{x\}$ is a fan of length five with ordering $(e_1,e_2,e_3,e_4,e_5)$ where $\{e_1,e_2,e_3\}$ is a triad,
    \item $\{e_1,e_3,e_5,x\}$ is a circuit,
    \item $T^* \cup \{x\}$ is a fan of length four with ends $x$ and $t_1 \in T^*$, and
    \item for some labelling $T^*=\{t_1,t_2,t_3\}$, the sets $\{t_1,t_2,e_1,e_2\}$ and $\{t_1,t_3,e_4,e_5\}$ are circuits.
\end{enumerate}
We say that $M$ is a \emph{quasi-triad-paddle with an augmented-fan petal} if, for each $i \in [m-1]$, the petal $P_m$ is an augmented fan affixed to $P_i$ (see \cref{aug_fig}). Furthermore, $X$ is a \emph{co-augmented fan affixed to $T^*$} if there is some $x \in X$ such that
\begin{enumerate}
    \item $X-\{x\}$ is a fan of length five with ordering $(e_1,e_2,e_3,e_4,e_5)$ where $\{e_1,e_2,e_3\}$ is a triangle,
    \item $\{e_1,e_3,e_5,x\}$ is a cocircuit, and 
    \item for some labelling $T^*=\{t_1,t_2,t_3\}$, the sets $\{t_1,t_2,e_1,x\}$ and $\{t_1,t_3,e_5,x\}$ are circuits.
\end{enumerate}
The matroid $M$ is a \emph{quasi-triad-paddle with a co-augmented-fan petal} if, for each $i \in [m-1]$, the petal $P_m$ is a co-augmented fan affixed to $P_i$ (see \cref{coaug_fig}).

A $4$-element subset $Q$ of $E(M)$ is a \emph{quad} if $Q$ is a circuit and a cocircuit. We say that $X$ is a \emph{quad affixed to $T^*$} if
\begin{enumerate}
    \item $X$ is a quad, and,
    \item for all $x \in X$, there exist distinct $x_1,x_2 \in X-\{x\}$ such that for each $i \in \{1,2\}$, there is a $4$-element circuit~$C_i$ such that $\{x,x_i\} \subseteq C_i \subseteq \{x,x_i\} \cup T^*$.
\end{enumerate}
Furthermore, $X$ is a \emph{near-quad affixed to $T^*$} if
\begin{enumerate}
    \item $X$ is a cocircuit,
    \item there is some $x \in X$ such that $X-\{x\}$ is a triangle, and
    \item there exist distinct $x_1,x_2 \in X-\{x\}$ such that, for each $i \in \{1,2\}$, there is a $4$-element circuit~$C_i$ such that $\{x,x_i\} \subseteq C_i \subseteq \{x,x_i\} \cup T^*$.
\end{enumerate}

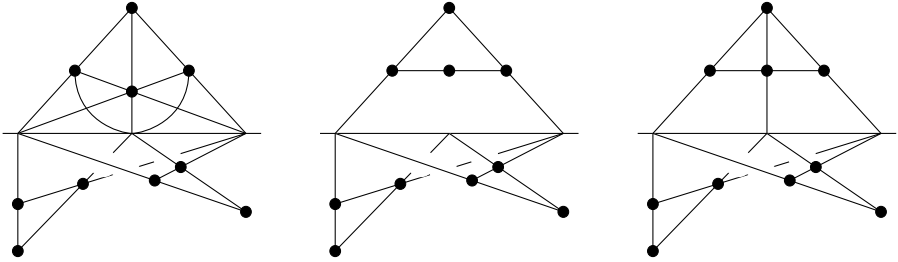
The matroid $M$ is a \emph{quasi-triad-paddle with a quad petal} (or a \emph{quasi-triad-paddle with a near-quad petal}) if, for each $i \in [m-1]$, the petal $P_m$ is a quad (or a near-quad, respectively) affixed to $P_i$.
It is not difficult to verify that in a quasi-triad-paddle, there are three different ways that a quad petal can appear, as shown in \cref{typea_fig1,typea_fig2,typea_fig3}, and two different ways that a near-quad petal can appear, as shown in \cref{typeb_fig1,typeb_fig2}.
In each of these four cases, it is easily checked that $M$ has no detachable pairs.

\subsection*{Accordions}
Let $F$ be a maximal fan of $M$ with ordering $(e_1,e_2,\ldots,e_{|F|})$, having even length at least four, such that $\{e_1,e_2,e_3\}$ is a triangle. 
Let $X \subseteq E(M) - F$ such that $|E(M)| \ge |X \cup F| + 2$. 

\input{accordion_figure}

We say that $X$ is a \emph{left-hand fan-type end of $F$} if $X \cup \{e_1\}$ is a maximal fan of length five with ordering $(e_1,g_2,g_3,g_4,g_5)$ such that $\{e_1,g_2,g_3\}$ is a triangle, and $\{e_1,e_2,g_3,g_5\}$ is a cocircuit.
Furthermore, $X$ is a \emph{right-hand fan-type end of $F$} if $X \cup \{e_{|F|}\}$ is a maximal fan of length five with ordering $(e_{|F|},h_2,h_3,h_4,h_5)$ such that $\{e_{|F|},h_2,h_3\}$ is a triad, and $\{e_{|F|-1},e_{|F|},h_3,h_5\}$ is a circuit.

We say that $X$ is a \emph{left-hand quad-type end of $F$} if $X = \{a_1,a_2,b_1,b_2\}$ is a quad such that
\begin{enumerate}
    \item $\{e_1,a_1,a_2\}$ and $\{e_1,b_1,b_2\}$ are triangles, each not contained in a $4$-element fan, and
    \item $\{e_1,e_2,a_1,b_1\}$ and $\{e_1,e_2,a_2,b_2\}$ are cocircuits.
\end{enumerate}
Also, $X$ is a \emph{right-hand quad-type end of $F$} if $X = \{c_1,c_2,d_1,d_2\}$ is a quad such that
\begin{enumerate}
    \item $\{e_{|F|},c_1,c_2\}$ and $\{e_{|F|},d_1,d_2\}$ are triads, each not contained in a $4$-element fan, and
    \item $\{e_{|F|-1},e_{|F|},c_1,d_1\}$ and $\{e_{|F|-1},e_{|F|},c_2,d_2\}$ are circuits.
\end{enumerate}

Lastly, $X$ is a \emph{left-hand triangle-type end of $F$} if $X \cup \{e_1\}$ is a triangle that is not contained in a $4$-element fan, and $X \cup \{e_1,e_2\}$ is a cocircuit; while $X$ is a \emph{right-hand triad-type end of $F$} if $X \cup \{e_{|F|}\}$ is a triad that is not contained in a $4$-element fan, and $X \cup \{e_{|F|-1},e_{|F|}\}$ is a circuit.

The matroid $M$ is an \emph{accordion} if $E(M)$ has a partition $(G,F,H)$ such that 
\begin{enumerate}
    \item $F$ is a maximal fan with even length at least four,
    \item $G$ is a left-hand fan-type, quad-type, or triangle-type end of $F$, and
    \item $H$ is a right-hand fan-type, quad-type, or triad-type end of $F$.
\end{enumerate}

Geometric representations of the nine types of accordion are illustrated in \cref{accordion_figs}. It is easily checked that accordions have no detachable pairs.

Observe that if $G$ is a left-hand fan-type, quad-type, or triangle-type end of $F$ in $M$, then $G$ is a right-hand fan-type, quad-type, or triad-type end of $F$ in $M^*$ respectively.
    Hence, if $M$ is an accordion with partition $(G,F,H)$, then $M^*$ is an accordion with partition $(H,F,G)$. 

The following lemmas further describe the structure of ends in accordions. We defer the proofs to \cref{det_prelims}, as they require preliminary results regarding connectivity seen in that section.

\begin{lemma}
    \label{accordlhfan}
    Let $M$ be an accordion with partition $(G,F,H)$ where $F$ is a maximal fan having even length at least four, and $G$ is a left-hand fan-type end of $F$.
    Suppose that $F$ has ordering $(e_1,e_2,\ldots,e_{|F|})$, where $\{e_1,e_2,e_3\}$ is a triangle, and $G \cup \{e_1\}$ has ordering $(e_1,g_2,g_3,g_4,g_5)$. 
    Then $\sqcap(\{g_2,g_4\}, H) = 1$, and $\sqcap^*(\{g_4,g_5\}, H) = 1$.
\end{lemma}

\begin{lemma}
    \label{accordlhtri}
    Let $M$ be an accordion with partition $(G,F,H)$ where $F$ is a maximal fan having even length at least four, and $G$ is a left-hand triangle-type end of $F$.
    Suppose that $F$ has ordering $(e_1,e_2,\ldots,e_{|F|})$, where $\{e_1,e_2,e_3\}$ is a triangle.
    Then $\sqcap(G, H) = \sqcap^*(G, H) = 1$.
\end{lemma}

\begin{lemma}
    \label{accordlhquad}
    Let $M$ be an accordion with partition $(G,F,H)$ where $F$ is a maximal fan having even length at least four, and $G$ is a left-hand quad-type end of $F$.
    Suppose that $F$ has ordering $(e_1,e_2,\ldots,e_{|F|})$, and $G = \{a_1,a_2,b_1,b_2\}$, where $\{e_1,e_2,e_3\}$, $\{e_1,a_1,a_2\}$ and $\{e_1,b_1,b_2\}$ are triangles, and $\{e_1,e_2,a_1,b_1\}$ and $\{e_1,e_2,a_2,b_2\}$ are cocircuits.
    Then
    \begin{enumerate}
        \item $\sqcap(\{a_1,b_1\}, H) = \sqcap(\{a_2,b_2\}, H) = 1$, and
        \item $\sqcap^*(\{a_1,a_2\}, H) = \sqcap^*(\{b_1,b_2\}, H) = 1$.
    \end{enumerate}
\end{lemma}

Note that when $G$ is a left-hand triangle-type end in an accordion, the definition does not allow for an element of $G$ to be the ``tip'' of the fan; when this occurs, however, the matroid is an even-fan-spike with tip and cotip.
    For example, suppose $M$ is a matroid whose ground set has a partition $(G,F,H)$ such that $G=\{x,y\}$ is a left-hand triangle-type end, $F=(e_1,e_2,\dotsc,e_{|F|})$ is even fan with $|F| \ge 4$, and $H$ is a right-hand quad-type end, but $\{x,e_1,e_2\}$ is a triad.
    Then $M$ is a even-fan-spike with tip $y$ and cotip $e_{|F|}$ having two distinct $4$-element fans with ends $y$ and $e_{|F|}$, as well as the even fan $F \cup G$.
    Similarly, in the case that $H$ is instead a right-hand triad-type end or fan-type end, then $M$ is a degenerate even-fan-spike with tip and cotip.

\section{Graphs with no detachable pairs} \label{exceptionalgraphs}

In this section, we define the simple $3$-connected graphs with no detachable pairs, appearing in \cref{detachable_graph}.  These are illustrated in \cref{exceptionalgraphsfig}.

A \emph{wheel} is a simple graph that can be obtained from a cycle by adding a single vertex that is adjacent to all vertices of the cycle.
This dominating vertex is called the \emph{hub} of the wheel.
A \emph{mutant wheel} can be constructed as follows.
Consider a wheel with distinct edges $a_1$, $b_1$, $a_2$, $b_2$, $a_3$ such that $\{a_1,b_1,a_2\}$ and $\{a_2,b_2,a_3\}$ are both triangles, and the edges $b_2$ and $b_3$ are not incident with the hub.
Let $u$ be the vertex incident to both $b_1$ and $a_2$, and let $v$ be the vertex incident to both $b_2$ and $a_3$. Subdivide the edge $a_1$, thus creating a new vertex $x$, and add an edge between $x$ and $u$, and also subdivide the edge $a_2$, creating a new vertex $y$, and add an edge between $y$ and $v$.

Next we define a twisted wheel.
Consider a copy of $K_4$ having non-adjacent edges $e = \{e_1,e_2\}$ and $f = \{f_1,f_2\}$.
A \emph{twisted wheel} is a graph that can be obtained by subdividing $e$ so that $j \ge 0$ new vertices are introduced, adding $j$ edges between each of these vertices and $f_1$; then subdividing $f$ so that $k \ge 0$ new vertices are introduced, and adding $k$ edges between each of the $k$ new vertices and $e_1$, where $j+k \ge 1$.

A warped wheel can be obtained from a twisted wheel by deleting the edge between $e_1$ and $f_1$, and contracting the edge between $e_2$ and $f_2$.
Alternatively, let $W_4$ be a wheel on five vertices with hub $h$, whose remaining vertices $v_1,v_2,v_3,v_4$ are such that $v_i$ is adjacent to $v_{i+1}$ for each $i$ when indices are interpreted modulo $4$.
A \emph{warped wheel} is a graph that can be obtained from $W_4$ by subdividing $hv_1$ so that $j \ge 1$ new vertices are introduced, adding $j$ edges between each of these vertices and $v_2$, and then subdividing $hv_3$ so that $k \ge 1$ new vertices are introduced, adding $k$ edges between each of these vertices and $v_4$.

A \emph{multi-wheel} is a graph that can be constructed as follows. Begin with a $3$-vertex path on vertices $u,h,v$, and add $k \ge 3$ parallel edges between $u$ and $v$.
If $k \ge 4$, then for each of the $k$ parallel edges, subdivide it at least once and join each of the resulting new vertices to $h$. If $k=3$, then for at least two of the three parallel edges, subdivide it at least once and join each of the resulting new vertices to $h$.
Finally, remove the edge between $u$ and $h$.
The multi-wheel is \emph{degenerate} if the vertices $u$ and $v$ are adjacent (in the above construction, this corresponds to the case where three parallel edges are added between $u$ and $v$, and one of these edges is not subdivided).
We note that a multi-wheel is referred to in \cite{detpairs3} as an ``unhinged multi-dimensional wheel''.

A \emph{stretched wheel} is the geometric dual of a degenerate multi-wheel. Alternatively, it can be constructed as follows.
Consider a wheel with hub $x$, let $y$ be any other vertex, and let $e$ be an edge incident to $y$ but not to $x$. Add a new vertex $z$ that is adjacent to $x$ and $y$. Subdivide $e$ so that $k \geq 1$ new edges are introduced, and add an edge between each new vertex and $z$.

Finally, we define $K^a_{3,m}$ and $K^b_{3,m}$.
Consider a copy of the complete bipartite graph $K_{3,m}$ with parts $\{u_1,u_2,u_3\}$ and $\{v_1,v_2,\ldots,v_m\}$.
The graph $K^a_{3,m}$ can be constructed from $K_{3,m}$ by adding a vertex $a$ that is adjacent to $u_1,u_2,u_3$, then adding a vertex $b$ that is adjacent to $a,u_1,u_3$.
The graph $K^b_{3,m}$ can be constructed from $K_{3,m}$ by adding a vertex $a$ that is adjacent to $u_1,u_2$, then adding a vertex $b$ that is adjacent to $a,u_2,u_3$, then finally adding an edge between $u_1$ and $u_3$.

The correspondence between these graphs and the matroids listed in \cref{detachable_main} is as follows (for full details, refer to the proof of \cref{detachable_graph} in \cref{graph_proof}).
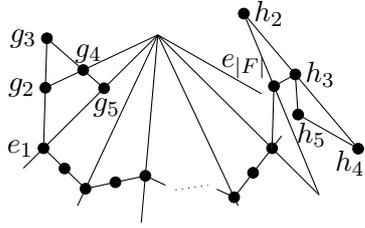
A mutant wheel corresponds to a graphic accordion; such an accordion has left- and right-hand fan-type ends.
A twisted wheel corresponds to a graphic even-fan-spike with tip and cotip.
A warped wheel corresponds to a graphic even-fan-spike (that is tipless and cotipless).
Note that, in both of these cases, such an even-fan-spike is degenerate.
A multi-wheel corresponds to a graphic even-fan-paddle.
The graphs $K^a_{3,m}$ and $K^b_{3,m}$ are quasi-triad-paddles with a co-augmented-fan petal and with an augmented-fan petal respectively.

\section{Preliminaries} \label{det_prelims}

Our notation and terminology follows Oxley~\cite{oxley!!!}, except where we specify otherwise.  We say that a set $X$ \emph{meets} a set $Y$ if $X \cap Y \neq \emptyset$.

\subsection*{Connectivity}

Recall that the \emph{connectivity} of $X$ in $M$ is
\[
    \lambda_M(X) = r(X) + r(E-X) - r(M).
\]
Equivalently,
\[
\lambda_M(X) = r(X) + r^*(X) - |X|.
\]
When it is clear that we are referring to the matroid $M$, we will often write $\lambda(X)$ instead of $\lambda_M(X)$. 
It follows from the definition that $\lambda_M(X) = \lambda_M(E-X)$ and $\lambda_{M^*}(X) = \lambda_M(X)$.

The next two lemmas are straightforward to prove (see, for example, \cite[Corollary 8.2.6, Proposition 8.2.14]{oxley!!!}). They will be applied freely throughout the proof of \Cref{detachable_main}. 

\begin{lemma} \label{delconnectivity}
	Let $M$ be a matroid, and let $X \subseteq E(M)$ and $e \in E(M)-X$. Then
	\[
	\lambda_{M / e}(X) = \begin{cases}
	\lambda_M(X)-1 & \mbox{if $e \in \cl(X)$ and $e$ is not a loop}, \\
	\lambda_M(X) & \mbox{otherwise}.
	\end{cases}
	\]
	Dually,
	\[
	\lambda_{M \backslash e}(X) = \begin{cases}
	\lambda_M(X)-1 & \mbox{if $e \in \cl^*(X)$ and $e$ is not a coloop}, \\
	\lambda_M(X) & \mbox{otherwise}.
	\end{cases}
	\]
\end{lemma}

\begin{lemma} \label{clconnectivity}
    Let $M$ be a matroid, let $X \subseteq E(M)$, and let $e \in E(M)-X$. Then
	\[
	\lambda(X \cup \{e\}) = \begin{cases}
	\lambda(X) - 1 & \mbox{if $e \in \cl(X)$ and $e \in \cl^*(X)$}, \\
	\lambda(X) & \mbox{if $e \in \cl(X)$ and $e \notin \cl^*(X)$}, \\
	\lambda(X) & \mbox{if $e \notin \cl(X)$ and $e \in \cl^*(X)$}, \\
	\lambda(X) + 1 & \mbox{if $e \notin \cl(X)$ and $e \notin \cl^*(X)$}.
	\end{cases}
	\]
\end{lemma}

For a matroid $M$ and $X \subseteq E(M)$, we say that $X$ is \emph{$k$-separating} if $\lambda(X) < k$, and $X$ is a \emph{$k$-separation} if $\lambda(X) = k-1$ and $|X| \geq k$ and $|E(M)-X| \geq k$. A matroid is $k$-connected if it contains no $k'$-separations, for all $k' < k$. 

Recall that a \emph{triangle} is a circuit of size three, a \emph{triad} is a cocircuit of size three, and a \emph{quad} is a $4$-element set that is both a circuit and a cocircuit.
If $M$ is $3$-connected, then a $3$-separation of $M$ of size three is either a triangle or a triad, while a $3$-separation of $M$ of size four either contains a triangle or a triad, or it is a quad.

The next two well-known lemmas are useful for identifying elements that may be deleted or contracted while retaining $3$-connectivity. The first is commonly referred to as Bixby's Lemma \cite[Theorem 1]{bixby}.

\begin{lemma}
	Let $M$ be a $3$-connected matroid and let $e \in E(M)$. Then either $\si(M / e)$ is $3$-connected or $\co(M \backslash e)$ is $3$-connected.
\end{lemma}

The next lemma is called Tutte's Triangle Lemma \cite[7.2]{wheelsandwhirls}.

\begin{lemma}
	Let $M$ be a $3$-connected matroid such that $|E(M)| \geq 4$. Let $T = \{e,e',e''\}$ be a triangle of $M$ such that neither $M \backslash e$ nor $M \backslash e'$ are $3$-connected. Then there exists a triad of $M$ containing either $\{e,e'\}$ or $\{e,e''\}$.
\end{lemma}

Applying Tutte's Triangle Lemma to $M^*$ rather than $M$ gives the following corollary, which we also refer to as Tutte's Triangle Lemma.

\begin{corollary}
	Let $M$ be a $3$-connected matroid such that $|E(M)| \geq 4$. Let $T^* = \{e,e',e''\}$ be a triad of $M$ such that neither $M / e$ nor $M / e'$ are $3$-connected. Then there exists a triangle of $M$ containing either $\{e,e'\}$ or $\{e,e''\}$.
\end{corollary}

One consequence of Tutte's Triangle Lemma is the following.
If $T$ is a triangle of $M$ that does not meet a triad, then there are at least two elements of $T$ that can be deleted while retaining $3$-connectivity. Dually, a triad that does not meet a triangle contains at least two elements that can be contracted while retaining $3$-connectivity.

For a proof of the following lemma, see, for example, \cite[Lemma~8.8.2]{oxley!!!}.
\begin{lemma} \label{segment_deletable}
Let $M$ be a $3$-connected matroid, and let $X \subseteq E(M)$ such that $r(X) = 2$ and $|X| \geq 4$. Then $M \backslash e$ is $3$-connected for all $e \in X$.
\end{lemma}

The next lemma is a special case of \cite[Lemma~3.8]{whittlestabs}.
\begin{lemma} \label{quad_detachable}
	Let $M$ be a $3$-connected matroid, let $X \subseteq E(M)$ be a quad, and let $e \in X$.
	If $e$ is not contained in a triad, then $M \backslash e$ is $3$-connected.
\end{lemma}

\subsection*{Fans}
Recall that a \emph{fan} of a matroid $M$ is a subset $F$ of $E(M)$ such that either $|F|=2$, or $|F| \ge 3$ and there is an ordering $(e_1,e_2,\ldots,e_{|F|})$ of $F$ such that $\{e_1,e_2,e_3\}$ is a triangle or a triad, and, for all $i \in [|F|-3]$, if $\{e_i,e_{i+1},e_{i+2}\}$ is a triangle then $\{e_{i+1},e_{i+2},e_{i+3}\}$ is a triad, and if $\{e_i,e_{i+1},e_{i+2}\}$ is a triad, then $\{e_{i+1},e_{i+2},e_{i+3}\}$ is a triangle.
A fan $F$ is \emph{maximal} if there is no fan $F'$ such that $F$ is a proper subset of $F'$.

Let $F$ be a fan of length $k \geq 3$ with ordering $(e_1,e_2,\ldots,e_k)$.
Note that if $k$ is even, then one of $\{e_1,e_2,e_3\}$ and $\{e_{k-2},e_{k-1},e_k\}$ is a triangle and the other is a triad. Similarly, if $k$ is odd, then $\{e_1,e_2,e_3\}$ and $\{e_{k-2},e_{k-1},e_k\}$ are either both triangles or both triads.

If $F$ is a fan with ordering $(e_1,e_2,\ldots,e_{|F|})$, then $(e_{|F|},e_{|F|-1},\ldots,e_1)$ is also a fan ordering of $F$. When exploiting this symmetry, we use the phrase ``up to reversing the ordering of $F$".
If $F$ has length at least five, then, up to reversing the ordering, $F$ has a unique ordering \cite{oxleywufans}.
However, if $F$ has length four and $(e_1,e_2,e_3,e_4)$ is an ordering of $F$, then $(e_1,e_3,e_2,e_4)$ is also an ordering of $F$.
Moreover, if $F$ has length three, then the ordering of $F$ is arbitrary.
Although a fan $F$ can have different orderings, it is often convenient to refer to $F$ by an ordering of $F$;
for example, we say ``$(e_1,e_2,\ldots,e_{|F|})$ is a fan'' as a shorthand for ``the set $\{e_1,e_2,\ldots,e_{|F|}\}$ is a fan with ordering $(e_1,e_2,\ldots,e_{|F|})$''.

The next four lemmas provide some properties of fans in $3$-connected matroids.
We omit the straightforward proofs.

\begin{lemma} \label{fan_rank}
	Let $M$ be a $3$-connected matroid, and let $F = (e_1,e_2,\ldots,e_{|F|})$ be a fan of $M$ such that $|E(M)| \geq |F| + 2$. Then
	\begin{align*}
	r(F) &= \begin{cases}
	\floor*{\frac{|F|}{2}}+1, & \mbox{if $\{e_1,e_2,e_3\}$ is a triangle}; \\[6pt]
	\ceil*{\frac{|F|}{2}}+1, & \mbox{if $\{e_1,e_2,e_3\}$ is a triad},
	\end{cases}\\
	\end{align*}
	and
	\begin{align*}
	r^*(F) &= \begin{cases}
	\ceil*{\frac{|F|}{2}}+1, & \mbox{if $\{e_1,e_2,e_3\}$ is a triangle}; \\[6pt]
	\floor*{\frac{|F|}{2}}+1, & \mbox{if $\{e_1,e_2,e_3\}$ is a triad}.
	\end{cases}\\
	\end{align*}
	In particular,
	\begin{align*}
	\lambda(F) &= 2.
	\end{align*}
\end{lemma}

\begin{lemma} \label{not_just_fan}
	Let $M$ be a $3$-connected matroid, and let $F$ be a fan of $M$ such that $|F| \geq 4$. Then either $M$ is a wheel or a whirl, or $|E(M)| \geq |F| + 2$.
\end{lemma}

\begin{lemma} \label{fan_ends}
	Let $M$ be a $3$-connected matroid that is not a wheel or a whirl, and let $F = (e_1,e_2,\ldots,e_{|F|})$ be a maximal fan of $M$ such that $|F| \geq 3$.
	\begin{enumerate}
		\item If $\{e_1,e_2,e_3\}$ is a triad, then $e_1$ is not contained in a triangle.
		\item If $\{e_1,e_2,e_3\}$ is a triangle, then $e_1$ is not contained in a triad.
	\end{enumerate}
\end{lemma}

A consequence of \Cref{fan_ends} is that if $F = (e_1,e_2,\ldots,e_{|F|})$ is a maximal fan of a $3$-connected matroid $M$ that is not a wheel or a whirl, and $|F| \geq 4$, then, by Tutte's Triangle Lemma, either (i) holds and $M / e_1$ is $3$-connected, or (ii) holds and $M \backslash e_1$ is $3$-connected. Of course, analogous outcomes also hold for $e_{|F|}$.

\begin{lemma} \label{unique_triangle}
    Let $M$ be a $3$-connected matroid that is not a wheel or a whirl, and let $F = (e_1,e_2,\ldots,e_{k})$ be a maximal fan of $M$ with length $k \ge 4$. Then, for all $i \in [k-1]$, both of the following hold:
	\begin{enumerate}
		\item if $\{e_i,e_{i+1}\}$ is contained in a triangle $T$, then either $T = \{e_{i-1},e_i,e_{i+1}\}$ or $T = \{e_i,e_{i+1},e_{i+2}\}$, and
		\item if $\{e_i,e_{i+1}\}$ is contained in a triad $T^*$, then either $T^* = \{e_{i-1},e_i,e_{i+1}\}$ or $T^* = \{e_i,e_{i+1},e_{i+2}\}$.
	\end{enumerate}
\end{lemma}

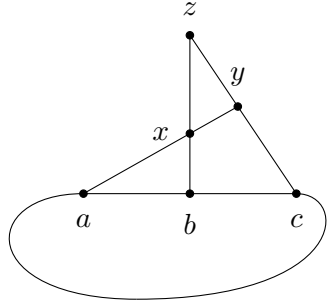
\begin{figure}
	\centering
	\begin{tikzpicture}[scale=.7]
	\coordinate (a) at (0,0);
	\coordinate (b) at (2,0);
	\coordinate (c) at (4,0);
	\coordinate (z) at (2,3);
	\coordinate (y) at ($(z)!0.45!(c)$);
	
	\draw (a) -- (c);
	\draw (z) -- (c);
	\draw[name path=line1] (y) -- (a);
	\draw[name path=line2] (z) -- (b);
	
	\filldraw[name intersections={of=line1 and line2}=-] (intersection-1) circle (2pt) node[label=left:$x$] {};
	\filldraw (a) circle (2pt) node[label=below:$a$] {};
	\filldraw (b) circle (2pt) node[label=below:$b$] {};
	\filldraw (c) circle (2pt) node[label=below:$c$] {};
	\filldraw (y) circle (2pt) node[label=above:$y$] {};
	\filldraw (z) circle (2pt) node[label=above:$z$] {};
	
	\draw (a) .. controls ($(a)+(-2,0)$) and ($(a)+(-2,-2)$) .. ($($(a)!0.5!(b)$)+(0,-2)$);
	\draw (c) .. controls ($(c)+(1,0)$) and ($(c)+(1,-2)$) .. ($($(a)!0.5!(b)$)+(0,-2)$);
	\end{tikzpicture}
	\caption{An $M(K_4)$-separator.} \label{mk4}
\end{figure}

An \emph{$M(K_4)$-separator} of a matroid $M$, pictured in \Cref{mk4}, is a set $\{a,b,c,x,y,z\} \subseteq E(M)$ such that $\{x,y,z\}$ is a triad, and $\{a,b,c\}$, $\{a,x,y\}$, $\{b,x,z\}$, and $\{c,y,z\}$ are all triangles.
It is well known that two distinct maximal fans of length at least four intersect in only their ends unless they form an $M(K_4)$-separator; we provide a proof of this result for completeness, as \cref{intersecting_fans}.
We first require the following lemma:

\begin{lemma} \label{intersecting_fans_helper}
	Let $M$ be a $3$-connected matroid and let $F_1 = (e_1,e_2,\ldots,e_{|F_1|})$ and $F_2 = (f_1,f_2,\ldots,f_{|F_2|})$ be distinct maximal fans of $M$ such that $|F_1| \geq 4$ and $|F_2| \geq 3$. Let $e \in F_1 \cap F_2$. Then $e \in \{e_2,e_3,\ldots,e_{|F_1|-1}\}$ if and only if $|F_2| \geq 4$ and $e \in \{f_2,f_3,\ldots,f_{|F_2|-1}\}$.
\end{lemma}

\begin{proof}
    First, assume $e \in \{e_2,e_3,\ldots,e_{|F_1|-1}\}$. Thus, $e$ is contained in both a triangle and a triad of $F_1$. Noting that $M$ has distinct maximal fans, and is therefore not a wheel or a whirl, \Cref{fan_ends} implies that $e \notin \{f_1,f_{|F_2|}\}$. Furthermore, if $|F_2| = 3$, then, as $F_2$ meets both a triangle and a triad, $F_2$ is contained in a $4$-element fan, contradicting the maximality of $F_2$. Thus, $|F_2| \geq 4$ and $e \in \{f_2,f_3,\ldots,f_{|F_2|-1}\}$, as desired. Conversely, if $|F_2| \geq 4$ and $e \in \{f_2,f_3,\ldots,f_{|F_2|-1}\}$, then $e$ is contained in a triangle and a triad, so \Cref{fan_ends} implies that $e \notin \{e_1,e_{|F_1|}\}$, which completes the proof.
\end{proof}

\begin{lemma} \label{intersecting_fans}
    Let $M$ be a $3$-connected matroid such that $|E(M)| \geq 8$. Let $F_1$ and $F_2$ be distinct maximal fans of $M$ such that $|F_1| \geq 4$ and $|F_2| \geq 3$, and $F_1 = (e_1,e_2,\ldots,e_{|F_1|})$. Then either $F_1 \cap F_2 \subseteq \{e_1,e_{|F_1|}\}$, or $F_1 \cup F_2$ is an $M(K_4)$-separator in either $M$ or $M^*$.
\end{lemma}

\begin{proof}
    Suppose $F_1 \cap F_2 \not\subseteq \{e_1,e_{|F_1|}\}$, that is, there exists $i \in \{2,3,\ldots,|F_1|-1\}$ such that $e_i \in F_2$. Since $F_1$ and $F_2$ are distinct, $F_1$ has an element that is not contained in $F_2$ so, up to reversing the ordering of $F_1$, we may assume that $e_{i-1} \notin F_2$. The set $\{e_{i-1},e_i,e_{i+1}\}$ is either a triangle or a triad. Up to duality, we may assume that $\{e_{i-1},e_i,e_{i+1}\}$ is a triangle, in which case $\{e_i,e_{i+1},e_{i+2}\}$ is independent. Let $F_2 = (f_1,f_2,\ldots,f_{|F_2|})$. By \Cref{intersecting_fans_helper}, we have that $|F_2| \geq 4$, and $e_i = f_j$ for some $j \in \{2,3,\ldots,|F_2|-1\}$. Now, $e_i$ is contained in a triad of $F_2$, and this triad is one of $\{f_{j-2},f_{j-1},e_i\}$, $\{f_{j-1},e_i,f_{j+1}\}$, or $\{e_i,f_{j+1},f_{j+2}\}$.
	
	First, suppose $\{f_{j-1},e_i,f_{j+1}\}$ is a triad. Then, by orthogonality with the triangle $\left\{e_{i-1},e_i,e_{i+1}\right\}$, and since $e_{i-1} \notin F_2$, we have that $e_{i+1} \in \{f_{j-1},f_{j+1}\}$. Now, $e_{i+1}$ is contained in both a triangle and  a triad, which implies, by \Cref{fan_ends}, that $e_{i+1} \notin \{f_1,f_{|F_2|}\}$. Therefore, if $e_{i+1} = f_{j-1}$, then $\{f_{j-2},f_{j-1},f_j\}$ is a triangle containing both $e_i$ and $e_{i+1}$, and if $e_{i+1} = f_{j+1}$, then $\{f_j,f_{j+1},f_{j+2}\}$ is a triangle containing both $e_i$ and $e_{i+1}$. But since $e_{i-1} \notin F_2$, this triangle is distinct from the triangle $\{e_{i-1},e_i,e_{i+1}\}$, contradicting \Cref{unique_triangle}.
	
	Therefore, either $\{f_{j-2},f_{j-1},e_i\}$ or $\{e_i,f_{j+1},f_{j+2}\}$ is a triad. Up to reversing the ordering of $F_2$, we may assume that $\{e_i,f_{j+1},f_{j+2}\}$ is a triad. By orthogonality with $\{e_{i-1},e_i,e_{i+1}\}$, this triad contains $e_{i+1}$. Suppose $e_{i+1} = f_{j+1}$. Since $e_i \notin \{f_1,f_{|F_2|}\}$, we have that $\{f_{j-1},f_j,f_{j+1}\}$ is a triangle containing both $e_i$ and $e_{i+1}$. This contradiction to \Cref{unique_triangle} implies that $e_{i+1} = f_{j+2}$. Now, $e_{i+1}$ is contained in both a triangle and a triad, so $e_{i+1} \notin \{f_1,f_{|F_2|}\}$. Therefore, $M$ has triangles $\{f_{j-1},e_i,f_{j+1}\}$ and $\{f_{j+1},e_{i+1},f_{j+3}\}$. Similarly, $e_{i+1} \notin \{e_1,e_{|F_1|}\}$, so $M$ has a triad  $\{e_i,e_{i+1},e_{i+2}\}$. By orthogonality, $e_{i+2} = f_{j+1}$. Furthermore, $e_{i+2}$ is contained in both a triangle and a triad, so $e_{i+2} \notin \{e_1,e_{|F_1|}\}$, which means $\{e_{i+1},e_{i+2},e_{i+3}\}$ is a triangle. Now, $\{f_{j+1},f_{j+2},f_{j+3}\}$ is also a triangle containing $\{e_{i+1},e_{i+2}\}$. \Cref{unique_triangle} implies that these are the same triangle, so $e_{i+3} = f_{j+3}$.
	
    We label these elements in the following way: $a = e_{i-1}$, $b = f_{j-1}$, $c = e_{i+3} = f_{j+3}$, $x = e_i = f_j$, $y = e_{i+1} = f_{j+2}$, $z = e_{i+2} = f_{j+1}$. Now, $\{x,y,z\}$ is a triad, and $\{a,x,y\}$, $\{b,x,z\}$, $\{c,y,z\}$ are all triangles. We complete the proof of the lemma by showing that $F_1 \cup F_2 = \{a,b,c,x,y,z\}$ is an $M(K_4)$-separator in $M$. It remains to prove that none of $a$, $b$, or $c$ are contained in triads and that $\{a,b,c\}$ is a triangle.
	
	First, assume that one of $a$, $b$, or $c$ is contained in a triad $T^*$. Orthogonality with the triangles $\{a,x,y\}$, $\{b,x,z\}$, and $\{c,y,z\}$ implies that $T^* \subseteq \{a,b,c,x,y,z\}$. But then $\lambda(\{a,b,c,x,y,z\}) \leq 1$, a contradiction since $|E(M)| \geq 8$. Hence, no such triad exists, so $F_1 \cup F_2 = \{a,b,c,x,y,z\}$. Now, we show that $\{a,b,c\}$ is a triangle. Since $\{x,y,z\}$ is a triad, by submodularity we have
    \begin{align*}
        r(\{a,b,c\}) + r(M) &\leq r(\{a,b,c,x,y,z\}) + r(E(M) - \{x,y,z\}) \\
        &= r(\{a,b,c,x,y,z\}) + r(M) - 1 \,,
    \end{align*}
	so $r(\{a,b,c\}) \leq 2$. Therefore, $r(\{a,b,c\}) = 2$, and $\{a,b,c,x,y,z\}$ is an $M(K_4)$-separator of $M$.
\end{proof}

We note also that if $X$ is an $M(K_4)$-separator in a matroid $M$, then any maximal fan contained in $X$ has length five. 
This can often be used to rule out the latter possibility in an application of \cref{intersecting_fans}.

\subsection*{Vertical and cyclic separations}
Let $M$ be a matroid. A \emph{vertical $k$-separation} of $M$ is a partition $(X,\{e\},Y)$ of $E(M)$ such that $\lambda(X) = k-1$ and  $\lambda(Y) = k-1$, and $e \in \cl(X) \cap \cl(Y)$, and $r(X) \geq k$ and $r(Y) \geq k$. A partition $(X,\{e\},Y)$ is a \emph{cyclic $k$-separation} if $(X,\{e\},Y)$ is a vertical $k$-separation of $M^*$.
The importance of vertical $3$-separations is illustrated by the following lemma (see \cite[Lemma 3.5]{whittlestabs}).

\begin{lemma} \label{basic_vertsep}
    Let $M$ be a $3$-connected matroid and let $e \in E(M)$. Then $M$ has a vertical $3$-separation of the form $(X,\{e\},Y)$ if and only if $\si(M/e)$ is not $3$-connected.
\end{lemma}

The following two lemmas about vertical separations will be useful.
We omit the proofs: the first is similar to \cite[Lemmas~4.4 and 4.5]{BS14}, and the second is a straightforward consequence of the first.

\begin{lemma} \label{fcl_vert_sep}
    Let $M$ be a $3$-connected matroid, let $(X,\{e\},Y)$ be a vertical $3$-separation of $M$, and let $y \in Y$.
    \begin{enumerate}
        \item If $y \in \cl(X)$, then $(X \cup \{y\}, \{e\}, Y-\{y\})$ is a vertical $3$-separation of $M$.
        \item If $y \in \cl^*(X)$ and $e$ is not contained in a triangle of $M$, then $(X \cup \{y\},\{e\},Y-\{y\})$ is a vertical $3$-separation of $M$.
    \end{enumerate}
\end{lemma}

\begin{lemma} \label{fan_vert_sep}
	Let $M$ be a $3$-connected matroid, and let $F$ be a fan of $M$ such that $|F| \geq 3$. Let $(X,\{e\},Y)$ be a vertical $3$-separation of $M$ such that $e \notin F$ and $e$ is not contained in a triangle. Then $M$ has a vertical $3$-separation $(X',\{e\},Y')$ such that $F \subseteq X'$.
\end{lemma}
Note that we will often apply \cref{fan_vert_sep} in the case where $|F| = 3$, that is, when $F$ is a triangle or a triad of $M$.

Naturally, applying \cref{fcl_vert_sep,fan_vert_sep} to $M^*$ give dual results concerning cyclic $3$-separations.

\subsection*{Accordions}

We now return to the proofs of \cref{accordlhfan,accordlhtri,accordlhquad}. 

\begin{proof}[Proof of \cref{accordlhfan}]
    By definition, $\{e_1,g_2,g_3\}$ is a triangle, and $\{e_1,e_2,g_3,g_5\}$ is a cocircuit.
    Note that $g_5 \in \cl^*(F \cup \{g_3\})$, so $\lambda(F \cup \{g_3,g_5\}) \leq 3$. Hence, 
	\begin{align}
        \lambda(H \cup \{g_2,g_4\}) \leq 3.
	\label{7.4_eqn1}
	\end{align}
    By orthogonality with the triangle $\{e_1,g_2,g_3\}$ and the triad $\{g_2,g_3,g_4\}$, we have that $g_2 \notin \cl^*(H)$ and $g_2 \notin \cl(H)$. Thus $\lambda(H \cup \{g_2\}) = \lambda(H)+1 = 3$. Now, $g_4 \notin \cl^*(H \cup \{g_2\})$ by orthogonality with $\{g_3,g_4,g_5\}$, and so, by (\ref{7.4_eqn1}), $g_4 \in \cl(H \cup \{g_2\})$. It follows that $r(H \cup \{g_2,g_4\}) = r(H) + 1$, so $\sqcap(\{g_2,g_4\},H) = 1$, as desired.
	
    To complete the proof, we show that $\sqcap^*(\{g_4,g_5\},H) = 1$. Orthogonality with the cocircuits $\{g_2,g_3,g_4\}$ and $\{e_2,e_1,g_3,g_5\}$ implies that 
    \[r(H \cup \{g_4,g_5\}) = r(H) + 2,\]
	and orthogonality with the circuit $\{g_3,g_4,g_5\}$ implies that
    \[r^*(H \cup \{g_4,g_5\}) \geq r^*(H) + 1.\]
	But $\lambda(H \cup \{g_4,g_5\}) = \lambda(F \cup \{g_2,g_3\}) \leq 3$, so
    \[r^*(H \cup \{g_4,g_5\}) = r^*(H) + 1,\]
    which means that $\sqcap^*(\{g_4,g_5\},H) = 1$.
\end{proof}

\begin{proof}[Proof of \cref{accordlhtri}]
    Suppose that $F$ has ordering $(e_1,e_2,\ldots,e_{|F|})$, where $\{e_1,e_2,e_3\}$ is a triangle, and $G=\{g_2,g_3\}$.  By definition, $\{e_1,g_2,g_3\}$ is a triangle, so $\lambda((F-\{e_1\}) \cup H)=2$, and $\{e_1,e_2,g_2,g_3\}$ is a cocircuit.
    By orthogonality with this cocircuit, we have that $g_2 \notin \cl(H)$, and by orthogonality with the triangle $\{e_1,g_2,g_3\}$, we have $g_2 \notin \cl^*(H)$. Since $\lambda(G \cup H) = 2$, it follows that $g_3 \in \cl(H \cup \{g_2\})$ and $g_3 \in \cl^*(H \cup \{g_2\})$. Thus $r(G \cup H) = r(H) + 1$ and $r^*(G \cup H) = r^*(H) + 1$. It follows that $\sqcap(G,H) = 1$, and $\sqcap^*(G,H) = 1$.
\end{proof}

\begin{proof}[Proof of \cref{accordlhquad}]
    As $\{e_1,a_1,a_2\}$ and $\{e_1,b_1,b_2\}$ are triangles in $M$, we have $r^*(H \cup \{a_1,b_1\}) = r^*(H) + 2$. Furthermore, as $\{e_1, e_2, a_1, b_1\}$ is a cocircuit, we have $r(H \cup \{a_1,b_1\}) \geq r(H) + 1$. But \[\lambda (H \cup \{a_1,b_1\}) = \lambda(F_1 \cup \{a_2,b_2\}) \leq 3.\]
	Thus, $r(H \cup \{a_1, b_1\}) = r(H) + 1$, so $\sqcap(\{a_1,b_1\},H) = 1$. Similarly, $\sqcap(\{a_2,b_2\},H) = 1$.
	
    Likewise, using the cocircuits $\{e_1, e_2, a_1, b_1\}$ and $\{e_1,e_2,a_2,b_2\}$, and the triangle $\{e_1,a_1,a_2\}$, we get $r(H \cup \{a_1,a_2\}) = r(H) + 2$ and $r^*(H \cup \{a_1,a_2\}) \geq r^*(H) + 1$. Since $\lambda(H \cup \{a_1,a_2\}) =\lambda(F \cup \{b_1,b_2\}) \leq 3$, we have that $r^*(H \cup \{a_1,a_2\}) = r^*(H) +  1$, so $\sqcap^*(\{a_1,a_2\},H) = 1$. Symmetrically, $\sqcap^*(\{b_1,b_2\},H) = 1$, thereby completing the proof of the lemma.
\end{proof}

\section{Connectivity lemmas} \label{lemmas}
In this section, we present some lemmas that will be useful throughout the proof of \Cref{detachable_main}.

\begin{lemma} \label{closure_deletable}
	Let $M$ be a $3$-connected matroid.
    Let $X \subseteq E(M)$ such that $\lambda(X) = 2$ and $|X| \geq 3$ and $|E(M)| \geq |X| + 4$.
    If $e \in \cl(X)-X$, then either $e$ is contained in a triad, or $M \backslash e$ is $3$-connected.
\end{lemma}

\begin{proof}
    Let $Y = E(M) - (X \cup \{e\})$. First, assume that both $r(X) > 2$ and $r(Y) > 2$. Then $\lambda_{M / e}(X) = 1$ and $|E(M / e)| \geq |X| + 3$, so $M / e$ is not $3$-connected. Furthermore, $\lambda_{\si(M/e)}(X) = 1$, and, since $r(X) > 2$ and $r(Y) > 2$, there are at least two elements of $X$ and two elements of $Y$ remaining in $\si(M /e)$. Therefore, $\si(M / e)$ is not $3$-connected, and so, by Bixby's Lemma, $\co(M \backslash e)$ is $3$-connected. It follows that either $M \backslash e$ is $3$-connected, or $e$ is contained in a triad.
	
	Now suppose either $r(X) = 2$ or $r(Y) = 2$. Without loss of generality, assume the former. Then $|X \cup \{e\}| \geq 4$ and $r(X \cup \{e\}) = 2$, so, by \Cref{segment_deletable}, the matroid $M \backslash e$ is $3$-connected. This completes the proof.
\end{proof}

\begin{lemma} \label{deletable_cocircuit} \label{contractable_circuit}
    Let $M$ be a $3$-connected matroid with no detachable pairs. Let $X \subseteq E(M)$ such that $|X| \geq 2$ and $|E(M)| \geq |X| + 4$. Let $e \in E(M) - X$ such that $M \backslash e$ is $3$-connected, and either $\lambda(X) = 2$ or $\lambda(X \cup \{e\}) = 2$. Furthermore, let $f \in \cl(X)-(X \cup e)$ such that $f$ is not contained in a triad of $M$. Then $M$ has a $4$-element cocircuit $\{e,f,g,h\}$ such that $g \in X$ and $h \notin X$.
\end{lemma}

\begin{proof}
	We first prove that there is a triad of $M \backslash e$ containing $f$. Suppose this is not the case. Since $M \backslash e$ is $3$-connected and $|X| \geq 2$ and $|E(M)| \geq |X| + 4$, we have that $\lambda_{M \backslash e}(X) \geq 2$. Therefore, if $\lambda_{M}(X) = 2$, \Cref{delconnectivity} implies that $\lambda_{M \backslash e}(X) = 2$. If $\lambda_{M}(X) \neq 2$, then $\lambda_{M}(X \cup \{e\}) = 2$. This implies that $\lambda_M(X) = 3$ and $e \in \cl^*(X)$. Again, \Cref{delconnectivity} implies that $\lambda_{M \backslash e}(X) = 2$. 
	If $|X| \geq 3$ and $|E(M)-(X \cup \{e\})| \geq 4$, then \Cref{closure_deletable} implies that $M \backslash e \backslash f$ is $3$-connected, so $M$ has a detachable pair, a contradiction. Thus, either $|X| = 2$ or $|E(M)-(X \cup \{e,f\})| = 2$. Noting that $f \in \cl_{M \backslash e}(X)$ and $f \in \cl_{M \backslash e}(E(M) - (X \cup \{e\}))$, this implies that $f$ is contained in a triangle $T$ of $M \backslash e$. Since $f$ is not contained in a triad of $M \backslash e$, Tutte's Triangle Lemma implies that there exists $x \in T$ such that $M \backslash e \backslash x$ is $3$-connected, a contradiction. 
	
    We deduce that $f$ is contained in a triad $T^*$ of $M \backslash e$. Since $f \in \cl(X)$, orthogonality implies that there exists $g \in T^* \cap X$. Furthermore, if $T^* \subseteq X \cup \{f\}$, then $f \in \cl_{M \backslash e}(X)$ and $f \in \cl^*_{M \backslash e}(X)$. This implies $\lambda_{M \backslash e}(X \cup \{f\}) < 2$, a contradiction to the $3$-connectivity of $M \backslash e$. Thus, $T^* = \{f,g,h\}$ with $h \notin X$. Since $f$ is not contained in a triad of $M$, we have that $T^* \cup \{e\}$ is a cocircuit of $M$, as required.
\end{proof}

\begin{lemma} \label{circuit_so_deletable}
	Let $M$ be a $3$-connected matroid. Let $C = \{e,f,g,h\}$ be a $4$-element circuit of $M$ such that $\{g,h\}$ is contained in a triad of $M$. If $e$ is not contained in a triad and $M / f$ is $3$-connected, then $M \backslash e$ is $3$-connected.
\end{lemma}

\begin{proof}
    Suppose $e$ is not contained in a triad and $M / f$ is $3$-connected, but $M \backslash e$ is not $3$-connected. Then $M$ has a cyclic $3$-separation $(P,\{e\},Q)$. By the dual of \cref{fan_vert_sep}, we may assume that the triad containing $\{g,h\}$ is contained in $P$. If $f \in P$, then $C-\{e\} \subseteq P$. This means that $e \in \cl(P) \cap \cl^*(Q)$, a contradiction to orthogonality. Thus, $f \in Q$, and $f \in \cl(P \cup \{e\})$. By \Cref{delconnectivity}, $\lambda_{M / f}(P \cup \{e\}) = \lambda_{M / f}(Q-\{f\}) = 1$. But $|P \cup \{e\}| \geq 4$ and $|Q-\{f\}| \geq 2$, so this contradicts the $3$-connectivity of $M / f$ and completes the proof.
\end{proof}

\begin{lemma} \label{deletable_circuit_gives_cocircuit}
	Let $M$ be a $3$-connected matroid with no detachable pairs. Let $C = \{e,f,g,h\}$ be a $4$-element circuit of $M$ such that $\{g,h\}$ is contained in a triad of $M$, and $e$ is not contained in a triad of $M$, and $f$ is contained in neither a triangle nor a triad of $M$. Let $x \in E(M) - C$ such that $M \backslash x$ is $3$-connected. Then $M$ has a $4$-element cocircuit containing $x$ and either $e$ or $f$.
\end{lemma}

\begin{proof}
    Suppose neither $e$ nor $f$ is contained in a triad of $M \backslash x$. Since $M \backslash x \backslash e$ is not $3$-connected, \Cref{circuit_so_deletable} implies that $M \backslash x / f$ is not $3$-connected. Since $f$ is not contained in a triangle of $M$, and thus is also not contained in a triangle of $M \backslash x$, this implies that $\si(M \backslash x / f)$ is not $3$-connected. Hence, by Bixby's Lemma, $\co(M \backslash x \backslash f)$ is $3$-connected. But $f$ is not contained in a triad of $M \backslash x$, so $M \backslash x \backslash f$ is $3$-connected, and $M$ has a detachable pair. This contradiction implies that $M \backslash x$ has a triad $T^*$ containing either $e$ or $f$. Since neither $e$ nor $f$ is contained in a triad of $M$, this means that $T^* \cup \{x\}$ is a $4$-element cocircuit of $M$, completing the proof.
\end{proof}

The following is a consequence of \cite[Proposition~8.2.7]{oxley!!!}.
\begin{lemma} \label{contract_then_delete}
	Let $M$ be a $3$-connected matroid, and let $e$ and $f$ be distinct elements of $E(M)$ such that $M / e \backslash f$ is $3$-connected. Then either $M \backslash f$ is $3$-connected, or $\{e,f\}$ is contained in a triad of $M$. 
\end{lemma}

\begin{lemma} \label{circuit_create_deletable_elements}
	Let $M$ be a $3$-connected matroid with no detachable pairs. Let $C$ be a $4$-element circuit of $M$, and let $e \in C$ such that $M / e$ is $3$-connected and is neither a wheel nor a whirl. Then there is a maximal fan of $M / e$ containing $C - \{e\}$ with ends $e^-$ and $e^+$ such that
	\begin{enumerate}
		\item either $\{e^-,e\}$ is contained in a triad of $M$ or $M \backslash e^-$ is $3$-connected, and
		\item either $\{e^+,e\}$ is contained in a triad of $M$ or $M \backslash e^+$ is $3$-connected.
	\end{enumerate}
\end{lemma}

\begin{proof}
    In $M / e$, the set $C-\{e\}$ is a triangle. If $C - \{e\}$ is not contained in a $4$-element fan of $M / e$, then Tutte's Triangle Lemma implies that there exist distinct $e^-,e^+ \in C-\{e\}$ such that $M / e \backslash e^-$ and $M / e \backslash e^+$ are $3$-connected. By \Cref{contract_then_delete}, either $\{e^-,e\}$ is contained in a triad of $M$, or $M \backslash e^-$ is $3$-connected. Similarly, either $\{e^+,e\}$ is contained in a triad of $M$, or $M \backslash e^+$ is $3$-connected. Thus, the result holds.
	
	Now assume that $M / e$ has a maximal fan of length at least four containing $C-\{e\}$. Let $e^-$ and $e^+$ be the ends of this fan. Since $M / e$ is not a wheel or a whirl, we have that either $e^-$ is contained in a triad and not a triangle, in which case $M / e / e^-$ is $3$-connected, or $e^-$ is contained in a triangle and not a triad, in which case $M / e \backslash e^-$ is $3$-connected. Since $M$ has no detachable pairs, $M / e \backslash e^-$ is $3$-connected. Similarly, $M / e \backslash e^+$ is $3$-connected. The lemma now follows from \Cref{contract_then_delete}.
\end{proof}

The next lemma will be used frequently throughout the proof of \Cref{detachable_main}. We introduce the following terminology. A \emph{deletion certificate} in a matroid $M$ is a triple $(e,X_1,\{X_2,X_3,\ldots,X_k\})$, where $e \in E(M)$, $k \geq 2$, and $X_i \subseteq E(M) - \{e\}$ for each $i \in [k]$, such that
\begin{enumerate}
	\item $X_1 \cap X_2 \cap \cdots \cap X_k = \emptyset$,
	\item either $\lambda(X_1) = 2$, or $X_1 \cup \{e\}$ is a quad,
	\item $e \in \cl(X_i)$ for all $i \in [k]$, and
	\item $e$ is not contained in a triad.
\end{enumerate}
For a set $Z \subseteq E(M)$ and a deletion certificate $\mathcal C = (e,X,\mathcal Y)$, we say that $Z$ \emph{contains} $\mathcal C$ (or $\mathcal C$ is \emph{contained in} $Z$) if $\{e\} \cup X \cup \bigcup_{Y \in \mathcal Y} Y \subseteq Z$. Intuitively, if $M$ is a matroid with a deletion certificate, and $M \backslash x$ is $3$-connected for some element $x$ that is not in the certificate, then $M$ has a detachable pair. We make this precise in what follows.

\begin{lemma} \label{deletable_collection}
	Let $M$ be a $3$-connected matroid with no detachable pairs. Let $X \subseteq E(M)$ such that $\lambda(X) = 2$, and $|E(M)| \geq |X| + 3$. If $X$ contains a deletion certificate, then, for all $x \in E(M) - X$, the matroid $M \backslash x$ is not $3$-connected.
\end{lemma}

\begin{proof}
    Let $(\{e\},X_1,\{X_2,X_3,\ldots,X_k\})$ be a deletion certificate contained in $X$. Suppose there exists $x \in E(M)-X$ such that $M \backslash x$ is $3$-connected. If $\lambda(X_1) = 2$, then, as $|E(M)| \geq |X| + 3 \geq |X_1| + 4$, it follows by \Cref{deletable_cocircuit} that $M$ has a $4$-element cocircuit containing $\{e,x\}$. Furthermore, if $X_1 \cup \{e\}$ is a quad, then, as $X_1 \cup \{e\}$ is still a quad in $M \backslash x$ and $M \backslash x \backslash e$ is not $3$-connected, \Cref{quad_detachable} implies that $M \backslash x$ has a triad containing $e$, so $M$ has a $4$-element cocircuit containing $\{e,x\}$. In either case, the matroid $M$ has a $4$-element cocircuit $C^*$ containing $\{e,x\}$. Since $e \in \cl(X_1)$, orthogonality implies that there exists $f \in C^*$ with $f \in X_1$. But $X_1 \cap X_2 \cap \cdots \cap X_k = \emptyset$, so there exists $i \in [k]$ such that $f \notin X_i$. Now, orthogonality implies that $C^*$ contains an element of $X_i$, so $C^* = \{x,e,f,g\}$ with $f \in X_1$ and $g \in X_i$. But now $x \in \cl^*(X)$, so $\lambda_{M \backslash x}(X) = 1$. Since $|E(M \backslash x)| \geq |X| + 2$, this contradicts that $M \backslash x$ is $3$-connected, which completes the proof.
\end{proof}

\begin{lemma} \label{deletable_collection1}
	Let $M$ be a $3$-connected matroid with no detachable pairs. Let $X \subseteq E(M)$ such that $\lambda(X) = 2$ and $|E(M)| \geq |X| + 3$, and suppose that $X$ contains a deletion certificate. If $y \in E(M) - X$ and $y$ is contained in a triangle, then $y$ is contained in a triad.
\end{lemma}

\begin{proof}
	Suppose there exists $y \in E(M) - X$ such that $y$ is contained in a triangle, but $y$ is not contained in a triad. If $y$ is contained in a $4$-element fan, then $y$ is an end of this fan since $y$ is not contained in a triad. This implies $M \backslash y$ is $3$-connected, contradicting \Cref{deletable_collection}. Hence, $y$ is not contained in a $4$-element fan. 
	
    Now suppose $|E(M)| = |X| + 3$. Since $\lambda(X) = 2$ and $y$ is not contained in a triad, we have that $E(M)-X$ is a triangle. Furthermore, $y$ is not contained in a $4$-element fan, so Tutte's Triangle Lemma implies that there exist distinct $e,f \in E(M)-X$ such that $M \backslash e$ and $M \backslash f$ are both $3$-connected. This contradiction to \Cref{deletable_collection} implies that $|E(M)| \geq |X| + 4$. 
	
    Let $T$ be a triangle containing $y$. By Tutte's Triangle Lemma, there exist distinct $e,f \in T$ such that $M \backslash e$ and $M \backslash f$ are both $3$-connected. Thus, $e,f \in X$. But now $y \in \cl(X)$ and $|E(M)| \geq |X| + 4$, which implies, by \Cref{closure_deletable}, that $M \backslash y$ is $3$-connected. This again contradicts \cref{deletable_collection}, which completes the proof.
\end{proof}

\begin{lemma} \label{deletable_collection_contractable_el}
	Let $M$ be a $3$-connected matroid with no detachable pairs. Let $X \subseteq E(M)$ such that $\lambda(X) = 2$, and $|E(M)| \geq |X| + 3$, and $X$ contains a deletion certificate. Suppose there exists $Y \subseteq X$ and $y \in X - Y$ such that $\lambda(Y) = 2$, and $y \in \cl^*(Y)$, and $y$ is not contained in a triangle of $M$. Furthermore, suppose, for all $y' \in Y \cup \{y\}$, that $y' \in \cl(X - \{y'\})$. Then every element of $E(M) - X$ is contained in a triad.
\end{lemma}

\begin{proof}
    First, we show that we may choose $Y$ and $y$ satisfying the hypothesis such that $M / y$ is $3$-connected. If $|Y| \geq 3$, then the dual of \Cref{closure_deletable} implies that $M / y$ is $3$-connected, as desired. Otherwise, $|Y| = 2$, so $Y \cup \{y\}$ is a triad. If $Y \cup \{y\}$ meets a triangle, then $Y \cup \{y\}$ is contained in a maximal fan of at least four elements. Since $y$ is not contained in a triangle, $y$ is an end of this fan, so $M / y$ is $3$-connected. Thus, we may assume $Y \cup \{y\}$ does not meet a triangle, in which case Tutte's Triangle Lemma implies that there exists $y' \in Y \cup \{y\}$ such that $M / y'$ is $3$-connected. Now, $y' \in \cl^*((Y \cup \{y\}) - \{y'\})$ and $y'$ is not contained in a triangle of $M$, so we may replace $y$ with $y'$ and $Y$ with $(Y \cup \{y\}) - \{y'\}$.
    
    Now suppose there exists $f \in E(M) - X$ such that $f$ is not contained in a triad of $M$. By \Cref{deletable_collection1}, the element $f$ is also not contained in a triangle. Now, Bixby's Lemma implies that either $M / f$ or $M \backslash f$ is $3$-connected. By \Cref{deletable_collection}, the matroid $M \backslash f$ is not $3$-connected, and so $M / f$ is $3$-connected. Since $|E(M)| \geq |X| + 3 \geq |Y| + 4$, the dual of \Cref{contractable_circuit} implies that there is a $4$-element circuit $C=\{f,y,z,g\}$ for some $z \in Y$ and $g \notin Y$. Furthermore, if $g \in X$, then $f \in \cl(X)$. But this contradicts the $3$-connectivity of $M / f$, since $|E(M)| \geq |X| + 3$, so $g \notin X$.
	
    We prove that $g$ is contained in a triad of $M$. Suppose this is not the case. The matroid $M / y$ is $3$-connected, and the set $\{f,g,z\}$ is a triangle of $M / y$. Furthermore, neither $f$ nor $g$ is contained in a triad of $M$, so neither $f$ nor $g$ is contained in a triad of $M / y$. This implies that $\{f,g,z\}$ does not meet a triad of $M / y$, so $\{f,g,z\}$ is a maximal fan.
    By \Cref{circuit_create_deletable_elements}, there exist distinct $y^-, y^+ \in \{f,g,z\}$ such that either $\{y^-,y\}$ is contained in a triad of $M$ or $M \backslash y^-$ is $3$-connected, and either $\{y^+,y\}$ is contained in a triad of $M$ or $M \backslash y^+$ is $3$-connected.
    Now either $y^- \in \{f,g\}$ or $y^+ \in \{f,g\}$. Without loss of generality, assume the former. Neither $f$ nor $g$ is contained in a triad, which implies $M \backslash y^-$ is $3$-connected. But $y^- \notin X$, contradicting \cref{deletable_collection}.
	
    So $g$ is contained in a triad $T^*$ of $M$. By orthogonality with $C$, the triad $T^*$ contains an element in $\{f,y,z\}$. Now, $f$ is not contained in a triad, so $T^*$ contains either $y$ or $z$. We have that $y \in \cl(X - \{y\})$ and $z \in \cl(X - \{z\})$, so orthogonality implies that $g \in \cl^*(X)$, and thus $\lambda(X \cup \{g\}) = 2$. Now, $f \in \cl(X \cup \{g\})$, so $\lambda_{M/f}(X \cup \{g\})=1$, but $M / f$ is $3$-connected, which implies that $|E(M / f)| \leq |X \cup \{g\}| + 1$, that is, $|E(M)| = |X| + 3$. But $\lambda(E(M) - X) = 2$, so $E(M) - X$ is either a triangle or a triad containing $f$, a contradiction. We conclude that $f$ is contained in a triad of $M$.
\end{proof}

Dually, a \emph{contraction certificate} of a matroid $M$ is a triple $(e,X_1,\{X_2,X_3,\ldots,X_k\})$, where $e \in E(M)$, $k \geq 2$, and $X_i \subseteq E(M) - \{e\}$ for each $i \in [k]$, such that
\begin{enumerate}
	\item $X_1 \cap X_2 \cap \cdots \cap X_k = \emptyset$,
	\item either $\lambda(X_1) = 2$, or $X_1 \cup \{e\}$ is a quad,
	\item $e \in \cl^*(X_i)$ for all $i \in [k]$, and
	\item $e$ is not contained in a triangle.	
\end{enumerate}
We will show, loosely speaking, that if a matroid with no detachable pairs has both a deletion and contraction certificate, then any element outside of these certificates is in a fan of length at least four.
First, we apply \cref{deletable_collection,deletable_collection1,deletable_collection_contractable_el} to $M^*$.

\begin{corollary} \label{contractable_collection}
	Let $M$ be a $3$-connected matroid with no detachable pairs. Let $X \subseteq E(M)$ such that $\lambda(X) = 2$, and $|E(M)| \geq |X| + 3$. If $X$ contains a contraction certificate, then, for all $x \in E(M)-X$, the matroid $M / x$ is not $3$-connected.
\end{corollary}

\begin{corollary} \label{contractable_collection1}
	Let $M$ be a $3$-connected matroid with no detachable pairs. Let $X \subseteq E(M)$ such that $\lambda(X) = 2$ and $|E(M)| \geq |X| + 3$, and suppose that $X$ contains a contraction certificate. If $y \in E(M) - X$ and $y$ is contained in a triad, then $y$ is contained in a triangle.
\end{corollary}

\begin{corollary} \label{contractable_collection_deletable_el}
	Let $M$ be a $3$-connected matroid with no detachable pairs. Let $X \subseteq E(M)$ such that $\lambda(X) = 2$, and $|E(M)| \geq |X| + 3$, and $X$ contains a contraction certificate. Suppose there exists $Y \subseteq X$ and $y \in X - Y$ such that $\lambda(Y) = 2$, and $y \in \cl(Y)$, and $y$ is not contained in a triad of $M$. Furthermore, suppose, for all $y' \in Y \cup \{y\}$, that $y' \in \cl^*(X - \{y'\})$. Then every element of $E(M) - X$ is contained in a triangle.
\end{corollary}

\begin{lemma} \label{no_other_elements}
	Let $M$ be a $3$-connected matroid with no detachable pairs. Let $X \subseteq E(M)$ such that $\lambda(X) = 2$, and $|E(M)| \geq |X| + 3$, and $X$ contains a deletion certificate. Let $Y \subseteq E(M)$ such that $\lambda(Y) = 2$, and $|E(M)| \geq |Y| + 3$, and $Y$ contains a contraction certificate. Then every element of $E(M) - (X \cup Y)$ is contained in a maximal fan of length at least four with ends in $X \cup Y$.
\end{lemma}

\begin{proof}
	Let $e \in E(M) - (X \cup Y)$. To show the result, it is sufficient to prove that $e$ is contained in both a triangle and a triad. If $e$ is contained in neither a triangle nor a triad, then Bixby's Lemma implies that either $M \backslash e$ or $M / e$ is $3$-connected, contradicting either \Cref{deletable_collection} or \Cref{contractable_collection}. By \Cref{deletable_collection1}, if $e$ is contained in a triangle then $e$ is also contained in a triad. Dually, by \Cref{contractable_collection1}, if $e$ is contained in a triad, then $e$ is also contained in a triangle. This completes the proof.
\end{proof}

We now consider specific structures which may arise in $3$-connected matroids with no detachable pairs.

\begin{lemma} \label{three_deletable}
    Let $M$ be a $3$-connected matroid with no detachable pairs. Let $X \subseteq E(M)$ such that $\lambda(X) = 2$, and $|X| \geq 3$, and $|E(M)| \geq |X| + 7$, and, for all $x \in X$, we have $x \in \cl^*(X - \{x\})$. Suppose there exist distinct $a,b,c \in E(M)-X$ such that $\{a,b,c\} \subseteq \cl(X)$ and none of $a$, $b$, and $c$ are contained in a triad. Then there exist distinct $d,e,f \in E(M) - (X \cup \{a,b,c\})$ such that $\{d,e,f\} \subseteq \cl^*(X \cup \{a,b,c\})$ and none of $d$, $e$, and $f$ are contained in a triangle.
\end{lemma}

\begin{proof}
    By \Cref{closure_deletable}, each of $M \backslash a$, $M \backslash b$, and $M \backslash c$ is $3$-connected. Hence, by \Cref{deletable_cocircuit}, there is a $4$-element cocircuit $C_1^* = \{a,b,d,x\}$ of $M$, where $x \in X$ and $d \notin X \cup \{a,b\}$. Moreover, $d \neq c$, for otherwise $\lambda(X \cup \{a,b,c\}) \le 1$. Similarly, $M$ has $4$-element cocircuits $\{a,c,e,y\}$ and $\{b,c,f,z\}$ with $y,z \in X$ and $e,f \notin X \cup \{a,b,c\}$. Note that these cocircuits are all distinct.
	
    If $d = e$, then cocircuit elimination implies that $M$ has a cocircuit $C^*$ contained in $\{a,b,c,x,y\}$. The cocircuit $C^*$ contains at least one of $a$, $b$, and $c$. If $a \in C^*$, then $a \in \cl^*(X \cup \{b,c\})$, so $\lambda(X \cup \{a,b,c\}) \le 1$, a contradiction. Similar contradictions are obtained if $b \in C^*$ or $c \in C^*$. Thus, $d \neq e$. By symmetry, $d$, $e$, and $f$ are distinct. Furthermore, $\{d,e,f\} \subseteq \cl^*(X \cup \{a,b,c\})$.
	
	To complete the proof, we show that none of $d$, $e$, and $f$ are contained in a triangle. Suppose $M$ has a triangle $T$ containing $d$. By orthogonality, $T$ contains an element of $\{a,b,x\}$. If $x \in T$, then, since $x \in \cl^*(X - \{x\})$, orthogonality implies that $T$ contains a second element of $X$. But now $d \in \cl(X \cup \{a,b\})$ and $d \in \cl^*(X \cup \{a,b\})$, a contradiction. If $a \in T$, then orthogonality with $\{a,c,e,y\}$ implies that $T$ contains one of $\{c,e,y\}$, so $d \in \cl(X \cup \{a,b,c,e\})$ and $d \in \cl^*(X \cup \{a,b,c,e\})$. This is a contradiction since $|E(M)| \geq |X \cup \{a,b,c,d,e\}| + 2$. Finally, if $b \in T$, then $T$ contains one of $\{c,f,z\}$, so $d \in \cl(X \cup \{a,b,c,f\})$ and $d \in \cl^*(X \cup \{a,b,c,f\})$. This contradiction shows that $d$ is not contained in a triangle, and, similarly, $e$ and $f$ are not contained in triangles.
\end{proof}

The following strengthens \cref{segment_deletable} for matroids with at least 11 elements.

\begin{lemma} \label{no_segment}
    Let $M$ be a $3$-connected matroid such that $|E(M)| \geq 11$. Suppose there exist distinct $a,b,c,d \in E(M)$ such that $r(\{a,b,c,d\}) = 2$. Then $M$ has a detachable pair.
\end{lemma}

\begin{proof}
    Suppose that $M$ has no detachable pairs. If $M$ has a triad $T^*$ that meets $\{a,b,c,d\}$, then orthogonality implies that $T^* \subseteq \{a,b,c,d\}$. But now $\lambda(\{a,b,c,d\}) \le 1$, a contradiction. Thus, $\{a,b,c,d\}$ does not meet a triad. It follows that $$(a,\{b,c\},\{\{b,d\},\{c,d\}\})$$ is a deletion certificate. We shall find an element $z \notin \{a,b,c,d\}$ such that $M \backslash z$ is $3$-connected. Since $\lambda(\{a,b,c,d\}) = 2$ and $|E(M)| \geq 7$, this will contradict \Cref{deletable_collection} and complete the proof.
	
    Let $x$ and $y$ be distinct elements in $\{a,b,c,d\}$. By \Cref{segment_deletable}, we have that $M \backslash x$ is $3$-connected. Thus, as $y \in \cl(\{a,b,c,d\} - \{x,y\})$, it follows by \Cref{contractable_circuit} that there is a $4$-element cocircuit of $M$ containing $\{x,y\}$ and another element of $\{a,b,c,d\}$, and an element that is not in $\{a,b,c,d\}$.
	
    In particular, $M$ has a $4$-element cocircuit $C_1^*$ containing $a$ and $b$. Without loss of generality, let $C_1^* = \{a,b,c,e\}$ with $e \notin \{a,b,c,d\}$. Similarly, $M$ has a $4$-element cocircuit containing $a$ and $d$, which we may assume is $C_2^* = \{a,b,d,f\}$ with $f \notin \{a,b,c,d\}$. If $e = f$, then cocircuit elimination implies $M$ has a cocircuit contained in $\{a,b,c,d\}$, in which case $\lambda(\{a,b,c,d\}) = 1$, a contradiction. So $e \neq f$. Similarly, $M$ has a $4$-element cocircuit containing $c$ and $d$, which we may take to be $C_3^* = \{a,c,d,g\}$ with $g \notin \{a,b,c,d,e,f\}$.
	
	We next apply the dual of \Cref{three_deletable} with $X = \{a,b,c,d\}$. Certainly, $\{e,f,g\} \subseteq \cl^*(\{a,b,c,d\})$ and, for all $x \in \{a,b,c,d\}$, we have that $x \in \cl(\{a,b,c,d\}-\{x\})$. Suppose $e$ is contained in a triangle $T$ of $M$. Then, by orthogonality with $C_1^*$, the triangle $T$ contains one of $\{a,b,c\}$. In turn, orthogonality with either $C_2^*$ or $C_3^*$ implies that $T$ contains a second element of $\{a,b,c,d,f,g\}$. But now $e \in \cl(\{a,b,c,d,f,g\})$ and $e \in \cl^*(\{a,b,c,d,f,g\})$, a contradiction. Hence, the element $e$, and symmetrically $f$ and $g$, is not contained in a triangle. Thus, \Cref{three_deletable} implies that $M$ has elements $h,i,j$ such that $\{h,i,j\} \subseteq \cl(\{a,b,c,d,e,f,g\})$ and none of $h$, $i$, and $j$ are contained in a triad. In particular, by \Cref{closure_deletable}, the matroid $M \backslash h$ is $3$-connected, a contradiction which completes the proof.
\end{proof}

\begin{lemma} \label{odd_fan}
	Let $M$ be a $3$-connected matroid with no detachable pairs. Let $F = (e_1,e_2,\ldots,e_{|F|})$ be a maximal fan with odd length at least five such that $\{e_1,e_2,e_3\}$ is a triangle. Then $|F| = 5$, and there exists $z \in E(M) - F$ such that $\{e_1,e_3,e_5,z\}$ is a cocircuit.
\end{lemma}

\begin{proof}
    Since $|F|$ is odd, the set $\{e_{|F|-2},e_{|F|-1},e_{|F|}\}$ is also a triangle. Therefore, $M \backslash e_{|F|}$ is $3$-connected. By \Cref{not_just_fan}, and observing that $M$ is not a wheel or a whirl since $M$ has a maximal fan of odd length, we have that $|E(M)| \geq |F| + 2 \geq |\{e_2,e_3\}| + 4$. Thus, as $e_1 \in \cl(\{e_2,e_3\})$, it follows by \Cref{contractable_circuit} that there is a $4$-element cocircuit $C^*$ of $M$ containing $\{e_1,e_{|F|}\}$. There exists $z \in C^*$, with $z \notin F$, as otherwise $e_1 \in \cl^*(F - \{e_1\})$ and $\lambda(F) < 2$. Furthermore, by orthogonality, $C^*$ contains one element of $\{e_2,e_3\}$ and one element of $\{e_{|F|-2},e_{|F|-1}\}$. The only possibility is $|F| = 5$ and $e_3 \in C^*$, which completes the proof.
\end{proof}

\begin{lemma} \label{disjoint_triads}
	Let $M$ be a $3$-connected matroid with no detachable pairs such that $|E(M)| \geq 8$. Let $F = (e_1,e_2,\ldots,e_{|F|})$ be a maximal fan of $M$ such that $|F| \geq 3$ and $\{e_1,e_2,e_3\}$ is a triad, and let $T^*$ be a triad of $M$ that is not contained in a $4$-element fan. Then one of the following holds:
	\begin{enumerate}
        \item $|F| = 3$ and $F \cap T^* \neq \emptyset$,
        \item $e_1 \in T^*$, 
        \item $F$ is a $4$-element-fan affixed to $T^*$, or
        \item $M | (F \cup T^*) \cong M(K_{2,3})$.
	\end{enumerate}
\end{lemma}

\begin{proof}
    Suppose neither (i) nor (ii) holds. Note that this implies, by \Cref{intersecting_fans}, that the triads $\{e_1,e_2,e_3\}$ and $T^*$ are disjoint. Let $x \in T^*$. If $|F| \geq 4$, then $M / e_1$ is $3$-connected. If $|F| = 3$, then $F$ is a triad not contained in a $4$-element fan, and Tutte's Triangle Lemma implies that at least two of $M / e_1$, $M / e_2$, and $M / e_3$ are $3$-connected, so, without loss of generality, we may assume that $M / e_1$ is $3$-connected. In either case, $x \in \cl^*(T^*-\{x\})$, so the dual of \Cref{contractable_circuit} implies that there is a $4$-element circuit $C_1$ of $M$ containing $\{x,e_1\}$ and another element of $T^*$. By orthogonality, $C_1 = \{e_1,e_i,x,y\}$ with $i \in \{2,3\}$ and $y \in T^*$. Let $z$ be the unique element of $T^* - \{x,y\}$. \Cref{contractable_circuit} again implies that there is a $4$-element circuit $C_2$ of $M$ containing $\{e_1,z\}$, another element of $T^*$, and an element in $\{e_2,e_3\}$. Without loss of generality, let $C_2 = \{e_1,e_j,x,z\}$ with $j \in \{2,3\}$. If $i = j$, then, by circuit elimination, $M$ has a circuit $C$ contained in $\{x,y,z,e_1\}$. By orthogonality with $\{e_1,e_2,e_3\}$, we have that $e_1 \notin C$. Therefore, $T^*$ contains a circuit, a contradiction to the $3$-connectivity of $M$. Hence, $i \neq j$, and so, without loss of generality, $C_1 = \{x,y,e_1,e_2\}$ and $C_2 = \{x,z,e_1,e_3\}$.
	
    If $|F| \geq 5$, then $C_2$ intersects the triad $\{e_3,e_4,e_5\}$ in one element, a contradiction. Therefore, $|F| \leq 4$. Suppose $|F| = 4$. In this case, we show that $F$ is a $4$-element-fan affixed to $T^*$. It suffices to show that $e_4 \in \cl(T^*)$. Since $\{e_1,e_2,e_3\}$ is a triad, submodularity implies that $$r(T^* \cup \{e_4\}) + r(M) \leq r(T^* \cup F)+r(M)-1.$$ Therefore, $r(T^* \cup \{e_4\}) = 3$, so $e_4 \in \cl(T^*)$, and $F$ is a $4$-element-fan affixed to $T^*$.
	
    Finally, suppose $|F| = 3$. Either $M / e_2$ or $M / e_3$ is $3$-connected. Without loss of generality, we may assume $M / e_2$ is $3$-connected. Since $z \in \cl^*(T^*-\{z\})$, the dual of \Cref{contractable_circuit} implies that $M$ has a $4$-element circuit $C_3$ containing $\{e_2,z\}$, one of $e_1$ and $e_3$, and one of $x$ and $y$. If $e_1 \in C_3$, then circuit elimination with $C_1$ implies that $M$ has a circuit contained in $T^* \cup \{e_2\}$, and orthogonality with $\{e_1,e_2,e_3\}$ implies that $M$ has a circuit contained in $T^*$, a contradiction. Similarly, if $x \in C_3$, then circuit elimination with $C_2$ and orthogonality implies that $M$ has a circuit in $\{e_1,e_2,e_3\}$. Therefore, $C_3 = \{e_2,e_3,y,z\}$, which implies that $M | (F \cup T^*) \cong M(K_{2,3})$, completing the proof.
\end{proof}

\begin{lemma} \label{k3m_paddle}
    Let $M$ be a $3$-connected matroid. Let $(P_1,P_2,\ldots,P_m)$ be a partition of $E(M)$, where $m \geq 2$, such that $|P_1| \geq 2$ and, for all $i \in \{2,3,\ldots,m\}$ and $j \in [m]-\{i\}$, the set $P_i$ is a triad and $r(P_i \cup P_j) = r(P_j) + 1$. Then $(P_1,P_2,\ldots,P_m)$ is a paddle of $M$.
\end{lemma}

\begin{proof}
	First, we show that $(P_1,P_2,\ldots,P_m)$ is an anemone of $M$. Let $J$ be a proper non-empty subset of $[m]$, and let $X = \bigcup_{i \in J} P_i$. We show that $\lambda(X) = 2$. First, assume that $1 \notin J$. If $|J| = 1$, then $X$ is a triad, so $\lambda(X) = 2$. Otherwise, let $i \in J$, and assume that $\lambda(X-P_i) = 2$. Now, $r(X) \leq r(X-P_i) + 1$, and, since $P_i$ is a triad, $r^*(X) \leq r^*(X-P_i) + 2$. Thus,
	\[
	\lambda(X) \leq (r(X-P_i) + 1) + (r^*(X-P_i) + 2) - (|X-P_i| + 3) = 2
	\]
	Thus, $\lambda(X) = 2$, as desired. Finally, if $1 \in J$, then $1 \notin [m]-J$. Hence, $\lambda(X) = \lambda(\bigcup_{i \in [m]-J} P_i) = 2$.
	
	Let $i,j$ be distinct elements of $[m]$. To complete the proof, we show that $\sqcap(P_i,P_j) = 2$. Suppose, without loss of generality, that $i \neq 1$. Then
	\begin{align*}
	\sqcap(P_i,P_j) &= r(P_i) + r(P_j) - r(P_i \cup P_j) \\
	&= 3 + r(P_j) - (r(P_j) + 1) = 2.
	\end{align*}
	Thus, $(P_1,P_2,\ldots,P_m)$ is a paddle of $M$.
\end{proof}

\section{Disjoint fans} \label{disjoint_fans}
Armed with the lemmas from the previous sections, we begin the proof of \Cref{detachable_main} in earnest. The proof of \Cref{detachable_main} is partitioned into four parts depending on the fans present in the matroid. In this section, we consider the case when the matroid has two disjoint maximal fans $F_1$ and $F_2$, where $F_1$ has length at least four and $F_2$ has length at least three. In particular, we prove the following theorem:

\begin{theorem} \label{disjoint_fans_with_like_ends}
	Let $M$ be a $3$-connected matroid such that $|E(M)| \geq 13$. Let $F_1 = (e_1,e_2,\ldots,e_{|F_1|})$ and $F_2 = (f_1,f_2,\ldots,f_{|F_2|})$ be disjoint maximal fans of $M$ such that $|F_1| \geq 4$ and $|F_2| \geq 3$. If $\{e_1,e_2,e_3\}$ and $\{f_1,f_2,f_3\}$ are both triads, then one of the following holds:
	\begin{enumerate}
		\item $M$ has a detachable pair,
		\item $M$ is an even-fan-spike,
		\item $M$ is a hinged triad-paddle, or
		\item $M$ is a quasi-triad-paddle with an augmented-fan petal.
	\end{enumerate}
\end{theorem}

\subsection*{\texorpdfstring{$\mathbf{F_2}$}{F2} has length three}
First, we consider the case where $F_2$ is a triad, and show, as \cref{even_fan_k3m}, that either $M$ has a detachable pair, or $M$ is a hinged triad-paddle.
We start with a lemma that shows, in particular, that if $M$ has no detachable pairs, then $F_1$ has length four.

\begin{lemma} \label{even_fan_k3m1}
	Let $M$ be a $3$-connected matroid with no detachable pairs such that $|E(M)| \geq 10$. Let $F_1 = (e_1,e_2,\ldots,e_{|F_1|})$ be a maximal fan of $M$ such that $|F_1| \geq 4$ and $\{e_1,e_2,e_3\}$ is a triad. Let $F_2$ be a triad of $M$ that is disjoint from $F_1$ and not contained in a $4$-element fan. Then
	\begin{enumerate}
        \item $F_1$ is a $4$-element-fan affixed to $F_2$, and
		\item every element of $E(M) - (F_1 \cup F_2)$ is contained in a triad.
	\end{enumerate}
\end{lemma}

\begin{proof}
	Since $F_1$ and $F_2$ are disjoint, \Cref{disjoint_triads}(i) and (ii) do not hold. Furthermore, $|F_1| \geq 4$, which means \Cref{disjoint_triads}(iv) does not hold. Therefore, \Cref{disjoint_triads}(iii) holds, and $F_1$ is a $4$-element-fan affixed to $F_2$.
	
    Note that $e_4 \in \cl(F_1 - \{e_4\})$ and, as (i) holds, $e_4 \in \cl(F_2)$. Furthermore, $\lambda(F_1 - \{e_4\}) = 2$, and $e_4$ is not contained in a triad by \Cref{fan_ends}. Therefore, $(e_4,F_1-\{e_4\},\{F_2\})$ is a deletion certificate. Also, $\lambda(F_1 \cup F_2) = 2$ and $|E(M)| \geq |F_1 \cup F_2| + 3 = 10$. Furthermore, $e_1 \in \cl^*(F_1-\{e_1\})$ and, for all $i \in \{1,2,3,4\}$, we have that $e_i \in \cl(F_1 \cup F_2)$. Thus, by \Cref{deletable_collection_contractable_el}, every element of $E(M)-(F_1 \cup F_2)$ is contained in a triad, which completes the proof. 
\end{proof}

The last lemma implies that there is a triad outside of $F_1 \cup F_2$.  The next lemma addresses when the triad is contained in a $4$-element fan, and the subsequent lemma addresses when it is not.

\begin{lemma} \label{even_fan_k3m2}
	Let $M$ be a $3$-connected matroid with no detachable pairs such that $|E(M)| \geq 10$. Let $F_1 = (e_1,e_2,e_3,e_4)$ be a maximal fan of $M$ such that $\{e_1,e_2,e_3\}$ is a triad. Let $F_2$ be a triad of $M$ that is disjoint from $F_1$ and not contained in a $4$-element fan. Furthermore, let $F_3$ be a maximal fan of $M$, distinct from $F_1$ and $F_2$, such that $|F_3| \geq 4$. Then $e_4 \in F_3$, and $F_3$ is a $4$-element-fan affixed to $F_2$.
\end{lemma}

\begin{proof}
	By \Cref{even_fan_k3m1}, $F_1$	is a $4$-element-fan affixed to $F_2$. This means that $M$ has a $4$-element circuit $C_1$ containing $\{e_1,e_2\}$ and two elements of $F_2$, and a $4$-element circuit $C_2$ containing $\{e_1,e_3\}$ and two elements of $F_2$. Suppose $C_1 \cap F_2 = C_2 \cap F_2$. Circuit elimination and orthogonality imply that $M$ has a circuit contained in $\{e_1,e_2,e_3\}$. But now $\lambda(\{e_1,e_2,e_3\}) < 2$, a contradiction to the $3$-connectivity of $M$. Thus, without loss of generality, take $F_2 = \{f_1,f_2,f_3\}$ such that $C_1 = \{e_1,e_2,f_1,f_2\}$ and $C_2 = \{e_1,e_3,f_1,f_3\}$. Also, let $(g_1,g_2,\ldots,g_{|F_3|})$ be an ordering of $F_3$.
	
    If $\{g_1,g_2,g_3\}$ is a triangle, then $g_1$ is not contained in a triad, and so \Cref{even_fan_k3m1} implies that $g_1 \in F_1 \cup F_2$. The only element of $F_1 \cup F_2$ that is not contained in a triad is $e_4$, so $g_1 = e_4$. Similarly, if $\{g_{|F_3|-2},g_{|F_3|-1},g_{|F_3|}\}$ is a triangle, then $g_{|F_3|} = e_4$. Therefore, as $g_1 \neq g_{|F_3|}$, either $\{g_1,g_2,g_3\}$ or $\{g_{|F_3|-2},g_{|F_3|-1},g_{|F_3|}\}$ is a triad. Without loss of generality, assume the former.
	
    Now, $|F_3| \geq 4$, so \Cref{disjoint_triads}(i) and (iv) do not hold. Suppose $g_1 \in F_1 \cup F_2$. \Cref{fan_ends} implies that $g_1 \neq e_4$, and \Cref{intersecting_fans} implies that $g_2, g_3 \notin F_1 \cup F_2$. Therefore, the triad $\{g_1,g_2,g_3\}$ intersects either the circuit $\{e_1,e_2,f_1,f_2\}$ or the circuit $\{e_1,e_3,f_1,f_3\}$ in one element. This contradiction to orthogonality implies that $g_1 \notin F_1 \cup F_2$, so \Cref{disjoint_triads}(ii) does not hold. Hence, $F_3$ is a $4$-element-fan affixed to $F_2$. This means that $|F_3| = 4$, so, since $F_3$ is maximal, $g_4$ is not contained in a triad, and thus $g_4 = e_4$, thereby completing the proof.
\end{proof}

\begin{lemma} \label{even_fan_k3m3}
    Let $M$ be a $3$-connected matroid with no detachable pairs such that $|E(M)| \geq 11$. Let $F_1 = (e_1,e_2,e_3,e_4)$ be a maximal fan of $M$ such that $\{e_1,e_2,e_3\}$ is a triad. Let $F_2$ be a triad of $M$ that is disjoint from $F_1$ and not contained in a $4$-element fan. Furthermore, let $F_3 \not\subseteq F_1 \cup F_2$ be a triad of $M$ that is not contained in a $4$-element fan. Then $F_1$ is a $4$-element-fan affixed to $F_3$, and $M | (F_2 \cup F_3) \cong M(K_{2,3})$.
\end{lemma}

\begin{proof}
    Since $F_1$ is a $4$-element-fan affixed to $F_2$, by \cref{even_fan_k3m1}, we may assume that $F_2 = \{f_1,f_2,f_3\}$ such that $\{e_1,e_2,f_1,f_2\}$ and $\{e_1,e_3,f_1,f_3\}$ are circuits. Suppose $F_1$ and $F_3$ are disjoint. Then \Cref{even_fan_k3m1} implies that $F_1$ is a $4$-element-fan affixed to $F_3$. Furthermore, orthogonality with the circuits $\{e_1,e_2,f_1,f_2\}$ and $\{e_1,e_3,f_1,f_3\}$ implies that $F_2$ and $F_3$ are disjoint. Therefore, by \Cref{disjoint_triads}, $M | (F_2 \cup F_3) \cong M(K_{2,3})$.
	
    Now suppose that $F_1 \cap F_3 \neq \emptyset$. This implies, by \cref{intersecting_fans,fan_ends}, that $e_1 \in F_3$. Since $F_3 \not\subseteq F_1 \cup F_2$, we have that $|F_3 \cap F_2| \leq 1$. Thus, orthogonality with $\{e_1,e_2,f_1,f_2\}$ and $\{e_1,e_3,f_1,f_3\}$ implies that $T^* = \{e_1,f_1,e\}$, for some $e \notin F_1 \cup F_2$. Now, $(e_1,F_1-\{e_1\},\{\{f_1,e\}\})$ is a contraction certificate, and $\lambda(F_1 \cup F_2 \cup \{e\})=2$. Additionally, $F_1 \cup F_2 \cup \{e\}$ contains a deletion certificate $(e_4,F_1-\{e_4\},\{F_2\})$.
	
    Let $g \notin F_1 \cup F_2 \cup \{e\}$. By \Cref{no_other_elements}, the element $g$ is contained in a maximal fan $G$ of length at least four. \Cref{even_fan_k3m2} implies that $G$ is a $4$-element-fan affixed to $F_2$, and $e_4 \in G$, so $G$ has an ordering $(g_1,g_2,g_3,e_4)$ such that $\{g_1,g_2,g_3\}$ is a triad. Furthermore, as $M / g_1$ is $3$-connected, $g_1 \in F_1 \cup \{e\}$, by \cref{contractable_collection}. But, by orthogonality, $g_1 \notin F_1$ so $g_1 = e$. Note that, since $G$ is a $4$-element-fan affixed to $F_2$, there is a circuit $C$ of $M$ containing $\{e,g_2\}$ and two elements of $F_2$.
	
    Since $|E(M)| \geq 11$, there exists $h \notin F_1 \cup F_2 \cup G$. As before, $h$ is contained in a maximal fan $H$ of length at least four with ordering $(e,h_2,h_3,e_4)$ such that $\{e,h_2,h_3\}$ is a triad and $H$ is disjoint from $F_2$. But this triad intersects the circuit $C$ in one element, a contradiction. Hence $F_1$ and $F_3$ are disjoint, which completes the proof.
\end{proof}

\begin{lemma} \label{even_fan_k3m}
    Let $M$ be a $3$-connected matroid with no detachable pairs such that $|E(M)| \geq 11$. Let $F_1 = (e_1,e_2,\ldots,e_{|F_1|})$ be a maximal fan of $M$ such that $|F_1| \geq 4$ and $\{e_1,e_2,e_3\}$ is a triad.
    If $M$ has a triad that is disjoint from $F_1$ and is not contained in a $4$-element fan, then $M$ is a hinged triad-paddle.
\end{lemma}

\begin{proof}
    Let $F_2$ be a triad of $M$ that is disjoint from $F_1$ and is not contained in a $4$-element fan.
    By \Cref{even_fan_k3m1}, $F_1$ is a $4$-element-fan affixed to $F_2$, and there is an element $e \notin F_1 \cup F_2$ that is contained in a triad $T^*$. If $T^*$ is contained in a $4$-element fan, then \Cref{even_fan_k3m2} implies that $T^* \cup \{e_4\}$ is a $4$-element-fan affixed to each of $F_1-\{e_4\}$ and $F_2$. Otherwise, \Cref{even_fan_k3m3} implies that $F_1$ is a $4$-element-fan affixed to $T^*$, and $M | (F_2 \cup T^*) \cong M(K_{2,3})$.
	
    It follows that there is a partition $(P_1,P_2,\ldots,P_m,\{e_4\})$ of $E(M)$ with $m \geq 3$ and $P_m = F_2$ such that $P_i$ is a triad for all $i \in [m]$. Furthermore, for all $i \in [m-1]$, either $P_i \cup \{e_4\}$ is a $4$-element-fan affixed to $P_m$, or $M | (P_i \cup P_m) \cong M(K_{2,3})$. By \Cref{k3m_paddle}, $(P_1,P_2,\ldots,P_m \cup \{e_4\})$ is a paddle of $M$, thereby completing the proof.
\end{proof}

\subsection*{\texorpdfstring{$\mathbf{F_2}$}{F2} is odd}
Next, we consider the case where $F_2$ has odd length at least five, and show, as \cref{augmented_fan_petal}, that either $M$ has a detachable pair, or $M$ is a quasi-triad-paddle with an augmented-fan petal.
We start with a lemma that shows, in particular, that if $M$ has no detachable pairs, then $F_1$ has length four and $F_2$ has length five.

\begin{lemma} \label{augmented_fan_petal1}
    Let $M$ be a $3$-connected matroid with no detachable pairs such that $|E(M)| \geq 12$. Let $F_1 = (e_1,e_2,\ldots,e_{|F_1|})$ and $F_2 = (f_1,f_2,\ldots,f_{|F_2|})$ be disjoint maximal fans of $M$ with length at least four, such that $|F_2|$ is odd. If $\{e_1,e_2,e_3\}$ and $\{f_1,f_2,f_3\}$ are both triads, then
	\begin{enumerate}
		\item $|F_1| = 4$ and $|F_2| = 5$ and $F_2 \cup \{e_4\}$ is an augmented fan affixed to $\{e_1,e_2,e_3\}$, and
		\item every element of $E(M) - (F_1 \cup F_2)$ is contained in a triad.
	\end{enumerate}
\end{lemma}

\begin{proof}
    The dual of \Cref{odd_fan} implies that $|F_2| = 5$. By the dual of \cref{contractable_circuit}, there is a $4$-element circuit $C_1$ of $M$ containing $\{e_1,f_1\}$. Orthogonality with the triad $\{e_1,e_2,e_3\}$ implies that either $e_2 \in C_1$ or $e_3 \in C_1$, and orthogonality with the triads $\{f_1,f_2,f_3\}$ and $\{f_3,f_4,f_5\}$ implies that $f_2 \in C_1$. Hence, $C_1 = \{e_1,e_i,f_1,f_2\}$ with $i \in \{2,3\}$. Similarly, $M$ has a $4$-element circuit $C_2 = \{e_1,e_j,f_4,f_5\}$ with $j \in \{2,3\}$. If $i = j$, then circuit elimination implies $M$ has a circuit contained in $\{f_1,f_2,f_4,f_5,e_1\}$, and $e_1$ is not contained in this circuit by orthogonality with $\{e_1,e_2,e_3\}$. But now $M$ has a circuit contained in $\{f_1,f_2,f_4,f_5\}$, which means $\lambda(F_2) \leq 1$, a contradiction. Therefore, either $i = 3$ or $j = 3$, which contradicts orthogonality with $\{e_3,e_4,e_5\}$ if $|F_1| \geq 5$. Hence, $|F_1| = 4$.
	
    Note that $r(F_1 \cup F_2) = r(F_2) + 1$ and $\lambda(F_1 \cup F_2) = 2$. Furthermore, by orthogonality with the triad $\{e_1,e_2,e_3\}$, we have that $e_1,e_2,e_3 \notin \cl(F_2 \cup \{e_4\})$. It follows that $e_4 \in \cl(F_2)$, so $(e_4,F_1-\{e_4\},\{F_2\})$ is a deletion certificate. Since $e_1 \in \cl^*(F_1 - \{e_1\})$ and $|E(M)| \geq |F_1 \cup F_2| + 3 = 12$, \Cref{deletable_collection_contractable_el} implies that every element of $E(M) - (F_1 \cup F_2)$ is contained in a triad.
	
	Now, to show that $F_2 \cup \{e_4\}$ is an augmented fan affixed to $\{e_1,e_2,e_3\}$, it remains to show that $\{f_1,f_3,f_5,e_4\}$ is a circuit of $M$. By the dual of \Cref{odd_fan}, there is a $4$-element circuit $\{f_1,f_3,f_5,z\}$, with $z \notin F_2$. Assume, towards a contradiction, that $z \neq e_4$. It follows, by orthogonality, that $z \notin F_1 \cup F_2$, and thus $z$ is contained in a triad $T^*$. Orthogonality with the circuit $\{f_1,f_3,f_5,z\}$ implies that either $f_1 \in T^*$ or $f_5 \in T^*$. Furthermore, orthogonality with either $C_1$ or $C_2$ implies that $e_1 \in T^*$.
	But now $z \in \cl(F_2)$ and $z \in \cl^*(F_1 \cup F_2)$, so $\lambda(F_1 \cup F_2 \cup \{z\}) \leq 1$. This is a contradiction, since $|E(M)| \geq |F_1 \cup F_2 \cup \{z\}| + 2 = 12$, so $z=e_4$, as required.
\end{proof}

\begin{lemma} \label{augmented_fan_petal2}
    Let $M$ be a $3$-connected matroid with no detachable pairs such that $|E(M)| \geq 12$. Let $F_1 = (e_1,e_2,e_3,e_4)$ and $F_2 = (f_1,f_2,f_3,f_4,f_5)$ be disjoint maximal fans of $M$ such that $\{e_1,e_2,e_3\}$ and $\{f_1,f_2,f_3\}$ are both triads. Let $e \in E(M) - (F_1 \cup F_2)$. Then $e$ is contained in a triad $T^*$ such that $F_2 \cup \{e_4\}$ is an augmented fan affixed to $T^*$, and $M | (T^* \cup \{e_1,e_2,e_3\}) \cong M(K_{2,3})$.
\end{lemma}

\begin{proof}
    By \Cref{augmented_fan_petal1}, $F_2 \cup \{e_4\}$ is an augmented fan affixed to $\{e_1,e_2,e_3\}$, and $e$ is contained in a triad $T^*$. Suppose $T^*$ is not contained in a $4$-element fan. Since $|F_2| = 5$, \Cref{disjoint_triads}(i), (iii), and (iv) do not hold. Thus, $f_1 \in T^*$. Furthermore, by reversing the ordering of $F_2$, we see that $f_5 \in T^*$. Hence, $T^* = \{f_1,f_5,e\}$. But now $F_1$ and $T^*$ are disjoint. This contradicts \Cref{even_fan_k3m}, since $M$ has a $5$-element fan $F_2$.
	
    So $T^*$ is contained in a $4$-element fan. Let $F_3$ be the maximal fan containing $T^*$, and let $(g_1,g_2,\ldots,g_{|F_3|})$ be an ordering of $F_3$. Suppose that $g_1 \in F_2$. Then, by \Cref{fan_ends}, $g_1 \in \{f_1,f_5\}$ and $\{g_1,g_2,g_3\}$ is a triad. Since $F_2 \cup \{e_4\}$ is an augmented fan affixed to $\{e_1,e_2,e_3\}$, orthogonality implies that $F_2 \cup F_3$ is not an $M(K_4)$-separator in $M^*$. Thus, \Cref{intersecting_fans} implies that $g_2,g_3 \notin F_2$, and \Cref{fan_ends} implies that $e_4 \notin \{g_1,g_2,g_3\}$. But now the triad $\{g_1,g_2,g_3\}$ intersects the circuit $\{f_1,f_3,f_5,e_4\}$ in one element, a contradiction. So $g_1 \notin F_2$ and, similarly, $g_{|F_3|} \notin F_2$, which implies $F_2$ and $F_3$ are disjoint. If $\{g_1,g_2,g_3\}$ is a triangle, then $g_1 \in F_1 \cup F_2$, so $g_1 = e_4$. Similarly, if $\{g_{|F_3|-2},g_{|F_3|-1},g_{|F_3|}\}$ is a triangle, then $g_{|F_3|} = e_4$. Therefore, either $\{g_1,g_2,g_3\}$ or $\{g_{|F_3|-2},g_{|F_3|-1},g_{|F_3|}\}$ is a triad, so we may assume that $\{g_1,g_2,g_3\}$ is a triad. Thus, by \Cref{augmented_fan_petal1}, $|F_3| = 4$ and $F_2 \cup \{g_4\}$ is an augmented fan affixed to $T^* = \{g_1,g_2,g_3\}$, with $g_4 = e_4$. Finally, since $F_2 \cup \{e_4\}$ is an augmented fan affixed to both $\{e_1,e_2,e_3\}$ and $T^*$, circuit elimination and orthogonality implies that $M | (T^* \cup \{e_1,e_2,e_3\}) \cong M(K_{2,3})$.
\end{proof}

\begin{lemma} \label{augmented_fan_petal}
    Let $M$ be a $3$-connected matroid with no detachable pairs such that $|E(M)| \geq 12$. Let $F_1 = (e_1,e_2,\ldots,e_{|F_1|})$ and $F_2 = (f_1,f_2,\ldots,f_{|F_2|})$ be disjoint maximal fans of $M$ with length at least four, such that $F_2$ is odd. If $\{e_1,e_2,e_3\}$ and $\{f_1,f_2,f_3\}$ are both triads, then $M$ is a quasi-triad-paddle with an augmented-fan petal.
\end{lemma}

\begin{proof}
    By \Cref{augmented_fan_petal1}, we have that $|F_1| = 4$, $|F_2| = 5$, and $F_2 \cup \{e_4\}$ is an augmented fan affixed to $\{e_1,e_2,e_3\}$. Let $e \notin F_1 \cup F_2$. By \Cref{augmented_fan_petal2}, there exists a triad $T^*$ of $M$ containing $e$ such that $F_2 \cup \{e_4\}$ is an augmented fan affixed to $T^*$, and $M | (\{e_1,e_2,e_3\} \cup T^*) \cong M(K_{2,3})$. It follows that $E(M)$ has a partition $(P_1,P_2,\ldots,P_m)$ such that $P_m = F_2 \cup \{e_4\}$ and $M \backslash P_m \cong M(K_{3,m-1})$ and, for  all $i \in [m-1]$, the set $P_i$ is a triad and $P_m$ is an augmented fan affixed to $P_i$. By \Cref{k3m_paddle}, $(P_1,P_2,\ldots,P_m)$ is a paddle of $M$, completing the proof.
\end{proof}

\subsection*{\texorpdfstring{$\mathbf{F_1}$}{F1} and \texorpdfstring{$\mathbf{F_2}$}{F2} are even}
Finally, we consider the case where both $F_1$ and $F_2$ are even, and show that if $M$ has no detachable pairs, then $M$ is an even-fan-spike.
The next lemma shows that there are two cases to consider.  Subsequently, we prove a series of lemmas that are used in both cases.  \Cref{even_fan_spike4} then addresses the case where \cref{even_fan_spike1}(i) holds, and \cref{even_fan_spike} addresses the case where \cref{even_fan_spike1}(ii) holds.

Notice that, towards proving \cref{disjoint_fans_with_like_ends}, we may assume that $M$ has two disjoint maximal fans, each of which is even with length at least four. However, certain lemmas apply when one of the fans has length two; these lemmas will be useful again later on.

\begin{lemma} \label{even_fan_spike1}
    Let $M$ be a $3$-connected matroid with no detachable pairs. Let $F_1$ and $F_2$ be disjoint maximal fans of $M$, each of which is even with length at least four. Then there exist orderings $(e_1,e_2,\ldots,e_{|F_1|})$ and $(f_1,f_2,\ldots,f_{|F_2|})$ of $F_1$ and $F_2$ respectively such that $\{e_1,e_2,e_3\}$ and $\{f_1,f_2,f_3\}$ are triads, $\{e_1,e_2,f_1,f_2\}$ is a circuit, and either
    \begin{enumerate}
        \item $|F_1| = |F_2| = 4$ and $\{e_2,e_4,f_2,f_4\}$ is a cocircuit, or
        \item $\{e_{|F_1|-1},e_{|F_1|},f_{|F_2|-1},f_{|F_2|}\}$ is a cocircuit.
    \end{enumerate}
\end{lemma}

\begin{proof}
    Let $(e_1,e_2,\ldots,e_{|F_1|})$ and $(f_1,f_2,\ldots,f_{|F_2|})$ be orderings of $F_1$ and $F_2$ respectively such that $\{e_1,e_2,e_3\}$ and $\{f_1,f_2,f_3\}$ are triads, and $\{e_{|F_1|-2},e_{|F_1|-1},e_{|F_1|}\}$ and $\{f_{|F_1|-2},f_{|F_1|-1},f_{|F_1|}\}$ are triangles. By the dual of \Cref{contractable_circuit}, there is a $4$-element circuit $C$ of $M$ containing $\{e_1,f_1\}$. Orthogonality implies that the other two elements of $C$ are $e_2$ or $e_3$, and $f_2$ or $f_3$. If $|F_1| > 4$, then orthogonality with $\{e_3,e_4,e_5\}$ implies that $e_2 \in C$. Furthermore, if $|F_1| = 4$, then, up to the ordering of $F_1$, we may assume that $e_2 \in C$. Similarly, we may assume that $f_2 \in C$. Thus, $C = \{e_1,e_2,f_1,f_2\}$.
	
    By \Cref{deletable_cocircuit}, there is a $4$-element cocircuit $C^*$ of $M$ containing $\{e_{|F_1|},f_{|F_2|}\}$, and, by orthogonality, $e_{|F_1|-2}$ or $e_{|F_1|-1}$, and $f_{|F_2|-2}$ or $f_{|F_2|-1}$. If $C^* = \{e_{|F_1|-1},e_{|F_1|},f_{|F_2|-1},f_{|F_2|}\}$, then (ii) holds. Otherwise, either $e_{|F_1|-2} \in C^*$ or $f_{|F_2|-2} \in C^*$. Without loss of generality, assume the former. If $|F_1| > 4$, then $C^*$ intersects the triangle $\{e_{|F_1|-4},e_{|F_1|-3},e_{|F_1|-2}\}$ in one element, so $|F_1| = 4$. Now, $e_2 \in C \cap C^*$, so orthogonality implies that $f_2 \in C^*$. Thus, $|F_1| = |F_2| = 4$ and $\{e_2,e_4,f_2,f_4\}$ is a cocircuit, so (i) holds, thereby completing the proof of the lemma.
\end{proof}

\begin{lemma} \label{fan_closure}
	Let $M$ be a $3$-connected matroid. Let $F = (e_1,e_2,\ldots,e_{|F|})$ be a maximal fan of $M$ with length at least two such that either $|F| = 2$ or $\{e_1,e_2,e_3\}$ is a triad. Suppose there exists a $4$-element circuit $C = \{e_1,e_i,a,b\}$ of $M$ with $i \in \{2,3\}$ and $a,b \notin F$. Then for all $x \in E(M) - (F \cup C)$, we have that $x \notin \cl^*(F)$.
\end{lemma}

\begin{proof}
    Suppose, to the contrary, that there exists $e \in E(M) - (F \cup C)$ such that $e \in \cl^*(F)$. If $|F| = 2$, then $F \cup \{e\}$ is a triad, which contradicts the maximality of $F$. So we may assume that $|F| \geq 3$. Since $e_1 \in \cl^*(F-\{e_1\})$, we also have that $e \in \cl^*(F - \{e_1\})$, so $\lambda((F - \{e_1\}) \cup \{e\}) = 2$. The circuit~$C$ implies that $e_i \in \cl(E(M) - ((F - \{e_1\}) \cup \{e\}))$, so $\lambda((F - \{e_1,e_i\}) \cup \{e\}) = 2$. In turn, letting $e_j$ be the unique element in $\{e_2,e_3\} - \{e_i\}$, we have $e_j \in \cl^*(E(M) - ((F - \{e_1,e_i\}) \cup \{e\}))$, so $\lambda((F - \{e_1,e_2,e_3\}) \cup \{e\}) = 2$. Repeating in this way, we eventually see that $\lambda(\{e_{|F|-1},e_{|F|},e\}) = 2$, so $\{e_{|F|-1},e_{|F|},e\}$ is either a triangle or a triad. Since $e \in \cl^*(F)$, we have that $\{e_{|F|-1},e_{|F|},e\}$ is a triad. If $\{e_{|F|-2},e_{|F|-1},e_{|F|}\}$ is a triangle, then the fan $F$ is not maximal, a contradiction. Hence, $\{e_{|F|-2},e_{|F|-1},e_{|F|}\}$ is a triad. Orthogonality implies that $|F| = 3$, but now the triad $\{e_{|F|-1},e_{|F|},e\}$ intersects the circuit $C$ in one element, a contradiction.
\end{proof}

\begin{lemma} \label{even_fan_spike_separating}
    Let $M$ be a $3$-connected matroid. Let $F_1,F_2,\ldots,F_k$ be disjoint maximal fans of $M$, each having even length at least two. For all $i \in [k]$, let $F_i = (e^i_1,e^i_2,\ldots,e^i_{|F_i|})$ such that either $|F_i| = 2$ or $\{e^i_1,e^i_2,e^i_3\}$ is a triad. Furthermore, for all distinct $i,j \in [k]$, suppose there is a $4$-element circuit~$C_{i,j}$ containing $\{e^i_1,e^j_1\}$ such that $|C_{i,j} \cap F_i| = |C_{i,j} \cap F_j| = 2$, and a $4$-element cocircuit $C^*_{i,j}$ containing $\{e^i_{|F_i|},e^j_{|F_j|}\}$ such that $|C^*_{i,j} \cap F_i| = |C^*_{i,j} \cap F_j| = 2$. If $|E(M)| \geq |F_1 \cup F_2 \cup \cdots \cup F_k| + 2$, then 
	\begin{enumerate}
		\item $\lambda(\bigcup_{i \in J} F_i) = 2$ for all non-empty subsets $J \subseteq [k]$, and
		\item $\sqcap(F_i,F_j) = 1$ for all distinct $i,j \in [k]$.
	\end{enumerate}
\end{lemma}

\begin{proof}
Suppose $|E(M)| \geq |F_1 \cup F_2 \cup \cdots \cup F_k| + 2$.
Let $J$ be a non-empty subset of $[k]$, and let $X = \bigcup_{i \in J} F_i$. If $|J| = 1$, then $X$ is a fan, so $\lambda(X) = 2$. Otherwise, let $j \in J$, and suppose that $\lambda(X-F_j) = 2$. For some $i \in J-\{j\}$, the circuit $C_{i,j}$ implies that $e^j_1 \in \cl(X - \{e^j_1\})$. But $e^j_1 \notin \cl(F_j - \{e^j_1\})$, and so $r(X) \leq r(X - F_j) + r(F_j) - 1$. Similarly, $r^*(X) \leq r^*(X - F_j) + r^*(F_j) - 1$. Therefore,
\begin{align*}
 \lambda(X) &\leq (r(X-F_j) + r(F_j) - 1) + (r^*(X-F_j) + r^*(F_j) - 1) \\
 & \qquad \qquad \qquad - (|X-F_j| + |F_j|) \\
 &= \lambda(X-F_j) + \lambda(F_j) - 2 = 2.
\end{align*}
Since $M$ is $3$-connected and $|E(M)| \geq |X| + 2$, we have that $\lambda(X) = 2$. Furthermore, if $r(X) < r(X-F_j) + r(F_j) - 1$, then $\lambda(X) < 2$, so $r(X) = r(X - F_j) + r(F_j) - 1$. In particular, when $J = \{i,j\}$, for distinct $i,j \in [k]$, this implies that $r(F_i \cup F_j) = r(F_i) + r(F_j) - 1$, so $\sqcap(F_i, F_j) = 1$.
\end{proof}

\begin{lemma} \label{even_fan_spike2}
    Let $M$ be a $3$-connected matroid. Let $P_1,P_2,\ldots,P_m$ be disjoint maximal fans of $E(M)$, each having even length at least two, for $m \ge 2$. For all $i \in [m]$, let $P_i = (p^i_1,p^i_2,\ldots,p^i_{|P_i|})$ such that either $|P_i| = 2$, or $\{p^i_1,p^i_2,p^i_3\}$ is a triad. Furthermore, for all distinct $i, j \in [m]$, suppose there is a $4$-element circuit~$C_{i,j}$ containing $\{p^i_1,p^j_1\}$ such that $|C_{i,j} \cap P_i| = |C_{i,j} \cap P_j| = 2$, and a $4$-element cocircuit $C^*_{i,j}$ containing $\{p^i_{|P_i|},p^j_{|P_j|}\}$ such that $|C^*_{i,j} \cap P_i| = |C^*_{i,j} \cap P_j| = 2$.
    Suppose $M$ has no detachable pairs and $|E(M)| \geq 9$.
    If $|E(M)| \leq |P_1 \cup P_2 \cup \cdots \cup P_m| + 2$, then either \begin{enumerate}
        \item $E(M) = P_1 \cup P_2 \cup \cdots \cup P_m$ and either
            \begin{enumerate}[label=\rm(\alph*)]
                \item $m \ge 3$ and $M$ is a non-degenerate even-fan-spike with partition $(P_1,P_2,\ldots,P_m)$, or
                \item $m=2$ and $M$ is a degenerate even-fan-spike with partition $(P_1,P_2)$, or
            \end{enumerate}
        \item $E(M) = P_1 \cup P_2 \cup \cdots \cup P_m \cup \{x,y\}$, for distinct $x,y \notin P_1 \cup P_2 \cup \cdots \cup P_m$, and $M$ is a non-degenerate even-fan-spike with partition $(P_1,P_2,\ldots,P_m,\{x,y\})$.
	\end{enumerate}
\end{lemma}

\begin{proof}
	First, assume that $E(M) = P_1 \cup P_2 \cup \cdots \cup P_m$.
    If $m \ge 3$, then, by repeated applications of \cref{even_fan_spike_separating}, $\Phi=(P_1,P_2,\ldots,P_m)$ is a spike-like anemone, and it follows that $M$ is a non-degenerate even-fan-spike with partition $\Phi$, satisfying (i)(a).
    So we may assume that $m = 2$. It remains to show that $M$ is a degenerate even-fan-spike with partition $\Phi$. Suppose $|P_1| = 2$. Since $\lambda(P_2 - \{p^2_1\}) = 2$, we also have that $\lambda(P_1 \cup \{p^2_1\}) = 2$. But now $P_1 \cup \{p^2_1\}$ is either a triangle or a triad, contradicting the maximality of $P_1$. Thus, $|P_1| \geq 4$ and, similarly, $|P_2| \geq 4$. Since $|E(M)| \geq 9$, one of $P_1$ and $P_2$ has length at least six, so, by \cref{even_fan_spike1}, $\{p^1_1,p^1_2,p^2_1,p^2_2\}$ is a circuit and $\{p^1_{|P_1|-1},p^1_{|P_1|},p^2_{|P_2|-1},p^2_{|P_2|}\}$ is a cocircuit, and thus $M$ is a degenerate even-fan-spike with partition $\Phi$, satisfying (i)(b).
	
    Now suppose that $E(M) = P_1 \cup P_2 \cup \cdots \cup P_m \cup \{x\}$ for some $x \notin P_1 \cup P_2 \cup \cdots \cup P_m$. \Cref{even_fan_spike_separating} implies that $\lambda(P_1 \cup P_2 \cup \cdots \cup P_{m-1}) = 2$, so $\lambda(P_m \cup \{x\}) = 2$. Since $\lambda(P_m) = 2$, either $x \in \cl(P_m)$ or $x \in \cl^*(P_m)$. This contradicts either \Cref{fan_closure} or its dual.
	
    The last case to consider is when $E(M) = P_1 \cup P_2 \cup \cdots \cup P_m \cup \{x,y\}$ for distinct $x,y \notin P_1 \cup P_2 \cup \cdots \cup P_m$. For all proper non-empty subsets $J$ of $[m]$, we have that $\lambda(\bigcup_{i \in [m]-J} P_i) = 2$ by \Cref{even_fan_spike_separating}, so $\lambda(\{x,y\} \cup \bigcup_{i \in J} P_i) = 2$. This shows that $\Phi=(P_1,P_2,\ldots,P_m,\{x,y\})$ is an anemone. Also, for all $i \in [m]$, we have that $x \notin \cl(P_i)$ and $x \notin \cl^*(P_i)$ by \Cref{fan_closure} and its dual. Since $\lambda(P_i \cup \{x,y\}) = 2$, this implies that $y \in \cl(P_i \cup \{x\}) \cap \cl^*(P_i \cup \{x\})$. Therefore, $\sqcap(P_i,\{x,y\}) = r(P_i) + 2 - (r(P_i) + 1) = 1$. Hence, $\Phi$ is a spike-like anemone, and it follows that $M$ is a non-degenerate even-fan-spike with partition $\Phi$, satisfying (ii). This completes the proof of the lemma.
\end{proof}

\begin{lemma} \label{even_fan_spike3}
    Let $M$ be a $3$-connected matroid with no detachable pairs such that $|E(M)| \geq 12$. Let $F_1$ and $F_2$ be disjoint maximal fans of $M$, each having even length, with $|F_1| \ge 4$ and $|F_2| \ge 2$.  Let $F_1 = (e_1,e_2,\ldots,e_{|F_1|})$ and $F_2 = (f_1,f_2,\ldots,f_{|F_2|})$ such that $\{e_1,e_2,e_3\}$ is a triad, and either $|F_2| = 2$ or $\{f_1,f_2,f_3\}$ is a triad. Furthermore, suppose $M$ has a $4$-element circuit~$C$ containing $\{e_1,f_1\}$ such that $|C \cap F_1| = 2$ and $|C \cap F_2| = 2$, and a $4$-element cocircuit $C^*$ containing $\{e_{|F_1|},f_{|F_2|}\}$ such that $|C^* \cap F_1| = 2$ and $|C^* \cap F_2| = 2$. Suppose $|E(M)| \ge |F_1 \cup F_2| +3$. If $e \notin F_1 \cup F_2$ and $e$ is contained in a triangle or a triad, then $e$ is contained in a $4$-element fan of $M$.
\end{lemma}

\begin{proof}
    Suppose there exists $e \notin F_1 \cup F_2$ such that $e$ is contained in a triad~$T^*$ and is not contained in a $4$-element fan. If $T^*$ is disjoint from $F_1$, then $M$ is a hinged triad-paddle, by \cref{even_fan_k3m}, which contradicts the existence of two disjoint maximal fans with even length. Thus, $T^* \cap F_1 \neq \emptyset$, and so, by \Cref{intersecting_fans}, $T^* \cap F_1 = \{e_1\}$. Orthogonality with the circuit $C$ implies that $T^* \cap F_2 \neq \emptyset$. If $|F_2| \geq 4$, then \Cref{intersecting_fans} implies that $T^* = \{e_1,f_1,e\}$. On the other hand, if $|F_2| = 2$, then, up to switching the labelling of $f_1$ and $f_2$, we may assume that $T^* = \{e_1,f_1,e\}$. Hence, $T^* = \{e_1,f_1,e\}$ and, in particular, $e \in \cl^*(F_1 \cup F_2)$.
    Since $|E(M)| \ge |F_1 \cup F_2| + 3$, \cref{even_fan_spike_separating} implies that $\lambda(F_1 \cup F_2) = 2$, and hence $\lambda(F_1 \cup F_2 \cup \{e\}) = 2$.

    Suppose $|E(M)| = |F_1 \cup F_2| + 3$. Since $\lambda(E(M) - (F_1 \cup F_2)) = 2$, the set $E(M) - (F_1 \cup F_2)$ is either a triangle or a triad, which is disjoint from $F_1$ and $F_2$. By orthogonality with the circuit $C$ and the cocircuit $C^*$, we have that $E(M) - (F_1 \cup F_2)$ is not contained in a $4$-element fan. But $E(M) - (F_1 \cup F_2)$ is disjoint from $F_1$, contradicting \Cref{even_fan_k3m} or its dual.

    Therefore, we may assume that $|E(M)| \geq |F_1 \cup F_2| + 4$. The matroid $M / e_1$ is $3$-connected and $e \in \cl^*(F_1 \cup F_2) = \cl^*((F_1 - \{e_1\}) \cup F_2)$. Thus, the dual of \cref{contractable_circuit} implies that $M$ has a $4$-element circuit $C'$ containing $\{e,e_1\}$, either $e_2$ or $e_3$, and an element $f$ with $f \notin F_1 \cup F_2 \cup \{e\}$. Suppose $f$ is contained in a triad~$T_2^*$. We show that $e_1 \in T_2^*$. If $T_2^*$ is not contained in a $4$-element fan, then \Cref{even_fan_k3m} implies that $T_2^*$ meets $F_1$. Thus, $e_1 \in T_2^*$, by \Cref{intersecting_fans}. On the other hand, if $T_2^*$ is contained in a $4$-element fan, then $e \notin T_2^*$, since $e$ is not contained in a $4$-element fan. Orthogonality with the circuit $C'$ implies that $e_1 \in T_2^*$. Hence, in either case, $e_1 \in T_2^*$. Now, orthogonality with $C$ implies that $T_2^*$ meets $F_2$. But this means that $f \in \cl^*(F_1 \cup F_2)$ and $f \in \cl(F_1 \cup F_2 \cup \{e\})$, so $\lambda(F_1 \cup F_2 \cup \{e,f\}) = 1$. This is a contradiction, since $|E(M)| \geq |F_1 \cup F_2 \cup \{e,f\}| + 2$, so $f$ is not contained in a triad.
	
    First, suppose $|F_2| \geq 4$. Then $M \backslash f_{|F_2|}$ is $3$-connected, so \cref{deletable_cocircuit} implies that $M$ has a $4$-element cocircuit containing $\{f,f_{|F_2|}\}$, either $f_{|F_2|-2}$ or $f_{|F_2|-1}$, and a second element of $C'$. But now $f \in \cl^*(F_1 \cup F_2 \cup \{e\})$, a contradiction.
    
    Thus, $|F_2| = 2$. Observe that $(e_1,\{e,f_1\},\{F_1-\{e_1\}\})$ is a contraction certificate. Since $M \backslash e_{|F_1|}$ is $3$-connected, \cref{deletable_cocircuit} implies that $M$ has a $4$-element cocircuit $D^*$ containing $\{f,e_{|F_1|}\}$, either $e_{|F_1|-2}$ or $e_{|F_1|-1}$, and an element $g \notin F_1 \cup F_2 \cup \{e,f\}$. Orthogonality with $C'$ implies that $|F_1| = 4$ and $D^* = \{e_i,e_{|F_1|},f,g\}$ such that $i \in \{2,3\}$ and $e_i \in C'$. Now, $g \in \cl^*(F_1 \cup F_2 \cup \{e,f\})$ and $|E(M)| \geq 12 = |F_1 \cup F_2 \cup \{e,f\}| + 4$. By \Cref{contractable_collection}, the matroid $M / g$ is not $3$-connected, and so the dual of \Cref{closure_deletable} implies that $g$ is contained in a triangle $T$. Since $g \notin \cl(F_1 \cup F_2 \cup \{e,f\})$, orthogonality implies that $T = \{f,g,h\}$, where $h \notin F_1 \cup F_2 \cup \{e,f,g\}$. If $T$ is not contained in a $4$-element fan, then \Cref{even_fan_k3m} implies that $M^*$ is a hinged triad-paddle, a contradiction. So there is a maximal fan $F_3$ of $M$ with at least four elements, that contains $T$. Since $f$ is not contained in a triad, $f$ is an end of $F_3$. Let $g^+$ be the other end. Note that $g^+ \notin F_1 \cup F_2$ by orthogonality, and $g^+ \neq e$ since $e$ is not contained in a $4$-element fan. Hence, $M / g^+$ is not $3$-connected by \Cref{contractable_collection}. But this implies that $F_3$ is odd, and $F_3$ is disjoint from $F_1$, so \Cref{augmented_fan_petal} implies that $M^*$ is a quasi-triad-paddle with an augmented fan petal, a contradiction. 
    
    Hence, if $e$ is contained in a triad, then $e$ is contained in a $4$-element fan. A dual argument shows that if $e$ is contained in a triangle, then $e$ is contained in a $4$-element fan, completing the proof of the lemma.
\end{proof}

\begin{lemma} \label{even_fan_spike4}
    Let $M$ be a $3$-connected matroid such that $|E(M)| \geq 13$. Let $F_1 = (e_1,e_2,e_3,e_4)$ and $F_2 = (f_1,f_2,f_3,f_4)$ be disjoint maximal fans of $M$ such that $\{e_1,e_2,e_3\}$ and $\{f_1,f_2,f_3\}$ are triads, $\{e_1,e_2,f_1,f_2\}$ is a circuit, and $\{e_2,e_4,f_2,f_4\}$ is a cocircuit. Then $M$ has a detachable pair.
\end{lemma}

\begin{proof}
	Suppose, to the contrary, that $M$ has no detachable pairs. First, assume there exists $e \notin F_1 \cup F_2$ such that $e$ is contained in a triangle or triad. Then \Cref{even_fan_spike3} implies that there is a $4$-element fan of $M$ that contains $e$. Let $F_3$ be a maximal fan containing $e$ with ordering $(g_1,g_2,\ldots,g_{|F_3|})$. By \Cref{intersecting_fans}, $F_3 \cap (F_1 \cup F_2) \subseteq \{g_1,g_{|F_3|}\}$. Hence, orthogonality with the circuit $\{e_1,e_2,f_1,f_2\}$ and the cocircuit $\{e_2,e_4,f_2,f_4\}$ implies that $F_3$ is disjoint from $F_1$ and $F_2$. Furthermore, by \Cref{augmented_fan_petal}, $F_3$ is not odd. Thus, without loss of generality, we may assume that $\{g_1,g_2,g_3\}$ is a triad and $\{g_{|F_3|-2},g_{|F_3|-1},g_{|F_3|}\}$ is a triangle.
    Note also that, by \cref{even_fan_spike_separating}, $\lambda(F_1 \cup F_2) = 2$.
	
    It follows from \cref{even_fan_spike1} that there is a $4$-element circuit~$C$ containing $\{e_1,g_1\}$, and, by orthogonality, $C$ also contains $e_2$ or $e_3$, and $g_2$ or $g_3$.
    Orthogonality with $\{e_2,e_4,f_2,f_4\}$ implies that $e_3 \in C$. Furthermore, if $|F_3| \geq 5$, then orthogonality implies that $g_2 \in C$, and if $|F_3| = 4$, then we may assume that $g_2 \in C$ up to the ordering of $F_3$. Thus, $C = \{e_1,e_3,g_1,g_2\}$. By \Cref{even_fan_spike1}, either $\{e_2,e_4,g_{|F_3|-1},g_{|F_3|}\}$ is a cocircuit, or $|F_3| = 4$ and $\{e_3,e_4,g_2,g_4\}$ is a cocircuit. The former case contradicts orthogonality with the circuit $\{e_1,e_2,f_1,f_2\}$, so the latter holds. Similarly, $M$ has a $4$-element circuit containing $\{f_1,g_1\}$, and, by orthogonality with $\{e_2,e_4,f_2,f_4\}$ and $\{e_3,e_4,g_2,g_4\}$, this circuit is $\{f_1,f_3,g_1,g_3\}$. But now $\lambda(F_1 \cup F_2 \cup F_3) \leq 1$, which implies $E(M) \leq |F_1 \cup F_2 \cup F_3| + 1$. \Cref{even_fan_spike2} implies that $E(M) = F_1 \cup F_2 \cup F_3$, so that $|E(M)| = 12$, a contradiction.
	
    Now we may assume, for all $x \notin F_1 \cup F_2$, that $x$ is not contained in a triangle or a triad. Let $f \notin F_1 \cup F_2$. Bixby's Lemma implies that either $M / f$ or $M \backslash f$ is $3$-connected. Up to duality, we may assume the former. Since $M$ has no detachable pairs, and $e_1 \in \cl^*(F_1-\{e_1\})$, the dual of \cref{contractable_circuit} implies that $M$ has a $4$-element circuit $C_1$ containing $\{e_1,f\}$, either $e_2$ or $e_3$, and an element $g \notin F_1$.  By orthogonality, $g \notin F_2-\{f_4\}$.  Moreover, if $g = f_4$, then $f \in \cl(F_1 \cup F_2)$, contradicting that $M/f$ is $3$-connected.
    By orthogonality with $\{e_2,e_4,f_2,f_4\}$, we have that $C_1 = \{e_1,e_3,f,g\}$. Similarly, $M$ has a $4$-element circuit $C_2 = \{f_1,f_3,f,g'\}$, for $g' \notin F_1 \cup F_2$.
	
    Suppose $g = g'$. Then circuit elimination implies that $M$ has a circuit contained in $\{e_1,e_3,f_1,f_3,f\}$. But $M / f$ is $3$-connected, and so $f \notin \cl(F_1 \cup F_2)$, which means that $\{e_1,e_3,f_1,f_3\}$ is a circuit of $M$. This implies that $\lambda(F_1 \cup F_2) \leq 1$, a contradiction. Thus, $g \neq g'$. Now, $g$ is not contained in a triad, so \Cref{circuit_so_deletable} implies that $M \backslash g$ is $3$-connected. Hence, as $M$ has no detachable pairs, \cref{deletable_cocircuit} implies that $M$ has a $4$-element cocircuit containing $\{f_4,g\}$, either $f_2$ or $f_3$, and an element that is not contained in $(F_1 - \{e_1\}) \cup F_2$. Moreover, if this element is $e_1$, then $g \in \cl^*(F_1 \cup F_2)$, contradicting that $M \backslash g$ is $3$-connected. Thus, by orthogonality, $M$ has a cocircuit $\{f_3,f_4,f,g\}$. Similarly, $M$ has a cocircuit $\{e_3,e_4,f,g'\}$. But now $\lambda(F_1 \cup F_2 \cup \{f,g,g'\}) \leq 1$, a contradiction since $|E(M)| \geq 13$. We conclude that $M$ has a detachable pair.
\end{proof}

\begin{lemma} \label{even_fan_spike}
    Let $M$ be a $3$-connected matroid with no detachable pairs such that $|E(M)| \geq 13$.
    Let $F_1$ and $F_2$ be disjoint maximal fans, each having even length, with $|F_1| \ge 4$ and $|F_2| \ge 2$. Let $F_1 = (e_1,e_2,\ldots,e_{|F_1|})$ and $F_2 = (f_1,f_2,\ldots,f_{|F_2|})$ such that $\{e_1,e_2,e_3\}$ is a triad, and
    either $|F_2| = 2$ or $\{f_1,f_2,f_3\}$ is a triad. If $\{e_1,e_2,f_1,f_2\}$ is a circuit and $\{e_{|F_1|-1},e_{|F_1|},f_{|F_2|-1},f_{|F_2|}\}$ is a cocircuit, then $M$ is an even-fan-spike.
\end{lemma}

\begin{proof}
	By the assumptions of the lemma, we may choose, for $m \geq 2$, disjoint subsets $P_1,P_2,\ldots,P_m$ of $M$ such that, for all $i \in [m]$, the set $P_i = (p^i_1,p^i_2,\ldots,p^i_{|P_i|})$ is a maximal fan with even length at least two such that either $|P_i| = 2$ or $\{p^i_1,p^i_2,p^i_3\}$ is a triad, and, for all $j \in [m] - \{i\}$, the set $C_{i,j} = \{p^i_1,p^i_2,p^j_1,p^j_2\}$ is a circuit, and the set $C^*_{i,j} = \{p^i_{|P_i|-1},p^i_{|P_i|},p^j_{|P_j|-1},p^j_{|P_j|}\}$ is a cocircuit. Let $P_1, P_2, \ldots, P_m$ be a maximal collection of such subsets with $P_1 = F_1$, so that $|P_1| \geq 4$.
	
    If $|E(M)| \leq |P_1 \cup P_2 \cup \cdots \cup P_m| + 2$, then the lemma follows from \Cref{even_fan_spike2}.  So we may assume that $|E(M)| \geq |P_1 \cup P_2 \cup \cdots \cup P_m| + 3$.
    First, suppose that there exists $e \notin P_1 \cup P_2 \cup \cdots \cup P_m$ such that $e$ is contained in a triangle or a triad. By \Cref{even_fan_spike3}, $e$ is contained in a $4$-element fan. Let $P'$ be a maximal fan containing $e$. By orthogonality with the circuits $C_{i,j}$ and the cocircuits $C^*_{i,j}$, the fan $P'$ is disjoint from $P_i$ for all $i \in [m]$. Furthermore, by \Cref{augmented_fan_petal}, $|P'|$ is not odd. By Lemmas \ref{even_fan_spike1} and \ref{even_fan_spike4}, there exists an ordering $(p'_1,p'_2,\ldots,p'_{|P'|})$ of $P'$ such that $\{p'_1,p'_2,p'_3\}$ is a triad and $\{p'_1,p'_2,p^1_1,p^1_2\}$ is a circuit and $\{p'_{|P'|-1},p'_{|P'|},p^1_{|P_1|-1},p^1_{|P_1|}\}$ is a cocircuit. For all $i \in [m]$, circuit elimination with $C_{1,i}$ implies that $\{p'_1,p'_2,p^i_1,p^i_2\}$ is a circuit, and cocircuit elimination with $C^*_{1,i}$ implies that $\{p'_{|P'|-1},p'_{|P'|},p^i_{|P_i|-1},p^i_{|P_i|}\}$ is a cocircuit. But choosing $P_{m+1} = P'$ contradicts the maximality of the collection $P_1, P_2, \ldots, P_m$.
	
    Now we may assume that every element of $E(M) - (P_1 \cup P_2 \cup \cdots \cup P_m)$ is not contained in a triangle or a triad. Let $e$ be such an element. By Bixby's Lemma, either $M / e$ or $M \backslash e$ is $3$-connected. Without loss of generality, we may assume the former. Since $M$ has no detachable pairs, the dual of \cref{contractable_circuit} implies that $M$ has a $4$-element circuit~$C$ containing $\{e,p^1_1\}$, either $p^1_2$ or $p^1_3$, and an element $e' \notin P_1$.  Suppose that $e' \in P_i$ for some $i \in \{2,3,\dotsc,m\}$.  Then $e \in \cl(P_1 \cup P_i)$, contradicting that $M/e$ is $3$-connected. So $e' \notin P_1 \cup P_2 \cup \dotsm \cup P_m$.
   Furthermore, $p^1_3 \notin C$ by orthogonality with $\{p^1_3,p^1_4,p^1_5\}$ if $|P_1| \geq 5$, or by orthogonality with $\{p^1_3,p^1_4,p^2_{|P_2|-1},p^2_{|P_2|}\}$ if $|P_1| = 4$. Thus, $\{e,e',p^1_1,p^1_2\}$ is a circuit.
	
   Since $e'$ is not contained in a triad, \Cref{circuit_so_deletable} implies that $M \backslash e'$ is $3$-connected. Therefore, by \cref{deletable_cocircuit}, $M$ has a $4$-element cocircuit~$C^*$ containing $\{e',p^1_{|P_1|}\}$, either $p^1_{|P_1|-2}$ or $p^1_{|P_1|-1}$, and an element that is not contained in $P_1$. As before, this element is also not contained in $P_i$ for $i \in \{2,3,\dotsc,m\}$, for otherwise $e' \in \cl^*(P_1 \cup P_i)$. Orthogonality with $\{p^1_{|P_1|-4},p^1_{|P_1|-3},p^1_{|P_1|-2}\}$ if $|P_1| \geq 5$, or with $\{p^1_1,p^1_2,p^2_1,p^2_2\}$ if $|P_1| = 4$, implies that $p^1_{|P_1|-1} \in C^*$. Orthogonality with $C$ implies that $e \in C^*$, so $C^* = \{e,e',p^1_{|P_1|-1},p^1_{|P_1|}\}$. Now, $\{e,e'\}$ is a maximal $2$-element fan and, for all $i \in [m]$, circuit and cocircuit elimination with $C_{1,i}$ and $C^*_{1,i}$ implies that $\{e,e',p^i_1,p^i_2\}$ is a circuit and $\{e,e',p^i_{|P_i|-1},p^i_{|P_i|}\}$ is a cocircuit. Choosing $P_{m+1} = \{e,e'\}$ contradicts the maximality of the collection $P_1, P_2, \ldots, P_m$.\end{proof}

\subsection*{Putting it together}
We now combine the lemmas in this section to prove \Cref{disjoint_fans_with_like_ends}.
Recall that $F_1 = (e_1,e_2,\ldots,e_{|F_1|})$ and $F_2 = (f_1,f_2,\ldots,f_{|F_2|})$ are disjoint maximal fans of $M$ such that $|F_1| \geq 4$ and $|F_2| \geq 3$.

\begin{proof}[Proof of \Cref{disjoint_fans_with_like_ends}.]
    Suppose that $\{e_1,e_2,e_3\}$ and $\{f_1,f_2,f_3\}$ are both triads, and $M$ does not have a detachable pair. If $|F_2| = 3$, then \Cref{even_fan_k3m} implies that $M$ is a hinged triad-paddle, so (iii) holds. Otherwise, $|F_2| \geq 4$. If either $F_1$ or $F_2$ is odd, then \Cref{augmented_fan_petal} implies that $M$ is a quasi-triad-paddle with an augmented-fan petal, so (iv) holds. Finally, if both $|F_1|$ and $|F_2|$ are even, then \cref{even_fan_spike1,even_fan_spike4,even_fan_spike} combine to show that $M$ is an even-fan-spike, so (ii) holds, which completes the proof.
\end{proof}

\section{Intersecting fans} \label{fans_intersecting}
For the remainder of the proof of \Cref{detachable_main}, we may assume that $M$ does not have disjoint maximal fans $F_1 = (e_1,e_2,\ldots,e_{|F_1|})$ and $F_2 = (f_1,f_2,\ldots,f_{|F_2|})$ such that $|F_1| \geq 4$, and $|F_2| \geq 3$, and $\{e_1,e_2,e_3\}$ and $\{f_1,f_2,f_3\}$ are both triads.
Similarly, if $M$ has disjoint maximal fans $F_1$ and $F_2$ satisfying these conditions except that $\{e_1,e_2,e_3\}$ and $\{f_1,f_2,f_3\}$ are both triangles, then $M^*$ is one of the matroids described in \Cref{disjoint_fans_with_like_ends}, so we may assume that this is not the case either. As a shorthand for these assumptions, we shall say $M$ has no \emph{disjoint maximal fans with like ends}. This section concerns $3$-connected matroids that have two fans $F_1$ and $F_2$ with non-empty intersection. In particular, we prove the following theorem.

\begin{theorem} \label{intersecting_fans_detachable}
    Let $M$ be a $3$-connected matroid such that $|E(M)| \geq 13$, and suppose that $M$ has no disjoint maximal fans with like ends. Let $F_1$ and $F_2$ be distinct maximal fans of $M$ such that $|F_1| \geq 4$ and $|F_2| \geq 3$, and $F_1 \cap F_2 \neq \emptyset$. Then one of the following holds:
	\begin{enumerate}
		\item $M$ has a detachable pair,
        \item $M$ is an even-fan-spike with partition $(F_1, \{f,x\}, \{f',x'\}$), where $|F_2|=3$, $F_2-F_1 = \{f,f'\}$ and $x,x' \in E(M)-(F_1 \cup F_2)$,
		\item $M$ is an even-fan-spike with tip and cotip,
		\item $M$ is an accordion, or
		\item $M$ or $M^*$ is an even-fan-paddle.
	\end{enumerate} 
\end{theorem}

\subsection*{\texorpdfstring{$\mathbf{F_1}$}{F1} and \texorpdfstring{$\mathbf{F_2}$}{F2} are odd}
First, we consider the case where both $F_1$ and $F_2$ are odd. By \Cref{odd_fan}, we only need to consider when $F_1$ and $F_2$ have length three or five.
We handle the case where $\{|F_1|,|F_2|\}=\{3,5\}$ in \cref{odd_fan_no_triangle}, and the case where $|F_1|=|F_2|=5$ in \cref{odd_fans_intersecting}.

\begin{lemma} \label{odd_fan_triangle}
	Let $M$ be a $3$-connected matroid such that $|E(M)| \geq 13$. Let $F = (e_1,e_2,e_3,e_4,e_5)$ be a maximal fan of $M$, and suppose there exists $e \in E(M)-F$ such that $\{e_1,e_5,e\}$ is a triangle. Then $M$ has a detachable pair.
\end{lemma}

\begin{proof}
    Suppose, to the contrary, that $M$ has no detachable pairs. Since $e_1$ and $e_5$ are contained in the triangle $\{e_1,e_5,e\}$, it follows by \Cref{fan_ends} that $\{e_1,e_2,e_3\}$ and $\{e_3,e_4,e_5\}$ are triangles. Therefore, $e_1 \in \cl(\{e_2,e_3,e_4\})$ and $e_1 \in \cl(\{e_5,e\})$. Furthermore, $e_1$ is not contained in a triad. Hence, $(e_1,\{e_2,e_3,e_4\},\{\{e_5,e\}\})$ is a deletion certificate, and $\lambda(F \cup \{e\}) = 2$.  We complete the proof of the lemma by finding an element $x \notin F \cup \{e\}$ such that $M \backslash x$ is $3$-connected, a contradiction to \Cref{deletable_collection}.
	
    Now, $\{e_1,e_5,e\} \subseteq \cl(\{e_2,e_3,e_4\})$. Furthermore, each of $e_1$ and $e_5$ is not contained in a triad, and $e$ is also not contained in a triad, since orthogonality with $\{e_1,e_5,e\}$ implies that this triad contains either $e_1$ or $e_5$. Now, $|E(M)| \geq |\{e_2,e_3,e_4\}| + 7$, so \Cref{three_deletable} implies that $M$ has distinct elements $f,f',f'' \notin F \cup \{e\}$ such that $\{f,f',f''\} \subseteq \cl^*(F \cup \{e\})$ and none of $f$, $f'$, and $f''$ are contained in a triangle. Additionally, $|E(M)| \geq 13 = |F \cup \{e\}| + 7$ and, for all $y \in F \cup \{e\}$, we have that $y \in \cl((F \cup \{e\}) - \{y\})$. Hence, by the dual of \Cref{three_deletable}, there exist distinct elements $g, g', g''\not\in F\cup \{e, f, f', f''\}$ such that $\{g, g', g''\}\subseteq \cl(F\cup \{e, f, f', f''\})$ and none of $g, g', g''$ are contained in a triad. In particular, $M \backslash g$ is $3$-connected by \Cref{closure_deletable}, a contradiction.
\end{proof}

A consequence of \cref{odd_fan_triangle} is the following corollary, which implies that if a $3$-connected matroid has at least thirteen elements and no detachable pairs, then it has no $M(K_4)$-separators. 

\begin{corollary} \label{no_mk4}
	Let $M$ be a $3$-connected matroid such that $|E(M)| \geq 13$. If $M$ has an $M(K_4)$-separator, then $M$ has a detachable pair.
\end{corollary}

\begin{lemma} \label{odd_fan_no_triangle}
	Let $M$ be a $3$-connected matroid such that $|E(M)| \geq 13$. Let $F_1$ be a maximal fan of $M$ with ordering $(e_1,e_2,e_3,e_4,e_5)$ such that $\{e_1,e_2,e_3\}$ is a triangle. If $M$ has a triangle $T$ that is not contained in a $4$-element fan, then $M$ has a detachable pair.
\end{lemma}

\begin{proof}
	Suppose $M$ has no detachable pairs, and consider the dual of \Cref{disjoint_triads}. Since $|F_1| = 5$, we have that $F_1$ and $T$ do not satisfy \Cref{disjoint_triads}(i), (iii), or (iv). Hence, $e_1 \in T$. Furthermore, by reversing the ordering of $F_1$, \Cref{disjoint_triads} implies that $e_5 \in T$. Thus, $T = \{e_1,e_5,e\}$, for some $e \notin F$, contradicting \Cref{odd_fan_triangle}.
\end{proof}

\begin{lemma} \label{odd_fans_intersecting}
	Let $M$ be a $3$-connected matroid such that $|E(M)| \geq 13$. Let $F_1 = (e_1,e_2,e_3,e_4,e_5)$ and $F_2 = (f_1,f_2,f_3,f_4,f_5)$ be distinct maximal fans of $M$ such that $e_1 = f_1$. Then $M$ has a detachable pair.
\end{lemma}

\begin{proof}
    Up to duality, we may assume that $\{e_1,e_2,e_3\}$ is a triangle. Since $e_1 = f_1$, \Cref{fan_ends} implies that $\{f_1,f_2,f_3\}$ is also a triangle. Now assume, towards a contradiction, that $M$ does not have a detachable pair. \Cref{no_mk4} implies that $F_1 \cup F_2$ is not an $M(K_4)$-separator, so either $F_1 \cap F_2 = \{e_1\} = \{f_1\}$ or $F_ 1 \cap F_2 = \{e_1,e_5\} = \{f_1,f_5\}$. By \Cref{odd_fan}, there exist $z,z' \in E(M)$ such that $\{e_1,e_3,e_5,z\}$ and $\{f_1,f_3,f_5,z'\}$ are cocircuits. By orthogonality with $\{f_1,f_2,f_3\}$, we have that $z \in \{f_2,f_3\}$, and by orthogonality with $\{e_1,e_2,e_3\}$, we have that $z' \in \{e_2,e_3\}$. 
	
    First, suppose $F_1 \cap F_2 = \{e_1\}$. Now $\lambda(F_1 \cup \{f_2,f_3,f_4\}) = 2$. But $f_5 \in \cl(F_1 \cup \{f_2,f_3,f_4\})$ and $f_5 \in \cl^*(F_1 \cup \{f_2,f_3,f_4\})$. Thus, $\lambda(F_1 \cup F_2) \leq 1$, a contradiction as $|E(M)|\ge 13$. Otherwise, if $F_1 \cap F_2 = \{e_1,e_5\}$, then $\lambda(F_1 \cup \{f_2,f_3\}) = 2$ and $f_4 \in \cl(F_1 \cup \{f_2,f_3\}) \cap \cl^*(F_1 \cup \{f_2,f_3\})$. Again, $\lambda(F_1 \cup F_2) \leq 1$, a contradiction. This completes the proof of the lemma.
\end{proof}

\subsection*{\texorpdfstring{$\mathbf{F_1}$}{F1} and \texorpdfstring{$\mathbf{F_2}$}{F2} are even and intersect at both ends}
Now, we may assume that at least one of $F_1$ and $F_2$ is even. In the next two subsections, we consider when $F_1$ and $F_2$ are both even. We first consider the case when $F_1$ and $F_2$ intersect at both ends.

\begin{lemma} \label{tip_cotip1}
	Let $M$ be a $3$-connected matroid with no detachable pairs. Let $F_1 = (e_1,e_2,\ldots,e_{|F_1|})$ and $F_2 = (f_1,f_2,\ldots,f_{|F_2|})$ be distinct maximal fans of $M$ with even length at least four. If $e_1 = f_1$ and $e_{|F_1|} = f_{|F_2|}$, then every element of $M$ is contained in a maximal fan of length at least four with ends $e_1$ and $e_{|F_1|}$.
\end{lemma}

\begin{proof}
	Without loss of generality, assume that $\{e_1,e_2,e_3\}$ and $\{f_1,f_2,f_3\}$ are triangles, and $\{e_{|F_1|-2},e_{|F_1|-1},e_{|F_1|}\}$ and $\{f_{|F_2|-2},f_{|F_2|-1},f_{|F_2|}\}$ are triads. Clearly, the result holds if $E(M) = F_1 \cup F_2$. 
	
    Suppose that $E(M) = F_1 \cup F_2 \cup \{x\}$.
    By circuit elimination and orthogonality, $\{e_2,e_3,f_2,f_3\}$ is a circuit.  Similarly, $\{e_{|F_1|-2}, e_{|F_1|-1}, f_{|F_2|-2}, f_{|F_2|-1}\}$ is a cocircuit.  It follows that $\lambda((F_1 \cup F_2) - \{e_1,e_{|F_1|}\}) = 2$.  Thus, $\lambda(\{e_1,e_{|F_1|},x\}) = 2$, so $\{e_1,e_{|F_1|},x\}$ is either a triangle or a triad. This is a contradiction to orthogonality, since $\{e_1,e_2,e_3\}$ is a triangle and $\{e_{|F_1|-2},e_{|F_1|-1},e_{|F_1|}\}$ is a triad. 
	
    Next, suppose that $E(M) = F_1 \cup F_2 \cup \{x,y\}$. Since $\lambda((F_1 \cup F_2) - \{e_1\}) = 2$, we have that $\lambda(\{e_1,x,y\}) = 2$. Thus, $\{e_1,x,y\}$ is a triangle. Similarly, $\lambda(\{e_{|F_1|},x,y\}) = 2$, so $\{e_{|F_1|},x,y\}$ is a triad. Thus, $M$ has a maximal fan with ordering $(e_1,x,y,e_{|F_1|})$ and the lemma holds.
	
    Finally, suppose that $|E(M)| \geq |F_1 \cup F_2| + 3$. First note that $\lambda(F_1 \cup F_2) = 2$, $(e_1,F_1-\{e_1\},\{F_2-\{e_1\}\})$ is a deletion certificate, and $(e_{|F_1|},F_1-\{e_{|F_1|}\},\{F_2-\{e_{|F_1|}\}\})$ is a contraction certificate. Let $e \notin F_1 \cup F_2$. \Cref{no_other_elements} implies that $e$ is contained in a maximal fan $F_3$ of length at least four with ends in $F_1 \cup F_2$. \Cref{fan_ends} implies that the ends of $F_3$ are $e_1$ and $e_{|F_1|}$, completing the proof of the lemma.
\end{proof}

\begin{lemma} \label{tip_cotip}
	Let $M$ be a $3$-connected matroid with no detachable pairs. Let $F_1 = (e_1,e_2,\ldots,e_{|F_1|})$ and $F_2 = (f_1,f_2,\ldots,f_{|F_2|})$ be distinct maximal fans of $M$ with even length at least four. If $e_1 = f_1$ and $e_{|F_1|} = f_{|F_2|}$, then $M$ is an even-fan-spike with tip and cotip.
\end{lemma}

\begin{proof}
    Assume, without loss of generality, that $\{e_1,e_2,e_3\}$ and $\{f_1,f_2,f_3\}$ are triangles, and $\{e_{|F_1|-2},e_{|F_1|-1},e_{|F_1|}\}$ and $\{f_{|F_2|-2},f_{|F_2|-1},f_{|F_2|}\}$ are triads. If $E(M) = F_1 \cup F_2$, then $M$ is a degenerate even-fan-spike with tip~$e_1$ and cotip~$e_{|F_1|}$. Otherwise, choose a maximal collection of disjoint subsets $P_1,P_2,\ldots,P_m$ of $E(M)$ with $P_1=F_1$ and $P_2=F_2-\{e_1,e_{|F_1|}\}$ such that
	\begin{enumerate}
        \item for all $i \in [m]$, the set $P_i \cup \{e_1,e_{|F_1|}\}$ is an even fan with ends $e_1$ and $e_{|F_1|}$,
        \item for each non-empty subset $J$ of $[m]$, we have that $\lambda(\bigcup_{i \in J} P_i) \le 2$, and
		\item for all distinct $i,j \in [m]$, we have that $\sqcap(P_i,P_j) = 1$.
	\end{enumerate}

    Suppose there exists an element $e \in E(M) - (P_1 \cup P_2 \cup \cdots \cup P_m)$. By \Cref{tip_cotip1}, $e$ is contained in a maximal fan $F_3$ of length at least four, with ends $e_1$ and $e_{|F_1|}$. Let $P' = F_3 - \{e_1,e_{|F_1|}\}$. Then $\lambda(P')=2$. Let $J$ be a non-empty subset of $[m]$, and let $X = \bigcup_{i \in J} P_i$. 
    By submodularity, $r(X \cup P') \leq r(X) + r(P' \cup \{e_1\}) - 1 = r(X) + r(P') - 1$. 
    Similarly, $r^*(X \cup P') \leq r^*(X) + r^*(P') - 1$. It follows that $\lambda(X \cup P') \leq \lambda(X) +\lambda(P')-2 = 2$. Furthermore, for all $i \in [m]$, we have that $r(P_i \cup P') = r(P_i) + r(P') - 1$, so $\sqcap(P_i,P') = 1$. Thus the disjoint subsets $P_1, P_2, \ldots, P_m, P'$ satisfy (i)--(iii), contradicting maximality. We deduce that $E(M) = P_1 \cup P_2 \cup \cdots \cup P_m$.
    Now $(P_1,P_2,\ldots,P_m)$ is a spike-like anemone, and so $M$ is an even-fan-spike with tip and cotip, thereby completing the proof of the lemma.
\end{proof}

\subsection*{\texorpdfstring{$\mathbf{F_1}$}{F1} and \texorpdfstring{$\mathbf{F_2}$}{F2} are even and intersect at one end}
Next, we consider the case where $F_1$ and $F_2$ are both even, and intersect in exactly one element.

\begin{lemma} \label{even_fan_paddle1}
	Let $M$ be a $3$-connected matroid with no detachable pairs.
	Let $F_1 = (e_1,e_2,\ldots,e_{|F_1|})$ and $F_2 = (f_1,f_2,\ldots,f_{|F_2|})$ be distinct maximal fans of $M$ with even length at least four such that $\{e_1,e_2,e_3\}$ and $\{f_1,f_2,f_3\}$ are triangles. If $e_1 = f_1$ and $e_{|F_1|} \neq f_{|F_2|}$, and $|E(M)| \leq |F_1 \cup F_2| + 2$, then $M$ is a degenerate even-fan-paddle.
\end{lemma}

\begin{proof}
	The dual of \cref{contractable_circuit} and orthogonality implies that there is a $4$-element circuit $C$ of $M$ containing $\{e_{|F_1|},f_{|F_2|}\}$, and one of $\{e_{|F_1|-2},e_{|F_1|-1}\}$, and one of $\{f_{|F_2|-2},f_{|F_2|-1}\}$.
    By orthogonality, we may assume (up to swapping $e_2$ and $e_3$ when $|F_1|=4$, and $f_2$ and $f_3$ when $|F_2|=4$) that $C=\{e_{|F_1|-1},e_{|F_1|},f_{|F_2| - 1},f_{|F_2|}\}$.

	First, assume that $E(M) = F_1 \cup F_2$. Since $M$ is $3$-connected, the set $E(M) - \{e_{|F_1|},f_{|F_2|}\}$ is spanning, so \[e_{|F_1|} \in \cl((F_1 \cup F_2) - \{e_{|F_1|},f_{|F_2|}\}).\] Now, $f_{|F_2|-1} \in \cl(F_2 - \{f_{|F_2|-1},f_{|F_2|}\})$, so we have that $e_{|F_1|} \in \cl((F_1 \cup F_2) - \{e_{|F_1|},f_{|F_2|-1},f_{|F_2|}\})$. Orthogonality with the triad $\{f_{|F_2|-2},f_{|F_2|-1},f_{|F_2|}\}$ implies that \[e_{|F_1|} \in \cl((F_1 \cup F_2) - \{e_{|F_1|},f_{|F_2|-2},f_{|F_2|-1},f_{|F_2|}\}).\] Continuing in this way, we eventually see that $e_{|F_1|} \in \cl(F_1 - \{e_{|F_1|}\})$. But this means that $\lambda(F_1) = 1$, a contradiction. Thus $E(M)\neq F_1\cup F_2$.
	
    Next, assume that $E(M) = F_1 \cup F_2 \cup \{x\}$ with $x \notin F_1 \cup F_2$. Since $\lambda(F_1-\{e_1\}) = 2$, we also have that $\lambda(F_2 \cup \{x\}) = 2$. Thus, either $x \in \cl(F_2)$ or $x \in \cl^*(F_2)$. Due to the circuit~$C$, \cref{fan_closure} implies that $x \notin \cl^*(F_2)$, so $x \in \cl(F_2)$. Similarly, $x \in \cl(F_1)$.  Moreover, by submodularity, $r(\{e_1,x,e_{|F_1|-1},e_{|F_1|}\}) \le r(F_1 \cup \{x\}) + r(F_2 \cup \{x,e_{|F_1|-1},e_{|F_1|}\}) -r(M) \le 3$, and it follows that $\{e_1,x,e_{|F_1|-1},e_{|F_1|}\}$ is a circuit.  Similarly, $\{e_1,x,f_{|F_2|-1},f_{|F_2|}\}$ is a circuit.
    Hence $M$ is a degenerate even-fan-paddle with partition $(F_1-\{e_1\}, F_2-\{e_1\}, \{e_1,x\})$.
	
    Finally, assume that $E(M) = F_1 \cup F_2 \cup \{x,y\}$.
    Due to the circuit~$C$, we have $\lambda(F_1 \cup F_2) = 2$, so $\lambda((F_1 \cup F_2) - \{e_1\}) = 2$ and $\lambda(\{e_1,x,y\}) = 2$. Thus $\{e_1,x,y\}$ is a triangle. If $\{x,y\}$ is contained in a triad, then this triad contains either $e_{|F_1|}$ or $f_{|F_2|}$, which contradicts orthogonality with the circuit~$C$. Hence, $\{x,y\}$ is not contained in a triad, so $\{e_1,x,y\}$ is not contained in a $4$-element fan. By Tutte's Triangle Lemma, either $M \backslash x$ or $M \backslash y$ is $3$-connected. Without loss of generality, assume the former. \Cref{deletable_cocircuit} implies that $M$ has a $4$-element cocircuit~$C^*$ containing $\{e_1,x\}$, either $e_2$ or $e_3$, and either $f_2$ or $f_3$. If $|F_1| > 4$, then orthogonality implies that $e_2 \in C^*$, and if $|F_1| = 4$, then we may assume $e_2 \in C^*$ up to the ordering of $F_1$. Similarly, we may assume $f_2 \in C^*$, so that $C^* = \{e_2,e_1,f_2,x\}$. Now, $M \backslash x$ has a fan $(e_{|F_1|},e_{|F_1|-1},\ldots,e_2,e_1,f_2,f_3,\ldots,f_{|F_2|})$. Since $|E(M \backslash x)| = |F_1 \cup F_2| + 1$, \Cref{not_just_fan} implies that $M \backslash x$ is a wheel or a whirl. But $e_{|F_1|}$ is not contained in a triangle of $M$, so it is also not contained in a triangle of $M \backslash x$. This last contradiction completes the proof of the lemma.
\end{proof}

\begin{lemma} \label{even_fan_paddle2}
	Let $M$ be a $3$-connected matroid with no detachable pairs and no disjoint maximal fans with like ends. Let $F_1 = (e_1,e_2,\ldots,e_{|F_1|})$ and $F_2 = (f_1,f_2,\ldots,f_{|F_2|})$ be distinct maximal fans of $M$ with even length at least four such that $\{e_1,e_2,e_3\}$ and $\{f_1,f_2,f_3\}$ are triangles. Suppose $e_1 = f_1$ and $e_{|F_1|} \neq f_{|F_2|}$, and $|E(M)| \geq |F_1 \cup F_2| + 3$. Then, for all $x \notin F_1 \cup F_2$, the element $x$ is contained in a maximal fan of even length at least four with ends $e_1$ and $x^+$ such that $x^+ \notin F_1 \cup F_2$.
\end{lemma}

\begin{proof}
    By the dual of \Cref{contractable_circuit}, $M$ has a $4$-element circuit~$C$ containing $\{e_{|F_1|},f_{|F_2|}\}$. By orthogonality, we may assume $C=\{e_{|F_1|-1},e_{|F_1|},f_{|F_2|-1},f_{|F_2|}\}$. Therefore $\lambda(F_1 \cup F_2) = 2$. Also $(e_1,F_1-\{e_1\},\{F_2-\{e_1\}\})$ is a deletion certificate. Furthermore, $e_{|F_1|} \in \cl^*(F_1-\{e_{|F_1|}\})$ and, for each $i \in [|F_1|]$, we have that $e_i \in \cl((F_1 \cup F_2) - \{e_i\})$. Hence, by \Cref{deletable_collection_contractable_el}, every element of $E(M) - (F_1 \cup F_2)$ is contained in a triad.
	
    Suppose $M$ has a maximal fan $F_3 = (g_1,g_2,\ldots,g_{|F_3|})$, distinct from $F_1$ and $F_2$, such that $|F_3| \geq 4$. Since $M$ has no disjoint maximal fans with like ends, and $F_1$ and $F_2$ are even, we have that $F_1 \cap F_3 \neq \emptyset$ and $F_2 \cap F_3 \neq \emptyset$. Furthermore, orthogonality with the circuit~$C$ implies that $e_{|F_1|} \notin F_3$ and $f_{|F_2|} \notin F_3$. Therefore, $e_1 \in F_3$, and so, without loss of generality, $e_1 = g_1$ and $\{g_1,g_2,g_3\}$ is a triangle. Furthermore, $g_{|F_3|} \notin F_1 \cup F_2$. This implies that $g_{|F_3|}$ is contained in a triad, so $\{g_{|F_3|-2},g_{|F_3|-1},g_{|F_3|}\}$ is a triad. Hence, $F_3$ has even length.
	
    Let $e \notin F_1 \cup F_2$. To complete the proof, it remains to show that $e$ is contained in a $4$-element fan. Suppose that $e$ is contained in a triad $T^*$ that is not contained in a $4$-element fan. Since $M$ has no disjoint maximal fans with like ends, we have that $F_1 \cap T^* \neq \emptyset$ and $F_2 \cap T^* \neq \emptyset$. Hence, $T^* = \{e,e_{|F_1|},f_{|F_2|}\}$. Now, let $f \notin F_1 \cup F_2 \cup \{e\}$. The element $f$ is contained in a triad $T_2^*$. If $T_2^*$ is not contained in a $4$-element fan, then $T_2^* = \{f,e_{|F_1|},f_{|F_2|}\}$. But this means that $r^*(\{e,f,e_{|F_1|},f_{|F_2|}\}) = 2$, which, by the dual of \cref{no_segment}, contradicts that $M$ has no detachable pairs. So there is a maximal fan $F$ with length at least four containing $T_2^*$. By the previous paragraph, $F$ has even length and ends $e_1$ and $f^+$, say, with $f^+ \notin F_1 \cup F_2$. Furthermore, $e$ is not contained in a $4$-element fan, so $e \notin F$. But now $F \cap T^* = \emptyset$, and $M$ has a pair of disjoint maximal fans with like ends. This contradiction completes the proof of the lemma.
\end{proof}

\begin{lemma} \label{even_fan_paddle}
	Let $M$ be a $3$-connected matroid with no detachable pairs and no disjoint maximal fans with like ends. Let $F_1 = (e_1,e_2,\ldots,e_{|F_1|})$ and $F_2 = (f_1,f_2,\ldots,f_{|F_2|})$ be distinct maximal fans of $M$ with even length at least four such that $\{e_1,e_2,e_3\}$ and $\{f_1,f_2,f_3\}$ are triangles. Suppose $e_1 = f_1$ and $e_{|F_1|} \neq f_{|F_2|}$. Then $M$ is an even-fan-paddle.
\end{lemma}

\begin{proof}
    If $|E(M)| \leq |F_1 \cup F_2| + 2$, then $M$ is a degenerate even-fan-paddle by \Cref{even_fan_paddle1}. So we may assume that $|E(M)| \geq |F_1 \cup F_2| + 3$. By the dual of \Cref{contractable_circuit}, $M$ has a $4$-element circuit containing $\{e_{|F_1|},f_{|F_2|}\}$.  It follows that $\lambda(F_1 \cup F_2)=2$ and $\sqcap(F_1,F_2)=2$. Thus, we may choose a maximal collection of disjoint subsets $P_1,P_2,\ldots,P_m$ of $E(M)$ with $P_1=F_1$ and $m \geq 2$ such that
	\begin{enumerate}
		\item for all $i \in [m]$, the set $P_i \cup \{e_1\}$ is a maximal fan with even length at least four and ordering $(p^i_1,p^i_2,\ldots,p^i_{|P_i|},e_1)$,
        \item for each non-empty subset $J$ of $[m]$, we have $\lambda(\bigcup_{i \in J} P_i) \le 2$, and
		\item for all distinct $i,j \in [m]$, we have $\sqcap(P_i,P_j) = 2$.
	\end{enumerate}
    Furthermore, for distinct $i,j \in [m]$, the dual of \cref{contractable_circuit} implies that $M$ has a circuit $C_{i,j}$ containing $\{p^i_1,p^j_1\}$, either $p^i_2$ or $p^i_3$, and either $p^j_2$ or $p^j_3$.
	
    Towards a contradiction, suppose that there exists an element $e\in E(M)-(P_1\cup P_2\cup \cdots \cup P_m)$. By \Cref{even_fan_paddle2}, the element $e$ is contained in a set~$P'$ such that $P' \cup \{e_1\}$ is a maximal fan with even length at least four and ordering $(p'_1,p'_2,\ldots,p'_{|P'|},e_1)$. Furthermore, by the dual of \cref{contractable_circuit}, for each $i \in [m]$ there is a circuit containing $\{p'_1,p^i_1\}$. Let $I$ be a non-empty subset of $[m]$, and let $X = \bigcup_{i \in I} P_i$. Now, $p'_1 \in \cl((X \cup P') - \{p'_1\})$ and $p'_1 \notin \cl(P' - \{p'_1\})$, and $p'_{|P'|} \in \cl((X \cup P') - \{p'_{|P'|}\})$ and $p'_{|P'|} \notin \cl(P' - \{p'_{|P'|}\})$. Thus, $r(X \cup P') \leq r(X) + r(P') - 2$. Since $r^*(X\cup P')\le r^*(X)+r^*(P')$, we deduce that $\lambda(X \cup P') \le 2$. In particular, when $X = P_i$ for $i \in [m]$, we have $\lambda(P_i \cup P')=2$, implying $r(P_i \cup P') = r(P_i) + r(P') - 2$, so $\sqcap(P_i,P') = 2$. Thus the disjoint subsets $P_1, P_2, \ldots, P_m, P'$ satisfy (i)--(iii), contradicting maximality. Therefore $E(M)=P_1\cup P_2\cup \cdots\cup P_m$, implying $(P_1, P_2, \dots, P_m)$ is a paddle.
	
	Assume $|P_i \cup \{e_{1}\}| = 4$, for all $i \in [m]$. 
    Then both $(p_1^i,p_2^i,p_3^i,e_1)$ and $(p_1^i,p_3^i,p_2^i,e_1)$ are fan orderings of $P_i \cup \{e_1\}$, for each $i \in [m]$, so the existence of the circuit $C_{i,j}$, for all distinct $i,j \in [m]$, implies that $M$ is an even-fan-paddle.
	
    Without loss of generality, we may now assume that $|P_1 \cup \{e_{1}\}| > 4$. If $|P_1 \cup \{e_1\}|=5$, then we let $p^1_5=e_1$ in what follows. By orthogonality with $\{p^1_3,p^1_4,p^1_5\}$, the circuit $C_{1,i}$ contains $p^1_2$, for all $i \in \{2,3,\ldots,m\}$. Furthermore, either $|P_i \cup \{e_1\}| > 4$ and $C_{1,i}$ contains $p^i_2$, or $|P_i \cup \{e_1\}| = 4$ and we may choose the ordering of $P_i \cup \{e_1\}$ such that $p^i_2 \in C_{1,i}$. Now, for any other $j \in [m]$, circuit elimination between $C_{1,i}$ and $C_{1,j}$ implies that $\{p^i_1,p^i_2,p^j_1,p^j_2\}$ is a circuit. Hence, $M$ is an even-fan-paddle, completing the proof.
\end{proof}

\subsection*{Exactly one of \texorpdfstring{$\mathbf{F_1}$}{F1} and \texorpdfstring{$\mathbf{F_2}$}{F2} is odd}
Finally, we consider the case where exactly one of $F_1$ and $F_2$ is odd, and show that the resulting matroid is either an accordion or an even-fan-spike with three petals.

\begin{lemma} \label{accordion1}
    Let $M$ be a $3$-connected matroid with no detachable pairs and no disjoint maximal fans with like ends, such that $|E(M)| \geq 8$. Let $F_1 = (e_1,e_2,\ldots,e_{|F_1|})$ be a maximal fan of $M$ with even length at least four such that $\{e_1,e_2,e_3\}$ is a triangle, and let $F_2 = (f_1,f_2,f_3,f_4,f_5)$ be a maximal fan of $M$ such that $e_1 = f_1$. Then $|E(M)| \geq |F_1 \cup F_2| + 2$, and $F_2 - \{e_1\}$ is a left-hand fan-type end of $F_1$ in $M$. 
\end{lemma}

\begin{proof}
    By \Cref{fan_ends}, the set $\{f_1,f_2,f_3\}$ is a triangle.
    It follows from \cref{intersecting_fans} that $F_1 \cap F_2 = \{e_1\}$.
    \Cref{odd_fan} implies that there exists $z \notin F_2$ such that $\{f_1,f_3,f_5,z\}$ is a cocircuit. It now follows that $\lambda(F_1 \cup F_2)=2$. By orthogonality, and up to the ordering of $F_1$ if $|F_1| = 4$, we have that $z = e_2$. Hence $(F_1 \cup F_2) - \{f_5\}$ is a fan of $M \backslash f_5$. The element $e_{|F_1|}$ is not contained in a triangle of $M$, so it is also not contained in a triangle of $M \backslash f_5$. Thus $M \backslash f_5$ is not a wheel or a whirl, so \Cref{not_just_fan} implies that $|E(M \backslash f_5)| \geq |(F_1 \cup F_2) - \{f_5\}| + 2$, and thus $|E(M)| \geq |F_1 \cup F_2| + 2$.
    Thus $F_2 - \{e_1\}$ is a left-hand fan-type end of $F_1$ in $M$.
\end{proof}

\begin{lemma} \label{accordion2}
    Let $M$ be a $3$-connected matroid with no detachable pairs and no disjoint maximal fans with like ends. Let $F_1 = (e_1,e_2,\ldots,e_{|F_1|})$ be a maximal fan of $M$ with even length at least four such that $\{e_1,e_2,e_3\}$ is a triangle, and let $\{e_1,f_2,f_3\}$ be a triangle of $M$ that is not contained in a $4$-element fan, such that $\{e_1,e_2,f_2,f_3\}$ is a cocircuit.
    Then $|E(M)| \geq |F_1 \cup \{f_2,f_3\}| + 2$, and $\{f_2,f_3\}$ is a left-hand triangle-type end of $F_1$ in $M$. 
\end{lemma}

\begin{proof}
    By Tutte's Triangle Lemma, either $M \backslash f_2$ or $M \backslash f_3$ is $3$-connected. Without loss of generality, we may assume the latter. The matroid $M \backslash f_3$ has a fan $F_1 \cup \{f_2\}$. Furthermore, $M \backslash f_3$ is not a wheel or a whirl, since $e_{|F_1|}$ is not contained in a triangle. Thus, by \Cref{not_just_fan}, we have that $|E(M)| \geq |F_1 \cup \{f_2,f_3\}| + 2$. Thus $\{f_2,f_3\}$ is a left-hand triangle-type end of $F_1$ in $M$.
\end{proof}

Let $F_1$ be a maximal fan of $M$ with ordering $(e_1,e_2,\ldots,e_{|F_1|})$, having even length at least four, such that $\{e_1,e_2,e_3\}$ is a triangle, where $M$ is $3$-connected.

In the next lemma, we aim to show that if $M$ has no detachable pairs, but has distinct triangles $\{e_1,f_2,f_3\}$ and $\{e_1, g_2, g_3\}$ that are not in $4$-element fans, and a cocircuit $\{e_1,e_2,f_2,g_2\}$, then $\{f_2,f_3,g_2,g_3\}$ is a left-hand quad-type end of $F_1$. However, there is one problematic case we need to consider.

Let $X \subseteq E(M) - F_1$ such that $|X|=4$ and $|E(M)| \ge |X \cup F_1| + 2$. 
If $|F_1| = 4$ and, for some $X=\{f_2,f_3,g_2,g_3\}$, the sets $\{e_1,f_2,f_3\}$ and $\{e_1, g_2, g_3\}$ are triangles of $M$, each not contained in a $4$-element fan, and $\{e_1,e_2,f_2,g_2\}$ and $\{e_1,e_3,f_3,g_3\}$ are cocircuits, then we say $X$ is a \emph{left-hand almost-quad-type end of $F_1$}.
We also say $X$ is a right-hand almost-quad-type end of $F_1$ in $M$ when $X$ is a left-hand almost-quad-type end of $F_1$ in $M^*$.

We will eventually, in \cref{accordion4}, rule out the possibility of almost-quad-type ends by considering both the left- and right-hand ends in conjunction.

\begin{lemma} \label{accordion3}
    Let $M$ be a $3$-connected matroid with no detachable pairs and no disjoint maximal fans with like ends, such that $|E(M)| \geq 11$. Let $F_1 = (e_1,e_2,\ldots,e_{|F_1|})$ be a maximal fan of $M$ with even length at least four such that $\{e_1, e_2, e_3\}$ is a triangle. Let $\{e_1,f_2,f_3\}$ and $\{e_1,g_2,g_3\}$ be distinct triangles of $M$, each not contained in a $4$-element fan, such that $\{e_1,e_2,f_2,g_2\}$ is a cocircuit. Then $|E(M)| \geq |F_1 \cup \{f_2,f_3,g_2,g_3\}| + 2$, and $\{f_2,f_3,g_2,g_3\}$ is either
	\begin{enumerate}
        \item a left-hand quad-type end of $F_1$ in $M$, or
        \item a left-hand almost-quad-type end of $F_1$ in $M$.
	\end{enumerate}
\end{lemma}

\begin{proof}
    By \cref{no_segment}, $r(\{e_1,f_2,f_3,g_2,g_3\}) = 3$ and, in particular, the elements $f_2,f_3,g_2,g_3$ are distinct.
    We claim that $M \backslash f_3$ is $3$-connected. Suppose, to the contrary, that $M \backslash f_3$ is not $3$-connected. The element $f_3$ is not contained in a triad, so $M$ has a cyclic $3$-separation $(X,\{f_3\},Y)$. By the dual of \cref{fan_vert_sep}, we may assume that $F_1 \subseteq X$. If $f_2 \in X$, then $f_3 \in \cl(X)$, a contradiction. Furthermore, by the dual of \cref{fcl_vert_sep}, we have that $f_2 \notin \cl(X)$ and $f_2 \notin \cl^*(X)$. This implies that $g_2 \in Y$. In turn, $g_3 \in Y$, since $g_2 \notin \cl(X)$. But now $e_1 \in \cl(Y)$, so $M$ has a cyclic $3$-separation $(X-\{e_1\},\{f_3\},Y \cup \{e_1\})$ and $f_3 \in \cl(Y \cup \{e_1\})$. Thus, $\lambda(Y \cup \{e_1,f_3\}) < 2$, a contradiction. Thus, $M \backslash f_3$ is $3$-connected. By \cref{deletable_cocircuit} and orthogonality, $M$ has a $4$-element cocircuit $C^*$ containing $\{e_1,f_3\}$, either $e_2$ or $e_3$, and either $g_2$ or $g_3$. If $g_2 \in C^*$, then, by cocircuit elimination, $M$ has a cocircuit contained in $\{e_1,e_2,e_3,f_2,f_3\}$. But then $\lambda(\{e_1,e_2,e_3,f_2,f_3\}) = 2$, and $(e_1,\{e_2,e_3\},\{\{f_2,f_3\}\})$ is a deletion certificate. This contradicts \Cref{deletable_collection} since, by Tutte's Triangle Lemma, either $M \backslash g_2$ or $M \backslash g_3$ is $3$-connected. Hence, $g_3 \in C^*$. Furthermore, if $e_3 \in C^*$, then, by orthogonality, $|F_1| = 4$.
	
    Suppose $E(M) = F_1 \cup \{f_2,f_3,g_2,g_3\}$.  Then $\lambda(\{f_2,f_3,g_2,g_3\})=2$, so, as $r(\{f_2,f_3,g_2,g_3\}) = 3$, the set $\{f_2,f_3,g_2,g_3\}$ contains a cocircuit.  Since $\{f_2,f_3\}$ and $\{g_2,g_3\}$ are each not contained in a triad, $\{f_2,f_3,g_2,g_3\}$ is a cocircuit.  But $e_{|F_1|} \in \cl^*(\{f_2,f_3,g_2,g_3\})$, which implies, by orthogonality, that $\{f_2,f_3\}$ or $\{g_2,g_3\}$ is contained in a triad with $e_{|F_1|}$, a contradiction.
    Next, suppose $E(M) = F_1 \cup \{f_2,f_3,g_2,g_3,x\}$ for some element $x\not\in F_1\cup \{f_2, f_3, g_2, g_3\}$.
    We have $|F_1| \ge 4$, since $|E(M)| \geq 9$, so $\lambda(F_1-\{e_{|F_1|-1},e_{|F_1|}\}) = 2$ and, by repeatedly applying \cref{clconnectivity}, 
    $\lambda((F_1-\{e_{|F_1|-1},e_{|F_1|}\}) \cup \{f_2,f_3,g_2,g_3\}) = 2$. Therefore, $\lambda(\{e_{|F_1|-1},e_{|F_1|},x\}) = 2$, so $\{e_{|F_1|-1},e_{|F_1|},x\}$ is either a triangle or a triad. But this contradicts either the maximality of $F_1$ or orthogonality. Hence, $|E(M)| \geq |F_1 \cup \{f_2,f_3,g_2,g_3\}| + 2$.
	
    Now, if $C^*=\{e_1,e_3,f_3,g_3\}$, then $\{f_2,f_3,g_2,g_3\}$ is an almost-quad-type end of $F_1$.  So suppose that $C^*=\{e_1,e_2,f_3,g_3\}$. Then, by cocircuit elimination and orthogonality, $\{f_2,f_3,g_2,g_3\}$ is a cocircuit.  By circuit elimination and orthogonality, $\{f_2,f_3,g_2,g_3\}$ is also a circuit.  So $\{f_2,f_3,g_2,g_3\}$ is a quad-type end of $F_1$.
\end{proof}

\begin{lemma} \label{accordion4}
    Let $M$ be a $3$-connected matroid with no detachable pairs and no disjoint maximal fans with like ends, such that $|E(M)| \geq 13$. Let $F_1$ be a maximal fan of $M$ with even length at least four. Let $G \subseteq E(M)-F_1$ be a left-hand fan-type, triangle-type, quad-type, or almost-quad-type end of $F_1$. 
    Then $M$ is an accordion.
\end{lemma}

\begin{proof}
    Let $F_1 = (e_1,e_2,\ldots,e_{|F_1|})$ such that $\{e_1,e_2,e_3\}$ is a triangle. First, observe that $(e_1,\{e_2,e_3\},\{G\})$ is a deletion certificate, and $\lambda(G \cup \{e_1,e_2,e_3\}) = 2$. Let $H = E(M) - (F_1 \cup G)$, so $|H| \geq 2$, by definition.

    Suppose $|H| = 2$. Now $\lambda((F_1 \cup G) - \{e_{|F_1|}\}) = 2$, which implies that $H \cup \{e_{|F_1|}\}$ is a triad, as $e_{|F_1|}$ is not contained in a triangle. Furthermore $|F_1| > 4$, since $|E(M)| \geq 11$, which implies that $G$ is not an almost-quad-type end of $F_1$. Now $\lambda((F_1 \cup G) - \{e_{|F_1|-1},e_{|F_1|}\}) = 2$, and so $\lambda(H \cup \{e_{|F_1|-1},e_{|F_1|}\}) = 2$. Thus, either $e_{|F_1|-1} \in \cl(H \cup \{e_{|F_1|}\})$ or $e_{|F_1|-1} \in \cl^*(H \cup \{e_{|F_1|}\})$. In the latter case, $r^*(H \cup \{e_{|F_1|-1},e_{|F_1|}\}) = 2$, contradicting the dual of \Cref{no_segment}. Hence, since $e_{|F_1|}$ is not contained in a triangle, it follows that $H \cup \{e_{|F_1|-1},e_{|F_1|}\}$ is a circuit. By the dual of \Cref{accordion2}, the set $H$ is a right-hand triad-type end of $F_1$. So $M$ is an accordion. Hence, we may assume that $|H| \geq 3$.
	
    We next show that there is a triad of $M$ that meets $H$. Suppose this is not the case, that is, no element of $H$ is contained in a triad. Let $e \in H$. By \Cref{deletable_collection1}, the element $e$ is also not contained in a triangle. Furthermore, by \cref{deletable_collection}, $M \backslash e$ is not $3$-connected. Thus, by Bixby's Lemma, $M / e$ is $3$-connected. If $|E(M)| = |F_1 \cup G| + 3$, then, since $\lambda(H) = 2$, we have that $H$ is either a triangle or a triad. But no element of $H$ is contained in a triad or a triangle, a contradiction. So $|E(M)| \geq |F_1 \cup G| + 4$. Thus, the dual of \cref{contractable_circuit} implies that $M$ has a $4$-element circuit $C$ containing $\{e,e_{|F_1|}\}$, and either $e_{|F_1|-1}$ or $e_{|F_1|-2}$, and an element $f \in H-\{e\}$. But $f$ is not contained in a triad, so \Cref{circuit_so_deletable} implies that $M \backslash f$ is $3$-connected. This contradiction to \cref{deletable_collection} implies that $M$ has a triad that meets $H$.
	
    We consider two cases, depending on whether there is a triad that meets $H$, and is contained in a $4$-element fan. First, suppose $T^*$ is a triad that meets $H$, and is contained in a $4$-element fan.  Let $F_2$ be a maximal fan containing $T^*$. Since $M$ has no disjoint maximal fans with like ends, we have that $F_1 \cap F_2 \neq \emptyset$. By \cref{tip_cotip1,even_fan_paddle2}, the fan $F_2$ is odd. Therefore, by \Cref{odd_fan}, $|F_2| = 5$. Suppose that $e_1 \in F_2$.  By \cref{intersecting_fans,no_mk4}, $e_1$ is an end of $F_2$.  Thus, $M$ has a detachable pair, either by \cref{odd_fans_intersecting} when $G$ is a fan-type end, or by \cref{odd_fan_no_triangle} otherwise.  From this contradiction, we deduce that $e_1 \notin F_2$.  Thus, by \cref{intersecting_fans}, $e_{|F_1|} \in F_2$, where $e_{|F_1|}$ is an end of $F_2$. Let $H' = F_2 - \{e_{|F_1|}\}$.
    Now, the dual of \Cref{accordion1} implies that $H'$ is a right-hand fan-type end of $F_1$, where $H' \subseteq H$.

    Next, we consider the case where no triad that meets $H$ is contained in a $4$-element fan, and show that there is a set $H' \subseteq H$ that is a right-hand triad-type, quad-type, or almost-quad-type end of $F_1$.
    Let $T^*$ be a triad that meets $H$ and is not contained in a $4$-element fan.
    We have $F_1 \cap T^* \neq \emptyset$, which implies that $e_{|F_1|} \in T^*$. Let $T^* = \{e_{|F_1|},f_2,f_3\}$. Since $T^* \cap H \neq \emptyset$, we have $f_2,f_3 \in H$, by orthogonality. By Tutte's Triangle Lemma, we may assume that $M / f_2$ is $3$-connected. By the dual of \cref{deletable_cocircuit} and orthogonality, $M$ has a $4$-element circuit~$C$ containing $\{f_2,e_{|F_1|}\}$, either $e_{|F_1|-1}$ or $e_{|F_1|-2}$, and an element $e \in H$.
    If $|F_1| > 4$, then orthogonality with $\{e_{|F_1|-4},e_{|F_1|-3},e_{|F_1|-2}\}$ implies that $e_{|F_1|-1} \in C$. If $|F_1| = 4$, then, regardless of which type of left-hand end $G$ is, we have that $e_{|F_1|-2} = e_2 \in \cl^*(G \cup \{e_1\})$ and so, by orthogonality, $e_{|F_1|-1} \in C$. If $e = f_3$, then the dual of \Cref{accordion2} implies that $H' = \{f_2,f_3\}$ is a right-hand triad-type end of $F_1$. Suppose $e \neq f_3$. If $e$ is not contained in a triad, then \Cref{circuit_so_deletable} implies that $M \backslash e$ is $3$-connected, contradicting \cref{deletable_collection}. Thus, there is a triad $T_2^*$ of $M$ containing $e$. Note that, as $T_2^*$ meets $H$, it is not contained in a $4$-element fan. Furthermore, $T_2^* \cap F_1 \neq \emptyset$, so $e_{|F_1|} \in T_2^*$. Now, the dual of \Cref{accordion3} implies that $H' = (T^* \cup T_2^*) - \{e_{|F_1|}\}$ is a right-hand quad-type or almost-quad-type end of $F_1$.
	
    In either case, we have a set $H' \subseteq H$ that is a right-hand fan-type, triad-type, quad-type, or almost-quad-type end of $F_1$. Also note that there is a circuit of $M$ containing $\{e_{|F_1|-1},e_{|F_1|}\}$ and two elements of $H'$. Thus, by orthogonality, $G$ is not a left-hand almost-quad-type end of $F_1$. Similarly, $H'$ is not a right-hand almost-quad-type end of $F_1$. Now, $M$ has a contraction certificate $(e_{|F_1|},\{e_{|F_1|-1},e_{|F_1|-2}\},\{H'\})$ and $\lambda(H' \cup \{e_{|F_1|-2},e_{|F_1|-1},e_{|F_1|}\}) = 2$. Combined with the deletion certificate $(e_1,\{e_2,e_3\},\{G\})$, \Cref{no_other_elements} implies that every element of $E(M) - (F_1 \cup G \cup H')$ is contained in a $4$-element fan.
    Suppose $F$ is a maximal fan of $M$ with length at least four that is not contained in $F_1 \cup G \cup H'$.
    Then $F$ contains either $e_1$ or $e_{|F_1|}$, so, by \cref{tip_cotip1,even_fan_paddle2}, $F$ is odd. But now $F$ meets a maximal fan of odd length, which is contained in either $G \cup \{e_1\}$ or $H' \cup \{e_{|F_1|}\}$. This contradicts \cref{odd_fan_no_triangle} or \cref{odd_fans_intersecting}, so $E(M) = F_1 \cup G \cup H'$. Thus, $M$ is an accordion.
\end{proof}

\begin{lemma} \label{accordion5}
    Let $M$ be a $3$-connected matroid with no detachable pairs and no disjoint maximal fans with like ends. Let $F_1 = (e_1,e_2,\ldots,e_{|F_1|})$ be a maximal fan of $M$ with even length at least four such that $\{e_1,e_2,e_3\}$ is a triangle. Let $T = \{e_1,f_2,f_3\}$ be a triangle of $M$ and let $T^* = \{e_{|F_1|},g_2,g_3\}$ be a triad of $M$, such that $T \cap T^* = \emptyset$ and neither $T$ nor $T^*$ is contained in $F_1$, and let $e \in E(M) - (F_1 \cup \{f_2,f_3,g_2,g_3\})$ such that $\{e_1,e_2,f_2,e\}$ is a cocircuit. If $M$ has an element $x \neq e$ such that $x$ is not contained in a triangle or a triad and $M \backslash x$ is $3$-connected, then $x \in \cl^*(F_1 \cup \{f_2,f_3\})$ and $x \in \cl(F_1 \cup \{g_2,g_3\})$.
\end{lemma}

\begin{proof}
    \Cref{deletable_cocircuit} implies that $M$ has a $4$-element cocircuit containing $\{e_1,x\}$, and either $e_2$ or $e_3$, and either $f_2$ or $f_3$, so $x \in \cl^*(F_1 \cup \{f_2,f_3\})$. Now, suppose $M / x$ is not $3$-connected. Then $M$ has a vertical $3$-separation $(X,\{x\},Y)$, and we may assume, by \cref{fan_vert_sep}, that $F_1 \subseteq X$. If $\{f_2,f_3\} \subseteq X$, then $x \in \cl^*(X)$, contradicting orthogonality. This implies, by \cref{fcl_vert_sep}, that $f_2 \notin \cl(X)$ and, as $x$ is not contained in a triangle, $f_2 \notin \cl^*(X)$, from which it follows that $\{e,f_2,f_3\} \subseteq Y$. But now $e_1 \in \cl(Y)$, and $e_2 \in \cl^*(Y \cup \{e_1\})$. Repeating in this way, $(X - F_1,\{x\},Y \cup F_1)$ is a vertical $3$-separation of $M$. However, $x \in \cl^*(Y \cup F_1)$, a contradiction. Hence, $M / x$ is $3$-connected, so, by the dual of \cref{deletable_cocircuit}, $M$ has a $4$-element circuit containing $\{e_{|F_1|}, x\}$, and either $e_{|F_1|-2}$ or $e_{|F_1|-1}$, and either $g_2$ or $g_3$. Thus, $x \in \cl(F_1 \cup \{g_2,g_3\})$ as desired.
\end{proof}

\begin{lemma} \label{accordion6}
    Let $M$ be a $3$-connected matroid with no detachable pairs and no disjoint maximal fans with like ends, such that $|E(M)| \geq 13$. Let $F_1 = (e_1,e_2,\ldots,e_{|F_1|})$ be a maximal fan of $M$ with even length at least four such that $\{e_1,e_2,e_3\}$ is a triangle, and suppose every $4$-element fan of $M$ is contained in $F_1$. Let $\{e_1,f_2,f_3\}$ be a triangle of $M$ such that $M \backslash f_2$ is $3$-connected. Furthermore, let $e \in E(M) - (F_1 \cup \{f_2,f_3\})$ such that $\{e_1,e_2,f_2,e\}$ is a cocircuit and $e$ is not contained in a triangle. Then $M$ is an even-fan-spike with partition $(F_1, \{e,f_2\}, \{f_3,z\})$, for some $z \notin F_1 \cup \{e,f_2,f_3\}$.
\end{lemma}

\begin{proof}
    Since $M \backslash f_2$ is $3$-connected, the dual of \cref{circuit_so_deletable} implies that $M / e$ is $3$-connected. Therefore, by the dual of \cref{deletable_cocircuit}, $M$ has a $4$-element circuit $C$ containing $\{e_{|F_1|},e\}$, and one of $e_{|F_1|-2}$ and $e_{|F_1|-1}$, and an element $f \notin F_1 \cup \{e\}$. If $T^*$ is a triad of $M$ that is not contained in $F_1$, then, since $T^*$ is not contained in a $4$-element fan and $M$ has no disjoint maximal fans with like ends, we have that $T^* \cap F_1 \neq \emptyset$. Hence, $e_{|F_1|} \in T^*$, and, by orthogonality with $C$, either $e \in T^*$ or $f \in T^*$. Now, every triad of $M$ that is not contained in $F_1$ contains either $\{e_{|F_1|},e\}$ or $\{e_{|F_1|},f\}$. It follows, by \Cref{no_segment}, that there are at most two elements of $E(M) - (F_1 \cup \{f_2,f_3,e,f\})$ contained in triads.
	
    The strategy for this proof is to find a set $X$ with $F_1 \cup \{f_2,f_3,e,f\} \subseteq X$ and $\lambda(X) = 2$. Then $(e_1,F_1-\{e_1\},\{\{f_2,f_3\}\})$ is a deletion certificate contained in $X$, and $e_{|F_1|} \in \cl^*(F_1 - \{e_{|F_1|}\})$, and, for all $i \in [|F_1|]$, we have that $e_i \in \cl(X)-\{e_i\}$. Hence, if $|E(M)| \geq |X| + 3$, then \cref{deletable_collection_contractable_el} implies that every element of $E(M) - X$ is contained in a triad. But $E(M)-X$ has at most two elements contained in triads, so $|E(M)| \leq |X| + 2$.
	
    We set about finding such a set $X$. Suppose $f \neq f_2$. Orthogonality with the cocircuit $\{e_1,e_2,f_2,e\}$ implies that $e_2 \in C$, so $|F_1| = 4$. If $f = f_3$, then $\lambda(F_1 \cup \{f_2,f_3,e\}) = 2$, and so, letting $X=F_1 \cup \{f_2,f_3,e\}$, we have $|E(M)| \leq |X| + 2 = 9$, a contradiction. So $f \neq f_3$.

    Assume $f$ is not contained in a triad. Then \cref{circuit_so_deletable} implies that $M \backslash f$ is $3$-connected. Hence, by \cref{deletable_cocircuit}, $M$ has a $4$-element cocircuit containing $\{e_1,f\}$, either $e_2$ or $e_3$, and either $f_2$ or $f_3$. Now $\lambda(F_1 \cup \{f_2,f_3,e,f\}) = 2$, so $|E(M)| \le 10$, again a contradiction.
	
    Next, assume that $f$ is contained in a triad $T^*$. As each triad that is not contained in $F_1$ contains $e_{|F_1|}$, the triad $T^*$ contains $e_4$. If $e \in T^*$, then $\lambda(F_1 \cup \{f_2,f_3,e,f\})=2$, so $|E(M)| \le 10$, a contradiction. Now, by orthogonality, $T^*=\{f,e_4,h\}$ for some $h \notin F_1 \cup \{f_2,f_3,e,f\}$. Let $Z = F_1 \cup \{f_2,f_3,f,h\}$. Then $|Z \cup \{e\}| = 9$, and at most two elements of $E(M)-(Z \cup \{e\})$ are contained in triads.  So there exists an element $g \notin Z \cup \{e\}$ such that $g$ is not contained in a triad. If $g$ is contained in a triangle $T$, then $T$ contains $e_1$, since $M$ has no disjoint maximal fans with like ends. But $e \notin T$ since $e$ is not contained in a triangle, $e_2 \notin T$ by orthogonality, and $f_2 \notin T$ as otherwise $r(\{e_1,f_2,f_3,g\})=2$, a contradiction to \cref{no_segment}. Now $T$ intersects the cocircuit $\{e_1,e_2,f_2,e\}$ in a single element, which contradicts orthogonality. We deduce that $g$ is not contained in a triangle or a triad. By Bixby's Lemma, either $M \backslash g$ or $M / g$ is $3$-connected. Then, either \cref{accordion5} or its dual implies that $g \in \cl(Z)$ and $g \in \cl^*(Z)$. So $\lambda(Z \cup \{e,g\}) = 2$, implying $|E(M)| \le |Z \cup \{e,g\}| + 2$, a contradiction since $|E(M)| \geq 13 = |Z \cup \{e,g\}| + 3$.
	
    It now follows that $f = f_2$. This means that $\lambda(F_1 \cup \{f_2,f_3,e\}) = 2$, and so $|E(M)| \leq |F_1 \cup \{f_2,f_3,e\}| + 2$. Let $H =E(M) - (F_1 \cup \{e,f_2\})$, so $|H| \le 3$. We have that $\lambda(F_1 \cup \{e,f_2\}) \le 2$, so $\lambda(H) \le 2$. If $|H| = 3$, then $H$ is a triangle or a triad disjoint from $F_1$, a contradiction. Now, $|H| \le 2$, so $M$ is an even-fan-spike with partition $(F_1,\{e,f_2\},H)$, by \Cref{even_fan_spike2}(ii). 
\end{proof}

\begin{lemma} \label{accordion}
    Let $M$ be a $3$-connected matroid with no detachable pairs and no disjoint maximal fans with like ends, such that $|E(M)| \geq 13$. Suppose $M$ has a unique maximal fan $F_1$ having even length at least four, and let $F_1 = (e_1,e_2,\ldots,e_{|F_1|})$. Let $F_2 = (f_1,f_2,\ldots, f_{|F_2|})$ be a maximal fan of $M$ with odd length at least three such that $f_1 = e_1$. Then either $M$ is an accordion, or $|F_2|=3$ and $M$ is an even-fan-spike with partition $(F_1, \{e,f_2\}, \{f_3,z\})$ for some distinct $e,z \notin F_1 \cup F_2$.
\end{lemma}

\begin{proof}
    Assume that $\{e_1,e_2,e_3\}$ is a triangle.
Since $F_2$ is odd, it follows by \Cref{odd_fan} that $|F_2|\le 5$. Also $F_1 \cap F_2 \neq \emptyset$. If $|F_2|=5$, then, by \Cref{accordion1}, the set $G = F_2 - \{e_1\}$ is a left-hand fan-type end of $F_1$, and $|E(M)| \geq |F_1 \cup G| + 2$. Thus, $M$ is an accordion by \Cref{accordion4}.
	
We may now assume $|F_2|=3$ and that every $4$-element fan of $M$ is contained in $F_1$. Without loss of generality, we may also assume that $M \backslash f_2$ is $3$-connected by Tutte's Triangle Lemma.
Note that $e_{|F_1|} \notin \cl(F_1 - e_{|F_1|})$, so $e_{|F_1|} \in \cl^*( E(M)-F_1)$, and thus $|E(M)-F_1| \ge 3$.  Hence, by \cref{contractable_circuit} and orthogonality, there is some $e \notin F_1$ such that either $\{e_1,e_2,f_2,e\}$ is a $4$-element cocircuit of $M$, or $|F_1|=4$ and $\{e_1, e_3, f_2, e\}$ is a $4$-element cocircuit of $M$. If $e = f_3$, then $\{f_2,f_3\}$ is a left-hand triangle-type end of $F_1$ and $|E(M)| \geq |F_1 \cup \{f_2,f_3\}| + 2$, by \Cref{accordion2}. Again, $M$ is an accordion by \Cref{accordion4}.
	
Finally, suppose $e \neq f_3$. If $e$ is not contained in a triangle, then $M$ is an even-fan-spike with partition $(F_1, \{e,f_2\}, \{f_3,z\})$ for some $z \notin F_1 \cup \{e,f_2,f_3\}$, by \Cref{accordion6}. Otherwise, $e$ is contained in a triangle $T$, which contains $e_1$. By \cref{accordion3}, $(F_2 \cup T) - \{e_1\}$ is a left-hand quad-type or almost-quad-type end of $F_1$ and $|E(M)| \geq |F_1 \cup F_2 \cup T| + 2$. Therefore, $M$ is an accordion, by \Cref{accordion4}.
\end{proof}

\subsection*{Putting it together}
\begin{proof} [Proof of \Cref{intersecting_fans_detachable}.]
    Let $(e_1,e_2,\ldots,e_{|F_1|})$ and $(f_1,f_2,\ldots,f_{|F_2|})$ be orderings of $F_1$ and $F_2$ respectively. If $M$ has an $M(K_4)$-separator, then $M$ has a detachable pair by \Cref{no_mk4}, so (i) holds.  Thus we may assume that $F_1\cup F_2$ is not an $M(K_4)$-separator in $M$, and, dually, not an $M(K_4)$-separator in $M^*$. By \Cref{intersecting_fans}, we may also assume that $e_1 = f_1$, and up to duality, that $\{e_1,e_2,e_3\}$ is a triangle (noting that the outcomes in (ii)--(v) are self-dual). By \Cref{fan_ends}, the set $\{f_1,f_2,f_3\}$ is also a triangle. First suppose $F_1$ and $F_2$ are both odd. By \Cref{odd_fan}, $|F_1|=5$ and $|F_2| \in \{3,5\}$. Then \cref{odd_fan_no_triangle,odd_fans_intersecting} imply that $M$ has a detachable pair, so (i) holds. Next, suppose $M$ has distinct maximal fans having even length at least four.  Let $F_1$ and $F_2$ be such fans. If $F_1 \cap F_2 = \{e_1,e_{|F_1|}\}$, then $M$ is an even-fan-spike with tip and cotip by \Cref{tip_cotip}, so (iii) holds. If $F_1 \cap F_2 = \{e_1\}$, then $M$ is an even-fan-paddle by \Cref{even_fan_paddle}, so (v) holds. Finally, we may assume that $M$ has a unique maximal fan with even length at least four.  Without loss of generality, this fan is $F_1$, whereas $F_2$ is odd. Then \Cref{accordion} implies that either (ii) or (iv) holds.
\end{proof}

\section{Remaining \texorpdfstring{$4$}{4}-element fan cases} \label{fans}
We may now assume that $M$ has no disjoint maximal fans with like ends, and no distinct fans with non-empty intersection. This means that if $F_1 = (e_1,e_2,\ldots,e_{|F_1|})$ and $F_2 = (f_1,f_2,\ldots,f_{|F_2|})$ are distinct maximal fans such that $|F_1| \geq 4$ and $|F_2| \geq 3$, then $F_1$ and $F_2$ are disjoint and either $\{e_1,e_2,e_3\}$ and $\{e_{|F_1|-2},e_{|F_1|-1},e_{|F_1|}\}$ are both triangles and $\{f_1,f_2,f_3\}$ and $\{f_{|F_2|-2},f_{|F_2|-1},f_{|F_2|}\}$ are both triads, or vice versa. To refer to this assumption, we say that $M$ has \emph{no distinct maximal fans with like ends}. The goal of this section is to consider the case in which $M$ has a $4$-element fan, but no distinct maximal fans with like ends, and prove the following:

\begin{theorem} \label{single_fans}
    Let $M$ be a $3$-connected matroid with no distinct maximal fans with like ends, such that $|E(M)| \geq 13$, and suppose that $M$ has a maximal fan with length at least four. Then one of the following holds:
	\begin{enumerate}
		\item $M$ has a detachable pair,
		\item $M$ is a wheel or a whirl,
        \item $M$ is an even-fan-spike, or
		\item $M$ or $M^*$ is a quasi-triad-paddle with a co-augmented-fan petal.
	\end{enumerate}
\end{theorem}

\subsection*{Two fans of length at least four}
First, we consider the case where $M$ has two distinct maximal fans each with length at least four.

\begin{lemma} \label{unlike_odd_fans}
    Let $M$ be a $3$-connected matroid with no distinct maximal fans with like ends, such that $|E(M)| \geq 13$. Let $F_1$ and $F_2$ be distinct maximal fans of $M$ each with length at least four. Then $M$ has a detachable pair.
\end{lemma}

\begin{proof}
    Suppose, towards a contradiction, that $M$ does not have a detachable pair. Let $(e_1,e_2,\ldots,e_{|F_1|})$ be an ordering of $F_1$, and $(f_1,f_2,\ldots,f_{|F_2|})$ be an ordering of $F_2$. Since $M$ has no distinct maximal fans with like ends, we may assume that $\{e_1,e_2,e_3\}$ and $\{e_{|F_1|-2},e_{|F_1|-1},e_{|F_1|}\}$ are triangles, and $\{f_1,f_2,f_3\}$ and $\{f_{|F_2|-2},f_{|F_2|-1},f_{|F_2|}\}$ are triads. This implies that $F_1$ and $F_2$ are odd, so, by \Cref{odd_fan}, we have that $|F_1| = 5$ and $|F_2| = 5$. Furthermore, by the same lemma, there exists $z \notin F_1$ such that $\{e_1,e_3,e_5,z\}$ is a cocircuit, and there exists $z' \notin F_2$ such that $\{f_1,f_3,f_5,z'\}$ is a circuit. By orthogonality, $z\neq z'$.
	
    Since $M$ has no distinct maximal fans with like ends, every triangle or triad of $M$ is contained in $F_1$ or $F_2$. By orthogonality, this means that $z$ is not contained in a triangle. Since $z \in \cl^*(F_1)$, it follows by the dual of \cref{closure_deletable} that $M / z$ is $3$-connected. Similarly, $M \backslash z'$ is $3$-connected. We next show that $z \in \{f_1,f_5\}$. Suppose this is not the case. The dual of \cref{contractable_circuit} implies that $M$ has a $4$-element circuit $C_1$ containing $\{z,f_1\}$. By orthogonality with $\{f_1,f_2,f_3\}$ and $\{e_1,e_3,e_5,z\}$, the circuit $C_1$ contains $f_2$ and either $e_1$ or $e_5$. Without loss of generality, assume that $C_1 = \{z,e_1,f_1,f_2\}$. Also, again by the dual of \cref{contractable_circuit}, $M$ has a $4$-element circuit $C_2$ containing $\{z,f_4,f_5\}$ and either $e_1$ or $e_5$. If $e_1 \in C_2$, then circuit elimination implies $M$ has a circuit contained in $\{f_1,f_2,f_4,f_5\}$, a contradiction. So $C_2 = \{z,e_5,f_4,f_5\}$.
	
	Also, orthogonality with $\{e_1,e_3,e_5,z\}$ implies that $z' \notin \{e_1,e_5\}$. Hence, \Cref{deletable_cocircuit} and orthogonality implies that $M$ has cocircuits $C_1^* = \{z',f_1,e_1,e_2\}$ and $C_2^* = \{z',f_5,e_4,e_5\}$. But now $\lambda(F_1 \cup F_2 \cup \{z,z'\}) \leq 1$, so
$$|E(M)| \leq |F_1 \cup F_2 \cup \{z,z'\}| + 1 = 13.$$
But $|E(M)| \geq 13$, so $E(M) = F_1 \cup F_2 \cup \{z,z',x\}$, for some $x \notin F_1 \cup F_2 \cup \{z,z'\}$. As $\lambda(F_1 \cup \{z\}) = 2$ and $\lambda(F_2 \cup \{z'\}) = 2$, this implies that either $x \in \cl(F_1 \cup \{z\})$ and $x \in \cl(F_2 \cup \{z'\})$, or $x \in \cl^*(F_1 \cup \{z\})$ and $x \in \cl^*(F_2 \cup \{z'\})$. Up to duality, we may assume the former, in particular, $x \in \cl(F_2 \cup \{z'\})$. But $z' \in \cl(F_2)$, so $x \in \cl(F_2)$, and $\lambda(F_2 \cup \{x\}) = 2$. Thus $\lambda(F_1 \cup \{z,z'\}) = 2$. The cocircuits $C_1^*$ and $C_2^*$ imply that $\lambda(F_1 \cup \{z,z',f_1,f_5\}) = 2$, and the circuit $\{f_1,f_3,f_5,z'\}$ implies that $\lambda(F_1 \cup \{z,z',f_1,f_3,f_5\}) = 2$. Thus, $\lambda(\{f_2,f_4,x\}) = 2$, which implies by orthogonality that $\{f_2,f_4,x\}$ is a triad. But now $x \in \cl(F_2) \cap \cl^*(F_2)$, a contradiction.
	
Thus, $z \in \{f_1,f_5\}$. Dually, $z' \in \{e_1,e_5\}$. Then $(\{z\},F_1,F_2-\{z\})$ is a contraction certificate and $(\{z'\},F_1-\{z'\},F_2)$ is a deletion certificate and $\lambda(F_1 \cup F_2) = 2$. Since $|E(M)| \geq 13$, \Cref{no_other_elements} implies that every element of $M$ that is not contained in $F_1 \cup F_2$ is contained in a $4$-element fan. But $M$ has no distinct maximal fans with like ends, so $M$ has no other $4$-element fans.  Hence $E(M)=F_1 \cup F_2$, contradicting that $|E(M)| \ge 13$.
\end{proof}

\subsection*{Even fan of length at least four}
Next, we consider the case where $M$ has an even fan of length at least four, and show that $M$ is an even-fan-spike. In this case, as $M$ has no distinct maximal fans with like ends, we may also assume that $M$ has no other triangles or triads.

\begin{lemma} \label{simple_even_fan_spike1}
    Let $M$ be a $3$-connected matroid with no detachable pairs such that $|E(M)|\ge 13$. Let $F = (e_1,e_2,\ldots,e_{|F|})$ be a maximal fan of $M$ with even length at least four such that $\{e_1,e_2,e_3\}$ is a triangle. Suppose every triangle or triad of $M$ is contained in $F$. Let $e \notin F$ such that $M \backslash e$ is $3$-connected. Then either $M$ is an even-fan-spike, or $|F| = 4$ and there exist (not necessarily distinct) elements $f,g,h \in E(M) - F$ such that
	\begin{enumerate}
        \item for some $i \in \{2,3\}$, the set $\{e_1,e_i,e,f\}$ is a cocircuit, and $\{e_i,e_4,f,g\}$ is a circuit, and
		\item $\lambda(F \cup \{e,f,g,h\}) = 2$.
	\end{enumerate}
\end{lemma}

\begin{proof}
    Since every triangle or triad is contained in $F$, we have $|E(M)| \ge |F| + 4$.
    By \Cref{deletable_cocircuit} and orthogonality with $\{e_1,e_2,e_3\}$, there exists $f \notin F \cup \{e\}$ such that $C^*=\{e_1,e_i,e,f\}$ is a cocircuit of $M$ for some $i \in \{2,3\}$. Now, $f$ is not contained in a triangle, so the dual of \cref{circuit_so_deletable} implies that $M / f$ is $3$-connected. Thus, by the dual of \cref{contractable_circuit}, $M$ has a $4$-element circuit $C = \{e_{|F|},e_j,f,g\}$ for some $g \notin F \cup \{f\}$ and $j \in \{|F|-2,|F|-1\}$. If $e_j \neq e_i$, then orthogonality with $C^*$ implies that $g = e$. Furthermore, either $|F| > 4$ and orthogonality implies that $C^* = \{e_1,e_2,e,f\}$ and $C = \{e_{|F|-1},e_{|F|},e,f\}$, or $|F| = 4$ and we may choose an ordering of $F$ such that $C^* = \{e_1,e_2,e,f\}$ and $C = \{e_3,e_4,e,f\}$. In either case, \Cref{even_fan_spike} implies that $M$ is an even-fan-spike, as desired. Hence, $e_j = e_i$, which implies that $|F| = 4$.
	
    If $g = e$, then $\lambda(F \cup \{e,f\}) = 2$, and the result holds. Otherwise, \Cref{circuit_so_deletable} implies that $M \backslash g$ is $3$-connected. Thus, by \cref{deletable_cocircuit} again, $M$ has a $4$-element cocircuit $C_2^*$ containing $\{e_1,g\}$, either $e_2$ or $e_3$, and an element $h \notin F \cup \{g\}$. If $h \in \{e,f\}$, then $\lambda(F \cup \{e,f,g\}) = 2$, as desired. So assume that $h\not\in \{e, f\}$. Then orthogonality with $C$ implies that $C_2^* = \{e_1,e_i,g,h\}$, and, by the dual of \cref{circuit_so_deletable} again, $M / h$ is $3$-connected. Now, $M$ has a $4$-element circuit $C_2$ containing $\{e_4,h\}$ and either $e_2$ or $e_3$. If $e_i \in C_2$, then orthogonality with $C^*$ implies that either $e \in C_2$ or $f \in C_2$, and if $e_i \notin C_2$, then orthogonality with $C_2^*$ implies that $g \in C_2$. In either case, $\lambda(F \cup \{e,f,g,h\}) = 2$, completing the proof.
\end{proof}

\begin{lemma} \label{simple_even_fan_spike}
    Let $M$ be a $3$-connected matroid with no detachable pairs such that $|E(M)| \geq 13$. Let $F$ be a maximal fan of $M$ with even length at least four. If every triangle or triad of $M$ is contained in $F$, then $M$ is either a wheel, a whirl or an even-fan-spike.
\end{lemma}

\begin{proof}
    If $E(M) = F$, then $M$ is a wheel or a whirl by \Cref{not_just_fan}. Otherwise, let $e \in E(M) - F$, and suppose that $M$ is not an even-fan-spike. By Bixby's Lemma, either $M \backslash e$ or $M / e$ is $3$-connected and so, up to duality, we may assume that $M \backslash e$ is $3$-connected. By \Cref{simple_even_fan_spike1}, $|F|=4$ and there exists $f,g,h \notin F$ and an ordering $(e_1,e_2,e_3,e_4)$ of $F$ such that $\{e_1,e_2,e,f\}$ is a cocircuit and $\{e_2,e_4,f,g\}$ is a circuit, and $\lambda(F \cup \{e,f,g,h\}) = 2$. 
	
	Now, let $e' \notin F \cup \{e,f,g,h\}$. Then either $M\backslash e'$ or $M/e'$ is $3$-connected.
    Assume that $M\backslash e'$ is $3$-connected.  By \cref{simple_even_fan_spike1}, there exists $f',g',h' \notin F$ and $i \in \{2,3\}$ such that $\{e_1,e_i,e',f'\}$ is a cocircuit and $\{e_i,e_4,f',g'\}$ is a circuit and $\lambda(F \cup \{e',f',g',h'\}) = 2$. Furthermore, if $i = 2$, then orthogonality implies that $f' \in \{f,g\}$. But now $e' \in \cl^*(F \cup \{e,f,g,h\})$, which contradicts the fact that $M \backslash e'$ is $3$-connected. So $i = 3$.
	
    Since $|E(M)| \geq 13$, there exists $e'' \notin F \cup \{e,f,g,h,e',f',g',h'\}$ such that either $M\backslash e''$ or $M/e''$ is $3$-connected.  As in the previous paragraph, if $M\backslash e''$ is $3$-connected, then $M$ has a $4$-element cocircuit $\{e_1, e_3, e'', f''\}$, where $f''\not\in F$. But then orthogonality with the circuit $\{e_3,e_4,f',g'\}$ implies $f''\in \{f', g'\}$, and so $e'' \in \cl^*(F \cup \{e',f',g',h'\})$, a contradiction. On the other hand, if $M/e''$ is $3$-connected, then, by the dual of \Cref{simple_even_fan_spike1}, $M$ has a $4$-element circuit $C$ containing $\{e_4, e'', f''\}$, where $f''\not\in F$, and either $e_2$ or $e_3$. If $e_2\in C$, then orthogonality with the cocircuit $\{e_1,e_2,e,f\}$ implies that $f''\in \{e, f\}$, and so $e''\in \cl(F\cup \{e, f, g, h\})$, and if $e_3\in C$, then orthogonality with the cocircuit $\{e_1,e_3,e',f'\}$ implies that $f''\in \{e', f'\}$, and so $e''\in \cl(F\cup \{e', f', g', h'\})$. Again, each case is a contradiction, and so $M\backslash e'$ is not $3$-connected. An analogous argument applies in the case that $M/e'$ is $3$-connected.
\end{proof}

\subsection*{Odd fan of length at least five}
Finally, we consider the case where $M$ has an odd fan of length at least five. By \Cref{odd_fan}, this fan has length five. The next lemma is similar to \Cref{circuit_so_deletable} and will be useful in this subsection.

\begin{lemma} \label{special_circuit_so_deletable}
	Let $M$ be a $3$-connected matroid. Let $F = (e_1,e_2,e_3,e_4,e_5)$ be a maximal fan of $M$ such that $\{e_1,e_2,e_3\}$ is a triangle, and let $z \in E(M) - F$ such that $\{e_1,e_3,e_5,z\}$ is a cocircuit. If $M$ has a circuit $\{e_1,z,e,f\}$ such that $M / e$ is $3$-connected and $f$ is not contained in a triad, then $M \backslash f$ is $3$-connected.
\end{lemma}

\begin{proof}
    Suppose $M \backslash f$ is not $3$-connected, and note that $e,f \notin F$. Since $f$ is not contained in a triad, $M$ has a cyclic $3$-separation $(X,\{f\},Y)$ such that $F \subseteq X$ by the dual of \Cref{fan_vert_sep}. Now, $z \in \cl^*(F)$, so we may assume that $z \in X$. If $e \in X$, then $f \in \cl(X)$, a contradiction. Therefore, $e \in Y$, and $e \in \cl(X \cup \{f\})$, which contradicts the fact that $M / e$ is $3$-connected unless $r(Y)=2$ and $|Y|=2$. But then, in the exceptional case, $Y\cup \{f\}$ is a triad, a contradiction.
\end{proof}

\begin{lemma} \label{coaugmented1}
	Let $M$ be a $3$-connected matroid such that $|E(M)| \geq 11$.
    Let $F$ be a maximal fan of $M$ with length five. Suppose every triangle or triad of $M$ is contained in $F$. Then $M$ has a detachable pair.
\end{lemma}

\begin{proof}
    Suppose, to the contrary, that $M$ has no detachable pairs. Let $F = (e_1,e_2,e_3,e_4,e_5)$. By duality, we may assume that $\{e_1,e_2,e_3\}$ is a triangle. By \Cref{odd_fan}, there exists $z \in E(M) - F$ such that $\{e_1,e_3,e_5,z\}$ is a cocircuit, so $z \in \cl^*(F)$. Let $e \notin F \cup \{z\}$. Suppose $M / e$ is $3$-connected. Then, by the dual of \cref{contractable_circuit}, $M$ has a $4$-element circuit $C$ containing $\{e,z\}$, either $e_1$ or $e_5$, and an element $f \notin F \cup \{z\}$. Without loss of generality, $C = \{e_1,z,e,f\}$. Note that $(e_1,F-\{e_1\},\{\{z,e,f\}\})$ is a deletion certificate.
	
    By \Cref{special_circuit_so_deletable}, the matroid $M \backslash f$ is $3$-connected. Now, by \cref{deletable_cocircuit}, $M$ has a $4$-element cocircuit $C^*$ containing $\{e_5,f\}$, and, by orthogonality with $C$, either $C^* = \{e_4,e_5,z,f\}$ or $C^* = \{e_4,e_5,e,f\}$. In either case, $\lambda(F \cup \{z,e,f\}) = 2$. Furthermore, $z \in \cl^*(F)$ and, for all $x \in F \cup \{z\}$, we have that $x \in \cl(F \cup \{z,e,f\})$. Since $|E(M)| \geq 11$, \Cref{deletable_collection_contractable_el} implies that every element of $E(M) - (F \cup \{z,e,f\})$ is contained in a triad, a contradiction.
    We deduce that $M/e$ is not $3$-connected.
	
    Thus, by Bixby's Lemma, $M \backslash e$ is $3$-connected, and, furthermore, for all $x \in E(M) - (F \cup \{z\})$, the matroid $M / x$ is not $3$-connected. Now, by \cref{deletable_cocircuit} once more, $M$ has a $4$-element cocircuit $\{e_1,e_2,e,f'\}$, where $f' \notin F \cup \{z\}$. But then the dual of \Cref{circuit_so_deletable} implies that $M / f'$ is $3$-connected, a contradiction. This completes the proof of the lemma.
\end{proof}

\begin{lemma} \label{coaugmented2}
    Let $M$ be a $3$-connected matroid with no detachable pairs such that $|E(M)| \geq 13$. Let $F = (e_1,e_2,e_3,e_4,e_5)$ be a maximal fan of $M$ such that $\{e_1,e_2,e_3\}$ is a triangle, and every triangle of $M$ is contained in $F$. Let $z \in E(M) - F$ such that $\{e_1,e_3,e_5,z\}$ is a cocircuit. Then $M$ has a triad that is disjoint from $F \cup \{z\}$.
\end{lemma}

\begin{proof}
    Suppose that every triad of $M$ meets $F \cup \{z\}$. By \Cref{coaugmented1}, the matroid $M$ has a triad $T^*$ not contained in $F$. Now, $T^* \cap (F \cup \{z\}) \neq \emptyset$. By \cref{fan_ends,intersecting_fans}, we have that $T^*$ and $F$ are disjoint. Thus, $z \in T^*$, so let $T^* = \{z,e,f\}$, where $e, f\not\in F$. Note that $(z,F,\{\{e,f\}\})$ is a contraction certificate. Since $T^*$ is not contained in a $4$-element fan, it follows by Tutte's Triangle Lemma that either $M / e$ or $M / f$ is $3$-connected. We may assume that $M / e$ is $3$-connected. By the dual of \Cref{contractable_circuit}, there is a $4$-element circuit $\{e_i,z,e,g\}$ of $M$, for some $i \in \{1,5\}$ and $g \notin F \cup \{e,z\}$. Assume, without loss of generality, that $i = 1$.

    Now, $(e_1,F-\{e_1\},\{\{z,e,g\}\})$ is a deletion certificate. If $g = f$, then $\lambda(F \cup \{z,e,f\}) = 2$, and $F \cup \{z,e,f\}$ contains both a deletion and a contraction certificate, which contradicts \Cref{no_other_elements}. Hence, $g \neq f$.
    Suppose $g$ is not contained in a triad. \Cref{special_circuit_so_deletable} implies that $M \backslash g$ is $3$-connected. Thus, by \cref{deletable_cocircuit}, $M$ has a $4$-element cocircuit containing $\{e_4,e_5,g\}$ and an element of $\{e,z,e_1\}$. Now $\lambda(F \cup \{z,e,f,g\}) = 2$, again contradicting \Cref{no_other_elements}.
	Thus $g$ is contained in a triad of $M$. This triad contains $z$, so $M$ has a triad $\{z,g,h\}$, for some $h \notin F \cup \{z,e,f,g\}$.

    The dual of \cref{contractable_circuit} implies that $M$ has a $4$-element circuit $C$ containing $\{e,h\}$, and an element of $\{z,g\}$.
    By orthogonality, if $z \in C$, then $C$ also contains one of $e_1$ and $e_5$; and if $z \notin C$, then $C=\{e,f,g,h\}$.
    But now, in either case, $\lambda(F \cup \{z,e,f,g,h\}) = 2$, a contradiction to \Cref{no_other_elements}, thereby completing the proof of the lemma.
\end{proof}

\begin{lemma} \label{coaugmented3}
    Let $M$ be a $3$-connected matroid with no detachable pairs such that $|E(M)| \geq 13$. Let $F = (e_1,e_2,e_3,e_4,e_5)$ be a maximal fan of $M$ such that $\{e_1,e_2,e_3\}$ is a triangle and every triangle of $M$ is contained in $F$. Let $z \in E(M) - F$ such that $\{e_1,e_3,e_5,z\}$ is a cocircuit, and let $T^*$ be a triad of $M$ disjoint from $F\cup \{z\}$. Then 
	\begin{enumerate}
		\item $T^* = \{a,b,c\}$ such that $\{e_1,z,a,b\}$ and $\{e_5,z,b,c\}$ are circuits, and
		\item every element of $E(M) - (F \cup T^* \cup \{z\})$ is contained in a triad that is disjoint from $F \cup T^* \cup \{z\}$.
	\end{enumerate} 
\end{lemma}

\begin{proof}
    Using Tutte's Triangle Lemma, the dual of \Cref{contractable_circuit}, and orthogonality, it follows that we may label $T^* = \{a,b,c\}$ such that there are circuits $C_1 = \{e_i,z,a,b\}$ and $C_2 = \{e_j,z,b,c\}$, for some $i,j \in \{1,5\}$. If $i = j$, then circuit elimination and orthogonality with $\{e_1,e_3,e_5,z\}$ implies that $M$ has a circuit contained in $T^*$. This is a contradiction, so $i \neq j$, proving~(i).
	
    Now, $(e_1,F-\{e_1\},\{T^*\cup\{z\}\})$ is a deletion certificate. Furthermore, $\lambda(F \cup T^* \cup \{z\}) = 2$, with $z \in \cl^*(F)$ and, for all $x \in F \cup \{z\}$, we have that $x \in \cl((F \cup T^* \cup \{z\})-\{x\})$. Hence, by \Cref{deletable_collection_contractable_el}, every element of $E(M) - (F \cup T^* \cup \{z\})$ is contained in a triad. Let $e \in E(M) - (F \cup T^* \cup \{z\})$, let $T_2^*$ be a triad containing $e$, and suppose $T_2^* \cap (F \cup T^* \cup \{z\}) \neq \emptyset$. Then, by orthogonality, $T_2^* = \{z,b,e\}$. But now $(z,\{b,e\},\{F\})$ is a contraction certificate, and $\lambda(F \cup T^* \cup \{e,z\}) = 2$. This contradicts \cref{no_other_elements}, since every triangle of $M$ is contained in $F$. Therefore, $T_2^*$ is disjoint from $F \cup T^* \cup \{z\}$, establishing (ii).
\end{proof}

\begin{lemma} \label{coaugmented}
    Let $M$ be a $3$-connected matroid with no detachable pairs such that $|E(M)| \geq 13$. Let	$F = (e_1,e_2,\ldots,e_{|F|})$ be a maximal fan of $M$ with odd length at least five such that $\{e_1,e_2,e_3\}$ is a triangle and every triangle of $M$ is contained in $F$. Then $M$ is a quasi-triad-paddle with a co-augmented-fan petal.
\end{lemma}

\begin{proof}
    By \Cref{odd_fan}, we have that $|F| = 5$ and there exists $z \notin F$ such that $\{e_1,e_3,e_5,z\}$ is a cocircuit. By \Cref{coaugmented2}, there exists a triad $T_1^*$ disjoint from $F \cup \{z\}$, and by \Cref{coaugmented3}(i), we have that $T_1^* = \{a^1,b^1,c^1\}$ such that $\{e_1,z,a^1,b^1\}$ and $\{e_5,z,b^1,c^1\}$ are circuits. Let $e \notin F \cup \{z\} \cup T_1^*$. By \Cref{coaugmented3}(ii), there is a triad $T_2^*$ containing $e$, which is disjoint from $F \cup T_1^* \cup \{z\}$. By \Cref{coaugmented3}(i), $T_2^* = \{a^2,b^2,c^2\}$ such that $\{e_1,z,a^2,b^2\}$ and $\{e_5,z,b^2,c^2\}$ are circuits. Furthermore, \Cref{disjoint_triads} implies that $M | (T_1^* \cup T_2^*) \cong M(K_{2,3})$. 
	
    It follows that there is a partition $(P_1,P_2,\ldots,P_m)$ of $E(M)$, with $m \geq 3$, such that $P_1 = F \cup \{z\}$ and $M \backslash P_1 \cong M(K_{3,m-1})$ and, for all $i \in \{2,3,\ldots,m\}$, the set $P_i = \{a^i,b^i,c^i\}$ is a triad such that $\{e_1,z,a^i,b^i\}$ and $\{e_5,z,b^i,c^i\}$ are circuits, so $P_1$ is a co-augmented-fan petal affixed to $P_i$.
    By \Cref{k3m_paddle}, we have that $(P_1,P_2,\ldots,P_m)$ is a paddle of $M$, so $M$ is a quasi-triad-paddle with a co-augmented-fan petal, as required.
\end{proof}

\subsection*{Putting it together}
\begin{proof} [Proof of \Cref{single_fans}]
    Let $F$ be a maximal fan of $M$ with length at least four. If $M$ has a maximal fan $G$, distinct from $F$, with length at least four, then \cref{unlike_odd_fans} implies that $M$ has a detachable pair. So we may assume that every $4$-element fan of $M$ is contained in $F$. If $F$ has even length, then every triangle or triad of $M$ is contained in $F$, and \Cref{simple_even_fan_spike} implies that $M$ is a wheel, a whirl, or an even-fan-spike. Otherwise, $F$ is odd. Up to duality, we may assume that the ends of $F$ are contained in triangles. This means that every triangle of $M$ is contained in $F$, so the theorem follows from \Cref{coaugmented}.
\end{proof}

\section{No \texorpdfstring{$4$}{4}-element fans} \label{no_fans}
Lastly, we assume that $M$ has no $4$-element fans. In this section, we establish the following theorem, which together with Theorems~\ref{disjoint_fans_with_like_ends}, \ref{intersecting_fans_detachable}, and \ref{single_fans} completes the proof of \Cref{detachable_main}.

\begin{theorem} \label{no_4_element_fans}
	Let $M$ be a $3$-connected matroid with no $4$-element fans such that $|E(M)| \geq 13$. Then one of the following holds:
	\begin{enumerate}
		\item $M$ has a detachable pair,
		\item $M$ is a spike,
		\item $M$ or $M^*$ is a triad-paddle,
		\item $M$ or $M^*$ is a hinged triad-paddle,
		\item $M$ is a tri-paddle-copaddle, or
        \item $M$ or $M^*$ is a quasi-triad-paddle with a quad or near-quad petal.
	\end{enumerate}
\end{theorem}

\subsection*{Intersecting triads}
First, we consider the case where $M$ has two triads $T^*_1$ and $T^*_2$ with non-empty intersection. Suppose that $M$ has no detachable pairs. Using \cref{no_segment}, $|T^*_1 \cap T^*_2| = 1$.  \Cref{intersecting_plus_contractable} handles the case where there is an element~$e \notin T^*_1 \cup T^*_2$ such that $M/e$ is $3$-connected.  When there is no such element, \cref{intersecting_two_elements} handles the case where there are at least two elements not contained in a triangle or a triad. Together with \cref{intersecting_no_other_triads}, which shows that $T^*_1$ and $T^*_2$ are the only two triads of $M$, and duality, these bound $|E(M)|$.

Note that the next lemma applies even when $M$ has $4$-element fans.

\begin{lemma} \label{intersecting_triads_contractable}
    Let $M$ be a $3$-connected matroid. Let $T_1^* = \{t,a_1,a_2\}$ and $T_2^* = \{t,b_1,b_2\}$ be triads of $M$ such that $|T^*_1\cap T^*_2|=1$. Let $e \in E(M)-(T_1^* \cup T_2^*)$ such that $M / e$ is $3$-connected and $\{e,t,a_1,b_1\}$ is a circuit of $M$. Then $\si(M / a_2)$ is $3$-connected and $\si(M / b_2)$ is $3$-connected.
\end{lemma}

\begin{proof}
    We prove that $\si(M / a_2)$ is $3$-connected. The proof that $\si(M / b_2)$ is $3$-connected follows by symmetry. Suppose $\si(M / a_2)$ is not $3$-connected. Then $M$ has a vertical $3$-separation $(X,\{a_2\},Y)$. By \cref{fcl_vert_sep,fan_vert_sep}, we may assume that $T_2^* \subseteq X$ and $X \cup \{a_2\}$ is closed. Then $a_1 \in Y$, as otherwise $a_2 \in \cl^*(X)$. This further implies that $e \in Y$, as otherwise $a_1 \in \cl(X)$. Now $\lambda(X \cup \{a_1,a_2\}) = 2$. But $e \in \cl(X \cup \{a_1,a_2\})$, which contradicts the fact that $M / e$ is $3$-connected, since $|Y - \{a_1\}| \geq 2$.
\end{proof}

\begin{lemma} \label{intersecting_plus_contractable}
    Let $M$ be a $3$-connected matroid with no $4$-element fans such that $|E(M)| \geq 12$. Let $T_1^*$ and $T_2^*$ be triads of $M$ with $|T_1^* \cap T_2^*| = 1$, and let $e \in E(M)-(T_1^* \cup T_2^*)$ such that $M / e$ is $3$-connected. Then $M$ has a detachable pair.
\end{lemma}

\begin{proof}
    Suppose, towards a contradiction, that $M$ has no detachable pairs.  Let $T_1^* = \{t,a_1,a_2\}$ and $T_2^* = \{t,b_1,b_2\}$. Note that $(t,\{a_1,a_2\},\{\{b_1,b_2\}\})$ is a contraction certificate. Since $e \notin T_1^* \cup T_2^*$ and $M / e$ is $3$-connected, \Cref{contractable_collection} implies that $\lambda(T_1^* \cup T_2^*) > 2$. In particular, this means that $T_1^* \cup T_2^*$ is independent.
	
    By the dual of \cref{contractable_circuit}, there is a $4$-element circuit $C_1$ of $M$ containing $\{e,t\}$. By orthogonality, and without loss of generality, $C_1 = \{e,t,a_1,b_1\}$. By \Cref{intersecting_triads_contractable}, we have that $\si(M / a_2)$ is $3$-connected and, since $a_2$ is not contained in a triangle, $M / a_2$ is $3$-connected. This implies, by the dual of \Cref{contractable_circuit} and orthogonality, that $M$ has a $4$-element circuit $C_2$ containing $\{a_2,b_2\}$, an element of $\{t,a_1\}$, and an element of $\{t,b_1\}$. Furthermore, $C_2 \not \subseteq T_1^* \cup T_2^*$, and if $e \in C_2$, then circuit elimination between $C_1$ and $C_2$ implies that $M$ has a circuit contained in $T_1^* \cup T_2^*$. Therefore, $C_2 = \{f,t,a_2,b_2\}$ with $f \notin T_1^* \cup T_2^* \cup \{e\}$. Similarly, $M$ has a circuit $C_3$ containing $\{a_2,b_1\}$, and $C_3 = \{g,t,a_2,b_1\}$ with $g \notin T_1^* \cup T_2^* \cup \{e,f\}$. \Cref{intersecting_triads_contractable} also implies that $M / b_2$ is $3$-connected, so $M$ has a $4$-element circuit $C_4 = \{h,t,a_1,b_2\}$ with $h \notin T_1^* \cup T_2^* \cup \{e,f,g\}$.
	
    Now, $C_3 = \{g,t,a_2,b_1\}$ is a $4$-element circuit for which $\{t,b_1\}$ is contained in a triad, and $M / a_2$ is $3$-connected. \Cref{circuit_so_deletable} implies that either $g$ is contained in a triad, or $M \backslash g$ is $3$-connected. Symmetrically, either $h$ is contained in a triad or $M \backslash h$ is $3$-connected.
    
    First, suppose neither $g$ nor $h$ is contained in a triad of $M$. In $M / e$, the set $(a_2,a_1,t,b_1)$ is a fan, and $g \in \cl(\{a_2,a_1,t,b_1\})$. Since $g$ is not contained in a triad of $M / e$, \Cref{closure_deletable} implies that $M \backslash g / e$ is $3$-connected. Furthermore, $\{t,a_1,a_2\}$ and $\{t,b_1,b_2\}$ are triads of $M \backslash g$, and $\{e,t,a_1,b_1\}$ is a circuit of $M \backslash g$. Hence, by \Cref{intersecting_triads_contractable}, the matroid $M \backslash g / b_2$ is $3$-connected. Now, the element $h$ is contained in a circuit $\{h,t,a_1,b_2\}$ of $M \backslash g$ such that $M \backslash g / b_2$ is $3$-connected and $\{t,a_1\}$ is contained in a triad of $M \backslash g$. Since $M$ has no detachable pairs, \Cref{circuit_so_deletable} implies that $h$ is contained in a triad of $M \backslash g$. Since $h$ is not contained in a triad of $M$, this implies $M$ has a $4$-element cocircuit $C^*$ containing $\{g,h\}$. Orthogonality with $C_2$, $C_3$, and $C_4$ imply that $C^* \subseteq T_1^* \cup T_2^* \cup \{f,g,h\}$. But now $\lambda(T_1^* \cup T_2^* \cup \{f,g,h\}) = 2$, and $e \notin T_1^* \cup T_2^* \cup \{f,g,h\}$. This contradicts \Cref{contractable_collection}.
    
    Therefore, either $g$ or $h$ is contained in a triad of $M$. Without loss of generality, assume that $g$ is contained in a triad $T^*$. By orthogonality, $T^*$ is contained in $T_1^* \cup T_2^* \cup \{e,f,g,h\}$. If $e \notin T^*$, then $\lambda(T_1^* \cup T_2^* \cup \{f,g,h\}) = 2$, which contradicts \Cref{contractable_collection}. Hence, $e \in T^*$, so $T^* = \{g,e,b_1\}$. 
    
    Now, $\lambda(T_1^* \cup T_2^* \cup \{e,g\}) = 2$, and $f \in \cl(T_1^* \cup T_2^* \cup \{e,g\})$. We will show that $(f,T_1^* \cup T_2^* \cup \{e,g\}, \{\{g,b_1,b_2\},\{h,a_1,a_2\}\})$ is a deletion certificate. Circuit elimination between $C_2$ and $C_3$ implies that $M$ has a circuit contained in $\{f,g,a_2,b_1,b_2\}$. Orthogonality and the fact that $b_1$ and $b_2$ are not contained in triangles imply that this circuit is $\{f,g,b_1,b_2\}$. Similarly, circuit elimination between $C_2$ and $C_4$ implies that $\{f,h,a_1,a_2\}$ is a circuit of $M$. Hence, $f \in \cl(\{g,b_1,b_2\})$ and $f \in \cl(\{h,a_1,a_2\})$. Furthermore, if $f$ is contained in a triad, then this triad contains an element of $\{g,b_1,b_2\}$ and an element of $\{h,a_1,a_2\}$, in which case $f \in \cl(T_1^* \cup T_2^* \cup \{e,g,h\}) \cap \cl^*(T_1^* \cup T_2^* \cup \{e,g,h\})$, which contradicts that $\lambda(T_1^* \cup T_2^* \cup \{e,g,h\}) = 2$. Thus, $f$ is not contained in a triad, so $(f, T_1^* \cup T_2^* \cup \{e,g\}, \{\{g,b_1,b_2\},\{h,a_1,a_2\}\})$ is a deletion certificate. But now $T_1^* \cup T_2^* \cup \{e,f,g,h\}$ contains both a contraction certificate and a deletion certificate, and $\lambda(T_1^* \cup T_2^* \cup \{e,f,g,h\}) = 2$. This contradicts \Cref{no_other_elements}, since $M$ has no $4$-element fans, and completes the proof.
\end{proof}

\begin{lemma} \label{intersecting_no_other_triads}
	Let $M$ be a $3$-connected matroid with no detachable pairs and no $4$-element fans such that $|E(M)| \geq 12$. Let $T_1^*$ and $T_2^*$ be triads of $M$ such that $|T_1^* \cap T_2^*| = 1$. Then $M$ has no other triads.
\end{lemma}

\begin{proof}
    Suppose $M$ has a triad $T_3^*$ that is distinct from $T_1^*$ and $T_2^*$. If $|T_3^* - (T_1^* \cap T_2^*)| \geq 2$, then Tutte's Triangle Lemma implies that there exists $x \in T_3^* - (T_1^* \cap T_2^*)$ such that $M / x$ is $3$-connected, a contradiction to \Cref{intersecting_plus_contractable}. Thus, $|T_3^* - (T_1^* \cap T_2^*)| = 1$. By the dual of \Cref{no_segment}, we have that $|T_1^* \cap T_3^*| = 1$ and $|T_2^* \cap T_3^*| = 1$. This means that we can label the elements of $T_1^*$, $T_2^*$, and $T_3^*$ such that $T_1^* = \{a_1,b_1,a_2\}$, and $T_2^* = \{a_2,b_2,a_3\}$, and $T_3^* = \{a_3,b_3,a_1\}$. By \Cref{intersecting_plus_contractable}, none of $M / b_1$, $M / b_2$, and $M / b_3$ are $3$-connected.  So, by \cref{fcl_vert_sep,fan_vert_sep}, $M$ has a vertical $3$-separation $(X,\{b_3\},Y)$ such that $T_1^* \subseteq X$ and $X$ is coclosed. This implies $a_3 \in Y$, as otherwise $b_3 \in \cl^*(X)$, and, in turn, $b_2 \in Y$, as otherwise $a_3 \in \cl^*(X)-X$. But now $\lambda(X \cup \{b_3,a_3\}) = 2$, and $b_2 \in \cl^*(X \cup \{b_3,a_3\})$. Since no triad of $M$ intersects $T^*_2$ in two elements, $|Y - \{b_2,a_3\}| \geq 2$, and so $M / b_2$ is $3$-connected, a contradiction.
\end{proof}

The next lemma will be useful throughout this section.

\begin{lemma} \label{annoying_elements}
    Let $M$ be a $3$-connected matroid with no detachable pairs. Suppose that, for all $x \in E(M)$, if $x$ is not contained in a triad, then $M / x$ is not $3$-connected. Suppose there exist distinct $e, f \in E(M)$ such that neither $e$ nor $f$ is contained in a triangle or a triad. Then
    \begin{enumerate}
        \item there is a $4$-element cocircuit~$C^*$ containing $\{e,f\}$, and
        \item there is a triad $T^*$ such that $T^* \cap C^* = \{g\}$ for some $g \in E(M)-\{e,f\}$, and $M / g$ is $3$-connected.
    \end{enumerate}
\end{lemma}

\begin{proof}
    Towards a contradiction, suppose $\{e,f\}$ is not contained in a $4$-element cocircuit of $M$. Neither $M / e$ nor $M / f$ is $3$-connected, so Bixby's Lemma implies that $M \backslash e$ and $M \backslash f$ are $3$-connected. Since $f$ is not contained in a triangle or triad of $M$, and $\{e,f\}$ is not contained in a $4$-element cocircuit of $M$, we have that $f$ is not contained in a triangle or triad of $M \backslash e$. Hence, as $M \backslash e \backslash f$ is not $3$-connected, we have that $M \backslash e / f$ is $3$-connected. But then the dual of \cref{contract_then_delete} implies that $M / f$ is $3$-connected, a contradiction. Therefore, $M$ has a $4$-element cocircuit $C^*$ containing $\{e,f\}$.
	
	Now, $M / e$ is not $3$-connected, so $M$ has a vertical $3$-separation $(X, \{e\}, Y)$. If $|C^* \cap X| = 3$, then $e \in \cl^*(X)$, a contradiction. Likewise, $|C^* \cap Y| \neq 3$. Hence, without loss of generality, we may assume that $|C^* \cap X| = 2$ and $|C^* \cap Y| = 1$.
    Let $g$ be the unique element of $C^* \cap Y$. Then $g \in \cl^*(X \cup \{e\})$, so $\co(M \backslash g)$ is not $3$-connected. Thus, $\si(M / g)$ is $3$-connected. Suppose $g$ is contained in a triangle $T$.
    Now, by orthogonality and since neither $e$ nor $f$ is contained in a triangle, $C^*=\{e,f,g,h\}$ with $h \in T$. But now the dual of \cref{circuit_so_deletable} implies that $M / f$ is $3$-connected, a contradiction. This means that $M / g$ is $3$-connected, and so $g$ is contained in a triad $T^*$, and $g \notin \{e,f\}$.
	
    Let $C^*=\{e,f,g,h\}$ and suppose $|T^* \cap C^*| \ge 2$. Then $T^* \cap C^* = \{g,h\}$, so that $|T^* \cap X| \geq 1$. If $g \in \cl^*(X)$, then, by \Cref{fcl_vert_sep}, $(X \cup \{g\},\{e\},Y-\{g\})$ is a vertical $3$-separation of $M$, and $e \in \cl^*(X \cup \{g\})$, a contradiction. Thus, $g \notin \cl^*(X)$, so $|T^* \cap Y| = 2$ and $T^* \cap X = \{h\}$. This means that $(X-\{h\},\{e\},Y \cup \{h\})$ is a vertical $3$-separation. But $f \in \cl^*(Y \cup \{h,e\})$, so $M \backslash f$ is not $3$-connected, a contradiction. We deduce that $T^* \cap C^* = \{g\}$, which completes the proof of the lemma.
\end{proof}

\begin{lemma} \label{intersecting_two_elements}
    Let $M$ be a $3$-connected matroid with no $4$-element fans such that $|E(M)| \geq 12$. Suppose that
    \begin{enumerate}
        \item for all $x \in E(M)$, if $x$ is not contained in a triad, then $M / x$ is not $3$-connected,
        \item $M$ has triads $T_1^*$ and $T_2^*$ with $|T_1^* \cap T_2^*| = 1$, and
        \item there are distinct elements $e,f \in E(M)$, each of which is not contained in a triangle or a triad.
    \end{enumerate}
    Then $M$ has a detachable pair.
\end{lemma}

\begin{proof}
    Suppose $M$ has no detachable pairs. By \Cref{annoying_elements}, there exists a $4$-element cocircuit $C^*$ containing $\{e,f\}$, and a triad $T^*$ such that $C^* \cap T^* = \{g\}$, for some $g \notin \{e,f\}$, where $M / g$ is $3$-connected. By \Cref{intersecting_no_other_triads}, we have that $T^* = T_1^*$ or $T^* = T_2^*$. Without loss of generality, take $T^* = T_1^*$. Let $T_1^* = \{t,a_1,a_2\}$ and $T_2^* = \{t,b_1,b_2\}$. If $g = t$, then, since either $M / a_1$ or $M / a_2$ is $3$-connected by Tutte's Triangle Lemma, the dual of \cref{contractable_circuit} implies that $M$ has a $4$-element circuit containing $t$ and either $a_1$ or $a_2$. If $g\neq t$, then, since $M / g$ is $3$-connected, the dual of \cref{contractable_circuit} implies that $M$ has a $4$-element circuit containing $\{g,t\}$. In either case, by orthogonality $M$ has a $4$-element circuit $C = \{a_i,b_j,t,h\}$ for some element $h \in E(M)$ and $i,j \in \{1,2\}$ such that $g \in C$. Since $g \in C \cap C^*$, orthogonality implies that $|C \cap C^*| \geq 2$.
	
    Neither $M / e$ nor $M / f$ is $3$-connected, so $M$ has a vertical $3$-separation $(X,\{e\},Y)$. We may assume that $h \neq e$, for if $h = e$, then we can instead apply the argument that follows to a vertical $3$-separation $(X',\{f\},Y')$. We show that there is such a vertical $3$-separation in which $T_1^* \cup T_2^* \cup \{h\} \subseteq X$. By \Cref{fan_vert_sep}, we may assume that $T_1^* \subseteq X$. Furthermore, by \Cref{fcl_vert_sep}, we may assume that $|X\cap \{b_1, b_2\}|\neq1$. If $\{b_1,b_2\} \subseteq X$, then $h \in \cl(X)$, and the desired outcome follows. So assume $\{b_1,b_2\} \subseteq Y$. If $h \in X$, then $b_j \in \cl(X)$ and the desired outcome follows; whereas if $h \in Y$, then $t \in \cl^*(Y)$ and $a_i \in \cl(Y \cup \{t\})$, so, after interchanging the roles of $X$ and $Y$, we again obtain the desired outcome. Thus, we may assume that $T_1^* \cup T_2^* \cup \{h\} \subseteq X$.
	
	Now, $C \subseteq X$, so $|C^* \cap X| \geq 2$. If $|C^* \cap X| = 3$, then $e \in \cl^*(X)$, a contradiction. So $|C^* \cap X| = 2$, and there exists a unique element $y$ in $C^* \cap Y$. But $y \in \cl^*(X \cup \{e\})$, and $y$ is not contained in a triangle since such a triangle would contain a second element of $C^*$ and none of $e$, $f$, or $g$ are contained in a triangle. This means that $M / y$ is $3$-connected. However, now $y$ is contained in a triad distinct from $T_1^*$ and $T_2^*$, a contradiction to \Cref{intersecting_no_other_triads}, thereby completing the proof of the lemma.
\end{proof}

\subsection*{Disjoint triads}
We next move on to the case where $M$ has two disjoint triads. 
When $M$ has an element $e$, not contained in a triad, such that $M/e$ is $3$-connected, the case is handled by \cref{quad_petal}.  When no such element~$e$ exists, but there is some element that is not in a triangle or triad, this is handled by \cref{not_quad_petal}.  Finally, \cref{k3m} handles the case where every element of $M$ is in either a triangle or a triad.

\begin{lemma} \label{quad_thingy_deletable}
	Let $M$ be a $3$-connected matroid. Let $T^* = \{a_1,a_2,a_3\}$ be a triad of $M$, and let $e,f,g,h$ be distinct elements of $E(M) - T^*$ such that $\{a_1,a_2,e,f\}$ and $\{a_2,a_3,e,g\}$ are circuits, and $\{e,f,g,h\}$ is a cocircuit, and $h$ is not contained in a triangle. Then $M / h$ is $3$-connected.
\end{lemma}

\begin{proof}
	Suppose $M / h$ is not $3$-connected. Then $M$ has a vertical $3$-separation $(X,\{h\},Y)$ such that $T^* \subseteq X$. If $\{e,f,g\} \cap X \neq \emptyset$, then $\{e,f,g\} \subseteq \cl(X)$, so $(X \cup \{e,f,g\},\{h\},Y-\{e,f,g\})$ is a vertical $3$-separation. However, $h \in \cl^*(X \cup \{e,f,g\})$, a contradiction. Otherwise, $\{e,f,g\} \subseteq Y$, which means that $h \in \cl^*(Y)$, another contradiction. Therefore, $M / h$ is $3$-connected.
\end{proof}

\begin{lemma} \label{quad_petal1}
	Let $M$ be a $3$-connected matroid with no detachable pairs and no $4$-element fans such that $|E(M)| \geq 13$. Let $T_1^*$ and $T_2^*$ be disjoint triads of $M$, and let $e$ be an element of $M$ such that $e$ is not contained in a triangle or a triad and $M / e$ is $3$-connected. Then
	\begin{enumerate}
		\item there exists $X \subseteq E(M)$ such that $e \in X$ and $X$ is either a quad or near-quad affixed to $T_1^*$ and $T_2^*$, and
		\item every element of $E(M) - X$ is contained in a triad.
	\end{enumerate}
\end{lemma}

\begin{proof}
    By \Cref{disjoint_triads}, $M|(T_1^* \cup T_2^*) \cong M(K_{2,3})$, so we may assume that $T_1^* = \{a_1,a_2,a_3\}$ and $T_2^* = \{b_1,b_2,b_3\}$ such that, for all distinct $i,j \in \{1,2,3\}$, the set $\{a_i,a_j,b_i,b_j\}$ is a circuit. Furthermore, the dual of \cref{contractable_circuit} implies that $M$ has a $4$-element circuit $C_1$ containing $\{e,a_1\}$, $a_2$ or $a_3$, and some $f \notin T_1^*$. By orthogonality, and without loss of generality, $C_1 = \{a_1,a_2,e,f\}$ with $f \notin T_1^* \cup T_2^*$. Similarly, $M$ has a $4$-element circuit $C_2$ containing $\{a_2,e\}$, either $a_1$ or $a_3$, and some $g \notin T^*_1\cup T_2^*$. By circuit elimination, orthogonality, and since $e$ is not in a triangle, $C_2 = \{a_2,a_3,e,g\}$. Note that $g \neq f$, for otherwise $e \in \cl(T^*_1)$ by circuit elimination, which contradicts the fact that $M/e$ is $3$-connected.
	
    Suppose $f$ is contained in a triad. Since $e$ is not contained in a triad, orthogonality implies this triad is $\{a_1,b_1,f\}$. But this contradicts \Cref{intersecting_no_other_triads}. Thus, $f$ (and similarly $g$) is not contained in a triad of $M$. Now, $M \backslash f$ (and similarly $M \backslash g$) is $3$-connected by \Cref{circuit_so_deletable}. By \Cref{deletable_circuit_gives_cocircuit}, there is a $4$-element cocircuit $C^*$ of $M$ containing either $\{e,f\}$ or $\{f,g\}$. We prove that there is a $4$-element cocircuit of $M$ containing $\{f,g\}$, so suppose that $\{e,f\} \subseteq C^*$. If $C^*$ also contains $g$, then we have the desired result. Otherwise, orthogonality with $C_2$ implies that either $a_2 \in C^*$ or $a_3 \in C^*$. It follows that $C^* = \{a_i,b_i,e,f\}$ for some $i \in \{2,3\}$, so $\lambda(T_1^* \cup T_2^* \cup \{e,f\}) = 2$. In particular, as $M\backslash f$ is $3$-connected, $\lambda_{M \backslash f}(T_1^* \cup T_2^* \cup \{e\}) = 2$ and $g \in \cl(T_1^* \cup T_2^* \cup \{e\})$. Thus, since $M \backslash f \backslash g$ is not $3$-connected, the element $g$ is contained in a triad of $M \backslash f$, and thus $\{f,g\}$ is contained in a $4$-element cocircuit of $M$.
	
    In all cases, there is a $4$-element cocircuit of $M$ containing $\{f,g\}$. Suppose $e$ is not contained in this cocircuit. Then orthogonality with $C_1$ and $C_2$ implies that $M$ has a cocircuit $\{a_2,b_2,f,g\}$. But now $\lambda(T_1^* \cup T_2^* \cup \{f,g\}) = 2$ and $(a_2,\{a_1,a_3\},\{\{b_2,f,g\}\})$ is a contraction certificate. By \Cref{contractable_collection}, this is a contradiction, since $M / e$ is $3$-connected. It follows that $M$ has a cocircuit $\{e,f,g,h\}$ with $h \notin T_1^* \cup T_2^* \cup \{e,f,g\}$. Let $X=\{e,f,g,h\}$.
	
    First assume that $h$ is contained in a triangle $T$. By orthogonality, and since $e$ is not contained in a triangle, $T$ contains an element of $\{f,g\}$. Suppose $T$ contains exactly one of $f$ and $g$. Say $f \in T$ but $g \notin T$, so that $T = \{f,h,x\}$ for some $x \notin T_1^* \cup T_2^* \cup X$. Then \Cref{deletable_cocircuit} implies that $M$ has a $4$-element cocircuit $D^*$ containing $\{g,x\}$ and an element of $\{f,h\}$. By orthogonality with $C_2$, we have that $e \in D^*$, so either $D^* = \{e,f,g,x\}$ or $D^* = \{e,g,h,x\}$. But, in both cases, cocircuit elimination with $X$ implies that $M$ has a cocircuit contained in $\{f,g,h,x\}$, which is a contradiction to orthogonality. It follows by symmetry that $T = \{f,g,h\}$. Due to the cocircuit $X$, the triangle $T$, and the circuits $C_1$ and $C_2$, it now follows that $X$ is a near-quad affixed to $T_1^*$.  Furthermore, by circuit elimination and orthogonality, $\{b_1,b_2,e,f\}$ and $\{b_2,b_3,e,g\}$ are circuits, so $X$ is also a near-quad affixed to $T_2^*$. Now, $(f,\{g,h\},\{\{e,a_1,a_2\}\})$ is a deletion certificate, and $\lambda(T_1^* \cup X) = 2$. Additionally, $e \in \cl^*(\{f,g,h\})$. Thus, by \Cref{deletable_collection_contractable_el}, every element of $E(M) - (T_1^* \cup X)$ is contained in a triad, and so the lemma holds when $h$ is contained in a triangle.
	
    Now assume that $h$ is not contained in a triangle. By \Cref{quad_thingy_deletable}, the matroid $M / h$ is $3$-connected. Also, by the dual of \cref{contractable_circuit}, $M$ has a $4$-element circuit containing $\{a_2,h\}$, an element of $\{a_1,a_3\}$, and, by orthogonality, an element of $\{e,f,g\}$. Circuit elimination with either $\{a_1,a_2,e,f\}$ or $\{a_2,a_3,e,g\}$ implies that $X$ is a quad.
	
    If either $f$ or $g$ is contained in a triangle $T$, then $T=\{f, g, z\}$, where $z\not\in T^*_1\cup T^*_2\cup \{e, f, g, h\}$. By \Cref{quad_detachable}, $M\backslash e$ is $3$-connected, so, by \cref{deletable_cocircuit}, $M$ has a $4$-element cocircuit containing $\{z, e\}$, an element in $\{f,g\}$, and an element not in $\{f,g\}$, but this contradicts orthogonality. Therefore neither $f$ nor $g$ is contained in a triangle.
    Let $x$ be an arbitrary element of the quad $X$.  We have that $M / x$ is $3$-connected, by the dual of \Cref{quad_detachable}. Hence, by the dual of \cref{contractable_circuit} and orthogonality, $M$ has a $4$-element circuit $C'$ containing $\{a_1,x\}$, an element of $\{a_2,a_3\}$, and an element $x' \in X - \{x\}$. Similarly, $M$ has a $4$-element circuit containing $x$ and the unique element of $T_1^* - C'$, and another element of $T_1^*$, and an element $x'' \in X - \{x\}$. Note that $x' \neq x''$ since $x \notin \cl(T_1^*)$.  It follows that $X$ is a quad affixed to $T_1^*$ and, similarly, $X$ is a quad affixed to $T_2^*$.
    Now, $$(e,\{f,g,h\},\{\{f,a_1,a_2\},\{g,a_2,a_3\}\})$$ is a deletion certificate, and $a_1 \in \cl^*(\{a_2,a_3\})$. By \Cref{deletable_collection_contractable_el}, every element of $E(M) - (T_1^* \cup X)$ is contained in a triad, so the lemma also holds when $h$ is not contained in a triangle.
\end{proof}

\begin{lemma} \label{quad_petal}
    Let $M$ be a $3$-connected matroid with no detachable pairs and no $4$-element fans such that $|E(M)| \geq 13$. Let $T_1^*$ and $T_2^*$ be disjoint triads of $M$, and let $e$ be an element of $M$ such that $e$ is not contained in a triad, and $M / e$ is $3$-connected. Then $M$ is a quasi-triad-paddle with a quad or near-quad petal.
\end{lemma}

\begin{proof}
    By \Cref{quad_petal1}, there exists $X \subseteq E(M)$ such that $X$ is a quad or near-quad affixed to $T_1^*$ and $T_2^*$, and, for all $x \notin X \cup T_1^* \cup T_2^*$, the element $x$ is contained in a triad $T^*$. By orthogonality, $T^*$ is disjoint from $X \cup T_1^* \cup T_2^*$. Hence, by another application of \cref{quad_petal1}, $X$ is a quad or near-quad (respectively) affixed to $T^*$, and \cref{disjoint_triads} implies that $M | (T_1^* \cup T_2^* \cup T^*) \cong M(K_{3,3})$. It follows that $E(M)$ can be partitioned into $P_1,P_2,P_3,\ldots,P_m$ such that $P_1 = X$ and $M \backslash P_1 \cong M(K_{3,m-1})$ and, for all $i \in \{2,3,\ldots,m\}$, the set $P_i$ is a triad and $X$ is a quad or near-quad (respectively) affixed to $P_i$. \Cref{k3m_paddle} implies that $(P_1,P_2,\ldots,P_m)$ is a paddle of $M$, so $M$ is a quasi-triad-paddle, as required.
\end{proof}

\begin{lemma} \label{not_quad_petal1}
    Let $M$ be a $3$-connected matroid with no $4$-element fans such that $|E(M)| \geq 12$. Suppose that, for all $x \in E(M)$, if $x$ is not contained in a triad, then $M / x$ is not $3$-connected. Let $T_1^*$ and $T_2^*$ be disjoint triads of $M$, and let $e$ and $f$ be distinct elements of $E(M)$ that are not contained in a triangle or a triad. Then $M$ has a detachable pair.
\end{lemma}

\begin{proof}
    Suppose that $M$ has no detachable pairs. By \Cref{disjoint_triads}, we may assume that $T_1^* = \{a_1,a_2,a_3\}$ and $T_2^* = \{b_1,b_2,b_3\}$ such that, for all distinct $i,j \in \{1,2,3\}$, the set $\{a_i,a_j,b_i,b_j\}$ is a circuit. By \Cref{annoying_elements}, there exists a $4$-element cocircuit $C^*$ containing $\{e,f\}$, and there exists a triad $T^*$ such that $C^* \cap T^* = \{g\}$, where $M / g$ is $3$-connected.
	
    Suppose that $C^*$ and $T_1^* \cup T_2^*$ are disjoint. This means that $g \notin T_1^* \cup T_2^*$, so $T^* \neq T_1^*$ and $T^* \neq T_2^*$. Therefore, \cref{no_segment,intersecting_no_other_triads} imply that $T^*$ is disjoint from $T_1^*$ and $T_2^*$, so, by \cref{disjoint_triads}, $M | (T^* \cup T_1^* \cup T_2^*) \cong M(K_{3,3})$. In particular, $M$ has a $4$-element $C$ containing $g$, another element of $T^*$, and two elements of $T_1^*$. But $C^* \cap T^* = \{g\}$ and $C^* \cap T_1^* = \emptyset$, so $|C \cap C^*|=1$, a contradiction to orthogonality. 
	
    Thus $C^* \cap (T_1^* \cup T_2^*) \neq \emptyset$. Orthogonality implies that $C^* = \{a_i,b_i,e,f\}$, for some $i \in \{1,2,3\}$. Since $e$ is not contained in a triangle and $M / e$ is not $3$-connected, $M$ has a vertical $3$-separation $(X, \{e\}, Y)$. We may assume, by \cref{fcl_vert_sep}, that $T_1^* \subseteq X$. If $T_2^* \subseteq X$, then either $f \in X$, which means $e \in \cl^*(X)$, and so $M \backslash e$ is not $3$-connected, or $f \in Y$, which means $f \in \cl^*(X \cup \{e\})$, and so $M \backslash f$ is not $3$-connected. Either case is a contradiction, so, by \Cref{fcl_vert_sep}, $T_2^* \subseteq Y$. Now, either $f \in X$ or $f \in Y$. We may assume, without loss of generality, the former. It follows that $\lambda(X \cup \{e\} \cup T_2^*) < 2$, which implies that $|Y - T_2^*| = 1$. Hence, $Y = T_2^* \cup \{z\}$, for some element $z \in E(M) - (T_1^* \cup T_2^* \cup \{e,f\})$. Since $\lambda(Y) = 2$, either $z \in \cl(T_2^*)$ or $z \in \cl^*(T_2^*)$. If $z \in \cl(T_2^*)$, then $T_2^* \cup \{z\}$ is a $4$-element circuit, which contradicts orthogonality with $C^*$. Otherwise, $r^*(T_2^* \cup \{z\}) = 2$, contradicting \Cref{no_segment}. We deduce that $M$ has a detachable pair.
\end{proof}

\begin{lemma} \label{not_quad_petal2}
	Let $M$ be a $3$-connected matroid with no detachable pairs and no $4$-element fans such that $|E(M)| \geq 12$. Let $T_1^*$ and $T_2^*$ be disjoint triads of $M$, let $T$ be a triangle of $M$, and let $e$ be an element of $E(M)$ that is not contained in a triangle or a triad. Then there exists $f \in E(M)$ such that $f$ is not contained in a triangle or a triad and $M / f$ is $3$-connected.
\end{lemma}

\begin{proof}
	If $M / e$ is $3$-connected, then clearly the result holds. Therefore suppose that $M \backslash e$ is $3$-connected. By \Cref{deletable_cocircuit}, there is a $4$-element cocircuit $C^* = \{e,f,g,h\}$ such that $\{g,h\} \subseteq T$ and $f \notin T$.
	
    Suppose $f$ is contained in a triangle $T'$. By orthogonality, $T'$ contains an element of $\{e,g,h\}$. Furthermore, $e$ is not contained in a triangle, so $|T \cap T'| \in \{1,2\}$. But $|T \cap T'| \neq 2$ by \cref{no_segment}, and, as $M \backslash e$ is $3$-connected, $|T \cap T'| \neq 1$ by the dual of \Cref{intersecting_plus_contractable}. So $f$ is not contained in a triangle.

    Next, suppose $f$ is contained in a triad $T^*$. If $T^*$ meets $T_1^*$, then, by \cref{no_segment,intersecting_no_other_triads}, we have that $T^* = T_1^*$. Similarly, if $T^*$ meets $T_2^*$, then $T^* = T_2^*$. This means that $T^*$ is disjoint from at least one of $T_1^*$ and $T_2^*$. By \Cref{disjoint_triads}, there is a $4$-element circuit $C$ of $M$ containing $f$ and another element of $T^*$ and two elements of either $T_1^*$ or $T_2^*$. But $e$ is not contained in a triad, and $g$ and $h$ are not contained in triads since $M$ has no $4$-element fans. Therefore, $C$ intersects $C^*$ in one element, a contradiction. It now follows that $f$ is contained in neither a triangle nor a triad. \Cref{circuit_so_deletable} implies that $M / f$ is $3$-connected, as desired.
\end{proof}

\begin{lemma} \label{not_quad_petal}
    Let $M$ be a $3$-connected matroid with no detachable pairs and no $4$-element fans such that $|E(M)| \geq 12$. Suppose that, for all $x \in E(M)$, if $x$ is not contained in a triad, then $M / x$ is not $3$-connected. Let $T_1^*$ and $T_2^*$ be disjoint triads of $M$, and let $e$ be an element of $M$ that is not contained in a triangle or a triad. Then $M$ is a hinged triad-paddle.
\end{lemma}

\begin{proof}
    By \Cref{not_quad_petal2}, the matroid $M$ has no triangles. If there exists $f \neq e$ such that $f$ is not contained in a triangle or a triad, then \Cref{not_quad_petal1} implies that $M$ has a detachable pair. So every element of $E(M) - \{e\}$ is contained in a triad. Furthermore, by \cref{no_segment,intersecting_no_other_triads}, there are no distinct triads of $M$ with a non-empty intersection. Therefore, by \cref{disjoint_triads}, $M \backslash e \cong M(K_{3,m})$ and $E(M) - \{e\}$ has a partition $(P_1,P_2,\ldots,P_m)$ such that, for all $i \in [m]$, the set $P_i$ is a triad. Additionally, for all $i \in [m]$, we have $M \backslash (P_i \cup \{e\}) \cong M(K_{3,m-1})$; therefore, $\lambda(E(M) - (P_i \cup \{e\})) = 2$, so $\lambda(P_i \cup \{e\}) = 2$. By \Cref{no_segment}, we have that $e \notin \cl^*(P_i)$, so $e \in \cl(P_i)$, for each $i \in [m]$. It follows that $M$ is a hinged triad-paddle.
\end{proof}

\begin{lemma} \label{k3m}
    Let $M$ be a $3$-connected matroid with no detachable pairs and no $4$-element fans, such that $|E(M)| \geq 12$. Let $T_1^*$ and $T_2^*$ be disjoint triads of $M$, and suppose that every element of $M$ is contained in a triangle or a triad. Then $M$ is either a triad-paddle or a tri-paddle-copaddle.
\end{lemma}

\begin{proof}
    By the dual of \cref{no_segment}, any two triads of $M$ intersect in at most one element.  Moreover, by \cref{intersecting_no_other_triads} and since $M$ has two disjoint triads, it follows that all the triads in $M$ are pairwise disjoint. By \cref{disjoint_triads}, we may assume that $T_1^* = \{a_1,a_2,a_3\}$ and $T_2^* = \{b_1,b_2,b_3\}$ such that, for all distinct $i,j \in \{1,2,3\}$, the set $\{a_i,a_j,b_i,b_j\}$ is a circuit. Moreover, if $X$ is precisely the set of elements of $M$ contained in a triad, then $M | X \cong M(K_{3,s})$ for some $s \geq 2$. If $E(M) = X$, then $M \cong M(K_{3,m})$, so $M$ is a triad-paddle as required. Therefore, suppose there exists a triangle~$T$ disjoint from $X$. 
	
    We first consider the case where $E(M) = X \cup T$. Now $\lambda(X - T_1^*) = 2$, and so $\lambda(T_1^* \cup T) = 2$. Suppose there exists $z \in T$ such that $z \in \cl^*(T_1^* \cup (T-\{z\}))$. Then there is a cocircuit $C^*$ of $M$ contained in $T_1^* \cup T$ that contains $z$ and an element of $T_1^*$. But orthogonality with the circuits of the form $\{a_i,a_j,b_i,b_j\}$ implies that $T_1^* \subseteq C^*$, a contradiction. Since $\lambda(T_1^* \cup T) = 2$, it follows that $x \in \cl(T_1^*)$ for all $x \in T$.
    In particular, there exist distinct elements $y,z \in T$ such that, by Tutte's Triangle Lemma and \cref{deletable_cocircuit}, there is a $4$-element cocircuit~$C^*$ of $M$ containing $\{y,z\}$ and an element in $T_1^*$.  By orthogonality with circuits of the form $\{a_i,a_j,b_i,b_j\}$, we have $C^*=\{y,z,a_i,b_i\}$ for some $i \in [3]$.
    Since $|E(M)| \geq 10$, there exists a triad $\{c_1,c_2,c_3\} \subseteq X$, distinct from $T_1^*$ and $T_2^*$, such that $\{a_i,a_j,c_i,c_j\}$ is a circuit for any $j \in [3]-\{i\}$, a contradiction to orthogonality.
	
    It now follows that $M$ has triangle $T'$ distinct from $T$. Suppose $T$ meets $T'$. By \cref{no_segment} and the dual of \cref{intersecting_no_other_triads}, $|T \cap T'| = 1$ and there are no other elements of $M$ contained in a triangle, so $E(M) = X \cup T \cup T'$. Tutte's Triangle Lemma implies that there exists an element $x \in T - T'$ such that $M \backslash x$ is $3$-connected. So, for $y \in T'$, \Cref{deletable_cocircuit} implies that there is a $4$-element cocircuit $C_1^*$ of $M$ containing $\{x,y\}$ and an element in $T'-\{y\}$. Orthogonality implies that $C_1^* \subseteq T \cup T'$. 
    
    Let $z$ be the unique element of $(T \cup T') - C_1^*$. Note that either $C_1^*$ contains a triangle or $C_1^*$ is a quad. Hence, $\lambda(C_1^*) = 2$. Furthermore, $z \in \cl(C_1^*)$. If $z \in T'$, then, as before, \Cref{deletable_cocircuit} implies that there is a $4$-element cocircuit of $M$ which contains $\{x,z\}$. Furthermore, by orthogonality, this cocircuit is a subset of $T \cup T'$. But now $z \in \cl(C_1^*)$ and $z \in \cl^*(C_1^*)$, contradicting the $3$-connectivity of $M$. Hence, $z \in T - T'$. Since $z \in \cl(C_1^*)$, the matroid $M \backslash z$ is $3$-connected. \Cref{deletable_cocircuit} implies that there is a $4$-element cocircuit of $M$ which contains $\{y,z\}$, and this cocircuit is a subset of $T \cup T'$. Again, $z \in \cl^*(C_1^*)$, a contradiction.
	
    So any two triangles of $M$ are disjoint.  Thus $E(M) - X$ can be partitioned into disjoint triangles, and, by the dual of \cref{disjoint_triads}, we have that $M / X \cong M^*(K_{3,t})$, for some $t \geq 2$. Therefore $M$ is a tri-paddle-copaddle.
\end{proof}

\subsection*{One triad and at most one triangle}
The final case we need to consider is when $M$ has exactly one triad and at most one triangle.  The case where $M$ has one triangle is handled in \cref{triad_triangle}, whereas the case where $M$ has no triangles is handled in \cref{one_triad}.

\begin{lemma} \label{triad_triangle1}
	Let $M$ be a $3$-connected matroid with no detachable pairs and no $4$-element fans such that $|E(M)| \geq 11$. Let $T^*$ be a triad of $M$ and let $T$ be a triangle of $M$ such that $M$ has no other triads or triangles. Let $e \in E(M) - (T \cup T^*)$ such that $M / e$ is $3$-connected. Then there is a labelling $T^* = \{a_1,a_2,a_3\}$ and $T = \{b_1,b_2,b_3\}$ such that $\{a_1,a_2,e,b_1\}$ and $\{a_2,a_3,e,b_3\}$ are circuits.
\end{lemma}

\begin{proof}
By the dual of \cref{contractable_circuit}, there is a labelling $T^*=\{a_1, a_2, a_3\}$ such that $M$ has $4$-element circuits $\{a_1,a_2,e,f\}$ and $\{a_2,a_3,e,g\}$ for some $f,g \notin T^* \cup \{e\}$.
Note that $f \neq g$, for otherwise, by circuit elimination, $e \in \cl(T^*)$, which contradicts that $M/e$ is $3$-connected.
Now, $f$ and $g$ are not contained in triads, so, by \Cref{circuit_so_deletable}, we have that $M \backslash f$ and $M \backslash g$ are both $3$-connected.

Suppose $f \notin T$. Then, by \Cref{deletable_cocircuit}, there is a labelling $T=\{b_1, b_2, b_3\}$ such that $M$ has a $4$-element cocircuit $C_1^*$ containing $\{b_1,b_2,f\}$ and a $4$-element cocircuit $C_2^*$ containing $\{b_2,b_3,f\}$. Orthogonality implies that $C_1^*$ and $C_2^*$ each contain an element of $\{a_1,a_2,e\}$. If $g \in T$, then $\lambda(T^* \cup T \cup \{e,f\}) = 2$ and $(g,T-\{g\},\{\{a_2,a_3,e\}\})$ is a deletion certificate. But $a_1 \in \cl^*(\{a_2,a_3\})$, and, for all $i \in \{1,2,3\}$, we have that $a_i \in \cl((T^*-\{a_i\}) \cup T \cup \{e,f\})$. This contradicts \Cref{deletable_collection_contractable_el}. We deduce that $g \notin T$, so orthogonality with $\{a_2,a_3,e,g\}$ implies that $C_1^* = \{b_1,b_2,f,a_1\}$ and $C_2^* = \{b_2,b_3,f,a_1\}$. Cocircuit elimination implies that $M$ has a cocircuit contained in $\{b_1,b_2,b_3,f\}$ and so, by orthogonality, $M$ has a cocircuit contained in $\{b_1,b_2,b_3\}$. This contradiction implies that $f \in T$ and, similarly, $g \in T$. The lemma now follows.
\end{proof}

\begin{lemma} \label{triad_triangle_new}
	Let $M$ be a $3$-connected matroid with no detachable pairs and no $4$-element fans such that $|E(M)| \geq 11$. Let $T$ be a triangle of $M$, and let $T^*$ be a triad of $M$ such that $M$ has no other triangles or triads.
    There is at most element $e \in E(M) - (T \cup T^*)$ such that $M \backslash e$ is not $3$-connected.
\end{lemma}

\begin{proof}
	Let $e \in E(M) - (T \cup T^*)$ such that $M \backslash e$ is not $3$-connected. By Bixby's Lemma, $M / e$ is $3$-connected. Hence, by \Cref{triad_triangle1}, we may assume that $T^* = \{a_1,a_2,a_3\}$ and $T = \{b_1,b_2,b_3\}$ such that $\{a_1,a_2,e,b_1\}$ and $\{a_2,a_3,e,b_3\}$ are circuits. Now suppose there exists $f \in E(M) - (T \cup T^* \cup \{e\})$ such that $M \backslash f$ is not $3$-connected. This means that $M$ has a cyclic $3$-separation $(X, \{f\}, Y)$. By \Cref{fan_vert_sep}, we may assume that $T^* \subseteq X$. Furthermore, $M / f$ is $3$-connected, so \Cref{triad_triangle1} implies that $T \subseteq \cl(T^* \cup \{f\})$. If $T \subseteq \cl(X)$, then $f \in \cl(X)$, a contradiction. Therefore, we may assume $T \subseteq Y$. Since $T \not\subseteq \cl(X)$, we have that $e \in Y$. But $e \in \cl(T^* \cup T) \subseteq \cl(T^* \cup \{f\}) \subseteq \cl(X \cup \{f\})$. This contradicts the fact that $M / e$ is $3$-connected, and completes the proof.
\end{proof}

\begin{lemma} \label{triad_triangle}
	Let $M$ be a $3$-connected matroid with no $4$-element fans such that $|E(M)| \geq 11$. Let $T$ be a triangle of $M$ and let $T^*$ be a triad of $M$ such that $M$ has no other triangles or triads. Then $M$ has a detachable pair.
\end{lemma}

\begin{proof}
	By \Cref{triad_triangle_new}, there is at most one element $e \in E(M) - (T \cup T^*)$ such that $M \backslash e$ is not $3$-connected. Dually, there is at most one element $f \in E(M) - (T \cup T^*)$ such that $M / f$ is not $3$-connected. Therefore, there exists $g \in E(M) - (T \cup T^*)$ such that both $M / g$ and $M \backslash g$ are $3$-connected. By \Cref{triad_triangle1}, we have that $T \in \cl(T^* \cup \{g\})$, and, by the dual of \Cref{triad_triangle1}, we have that $T^* \in \cl^*(T \cup \{g\})$. This means that $\lambda(T \cup T^* \cup \{g\}) = 2$. Let $a \in T^*$, and let $b \in T$. Now, $(a,T^*-\{a\},\{T \cup \{g\}\})$ is a contraction certificate, and $(b,T-\{b\},\{T^* \cup \{g\}\})$ is a deletion certificate. This contradicts \Cref{no_other_elements}, and the lemma follows.
\end{proof}

\begin{lemma} \label{one_triad1}
    Let $M$ be a $3$-connected matroid with no triangles. Let $T^* = \{a_1,a_2,a_3\}$ be a triad of $M$ such that $M$ has no other triads, and let $e$ and $f$ be distinct elements of $E(M) - T^*$ such that $\{a_1,a_2,e,f\}$ is a circuit. Suppose there exists a set $X$ with $T^* \cup \{e,f\} \subseteq X \subseteq E(M)$, such that $\lambda(X) = 2$, the set $X$ contains a contraction certificate, and $|E(M)| \geq |X| + 3$. Then $M$ has a detachable pair.
\end{lemma}

\begin{proof}
    Note that $|E(M)| \geq |X| + 4$, as otherwise $E(M)-X$ is a triangle or a triad. Now, suppose $M$ has no detachable pairs, and let $x \notin X$. By \Cref{contractable_collection}, the matroid $M / x$ is not $3$-connected, so $M \backslash x$ is $3$-connected. \Cref{deletable_circuit_gives_cocircuit} implies that $M$ has a $4$-element cocircuit $C^*$ containing $x$ and either $e$ or $f$. Since $x \notin \cl^*(X)$, there exists $y \in C^*$ with $y \notin X \cup \{x\}$. Since $\{a_1, a_2, e, f\}$ is a circuit, it follows by orthogonality that $y \in \cl^*_{M \backslash x}(X)$. Therefore, as $|E(M \backslash x)| \geq |X| + 3$, the matroid $M \backslash x / y$ is $3$-connected. But then the dual of \Cref{contract_then_delete} implies that $M / y$ is $3$-connected, a contradiction to \Cref{contractable_collection}.  Hence $M$ has a detachable pair, as required.
\end{proof}

\begin{lemma} \label{one_triad2}
    Let $M$ be a $3$-connected matroid with no triangles such that $|E(M)| \geq 10$. Let $T^* = \{a_1,a_2,a_3\}$ be a triad of $M$ such that $M$ has no other triads, and let $e,f,g$ be distinct elements of $E(M) - T^*$ such that $\{a_1,a_2,e,f\}$ and $\{a_2,a_3,e,g\}$ are circuits, and $\{e,f,g\}$ is contained in a $4$-element cocircuit $C^*$. Then $M$ has a detachable pair.
\end{lemma}

\begin{proof}
    Suppose $M$ does not have a detachable pair. If $C^* \subseteq T^* \cup \{e,f,g\}$, then $\lambda(T^* \cup \{e,f,g\}) = 2$, and there is a unique element $x \in C^* \cap T^*$, so $(x, T^*-\{x\}, \{C^*-\{x\}\})$ is a contraction certificate. But this contradicts \Cref{one_triad1}. Thus, there exists $h \notin T^* \cup \{e,f,g\}$ such that $C^* = \{e,f,g,h\}$. By \Cref{quad_thingy_deletable}, the matroid $M / h$ is $3$-connected. Therefore, by the dual of \cref{contractable_circuit} and orthogonality, there is a $4$-element circuit $C$ of $M$ containing $\{a_2,h\}$, an element of $\{a_1,a_3\}$, and an element of $\{e,f,g\}$. Now $\lambda(T^* \cup \{e,f,g,h\}) = 2$. Furthermore, by circuit elimination with $\{a_1,a_2,e,f\}$ if $a_1 \in C$ or with $\{a_2,a_3,e,g\}$ if $a_3 \in C$, there is a circuit of $M$ contained in $\{e,f,g,h\}$. This implies that $\{e,f,g,h\}$ is a quad, so $(e,\{f,g,h\},\{\{a_1,a_2,f\},\{a_2,a_3,g\}\})$ is a deletion certificate. But $a_1 \in \cl^*(\{a_2,a_3\})$ and, for all $i \in \{1,2,3\}$, we have that $a_i \in \cl((T^* \cup \{e,f,g,h\})-\{a_i\})$, which contradicts \Cref{deletable_collection_contractable_el}.  So $M$ has a detachable pair, thereby completing the proof of the lemma.
\end{proof}

\begin{lemma} \label{one_triad3}
    Let $M$ be a $3$-connected matroid with no triangles such that $|E(M)| \geq 12$. Let $T^* = \{a_1,a_2,a_3\}$ be a triad of $M$ such that $M$ has no other triads, and let $e,f,g,h$ be distinct elements of $E(M) - T^*$ such that $\{a_1,a_2,e,f\}$ and $\{a_2,a_3,e,g\}$ are circuits, and $M$ has a cocircuit $C^*$ such that $h \in C^*$ and $|C^* \cap \{e,f,g\}| = 2$ and $|C^* \cap T^*| = 1$. Then $M$ has a detachable pair.
\end{lemma}

\begin{proof}
    Suppose $M$ does not have a detachable pair. Let $a_i$ be the unique element of $C^* \cap T^*$. Then $(a_i,T^*-\{a_i\},\{C^*-\{a_i\}\})$ is a contraction certificate.
    To begin with, we observe that if $h \in \cl(T^* \cup \{e,f,g\})$, then $\lambda(T^*\cup \{e, f, g, h\})=2$, in which case, by \Cref{one_triad1}, $M$ has a detachable pair.  So $h \notin \cl(T^* \cup \{e,f,g\})$.

    Next, we show that $M / h$ is $3$-connected. Suppose not. Then $M$ has a vertical $3$-separation $(X,\{h\},Y)$, and, without loss of generality, $T^* \subseteq X$. If $|\{e,f,g\} \cap X| \geq 1$, then $\{e,f,g\} \subseteq \cl(X)$, in which case we may assume that $\{e,f,g\} \subseteq X$. This implies that $h \in \cl^*(X)$, a contradiction. Thus, $\{e,f,g\} \subseteq Y$. But $a_i \in \cl^*(Y \cup \{h\})$ so $\lambda(Y \cup \{h,a_i\}) = 2$. Furthermore, the circuits $\{a_1,a_2,e,f\}$ and $\{a_2,a_3,e,g\}$ imply that $\lambda(Y \cup \{h\} \cup T^*) < 2$, and so $|X-T^*| \le 1$. If $|X-T^*|=0$, then $h\in \cl(T^*)$, a contradiction. So $|X-T^*|=1$. Let $z$ be the unique element of $X - T^*$. Then either $z \in \cl(T^*)$ or $z \in \cl^*(T^*)$. But the former case implies that $T^* \cup \{z\}$ is a circuit, which contradicts orthogonality with $C^*$, and the latter case implies that $r^*(T^* \cup \{z\}) = 2$, which contradicts the dual of \cref{no_segment}. Thus, $M / h$ is $3$-connected.
	
    Choose $j,k$ such that $\{i,j,k\}=\{1,2,3\}$.  By the dual of \cref{contractable_circuit}, and since $h \notin \cl(T^* \cup \{e,f,g\})$, there are circuits $\{a_i,a_j,h,f'\}$ and $\{a_i,a_k,h,g'\}$ for some $f',g' \in E(M) - (T^* \cup \{e,f,g,h\})$.  Furthermore, $f' \neq g'$, for otherwise, by circuit elimination, $h \in \cl(T^* \cup \{e,f,g\})$. \Cref{circuit_so_deletable} implies that $M \backslash f'$ is $3$-connected, and, in turn, \cref{deletable_circuit_gives_cocircuit} implies that $M$ has a $4$-element cocircuit $D^*$ containing either $\{f',h\}$ or $\{f',g'\}$. By \Cref{one_triad2}, the cocircuit $D^*$ does not contain $\{f',g',h\}$, so orthogonality with $\{a_i,a_j,h,f'\}$ and $\{a_i,a_k,h,g'\}$ implies that $D^*$ contains an element of $T^*$. Now, orthogonality with $\{a_1,a_2,e,f\}$ or $\{a_2,a_3,e,g\}$ implies that $D^*$ contains another element of $T^* \cup \{e,f,g\}$, so $\lambda(T^* \cup \{e,f,g,h,f',g'\}) = 2$.  But, as $|E(M)| \ge 12$, this contradicts \Cref{one_triad1}.  We deduce that $M$ has a detachable pair.
\end{proof}

\begin{lemma} \label{one_triad4}
    Let $M$ be a $3$-connected matroid with no triangles such that $|E(M)| \geq 12$. Let $T^*$ be a triad of $M$, and suppose $M$ has no other triads. Let $e \notin T^*$ such that $M / e$ is $3$-connected. Then $M$ has a detachable pair.
\end{lemma}

\begin{proof}
    Suppose $M$ does not have a detachable pair. By the dual of \Cref{contractable_circuit}, there are $4$-element circuits $\{a_1,a_2,e,f\}$ and $\{a_2,a_3,e,g\}$ for some labelling $T^* = \{a_1,a_2,a_3\}$ and elements $f,g \notin T^* \cup \{e\}$. Furthermore, $f \neq g$, as $e \notin \cl(T^*)$. By \Cref{circuit_so_deletable}, $M \backslash f$ is $3$-connected. By \Cref{deletable_circuit_gives_cocircuit}, there is a $4$-element cocircuit $C^*$ of $M$ containing either $\{e,f\}$ or $\{f,g\}$.
    By \Cref{one_triad2}, $\{e,f,g\} \not\subseteq C^*$. Therefore, $C^*$ contains an element of $T^*$, by orthogonality. Furthermore, \Cref{one_triad3} implies that $|C^* \cap T^*| \neq 1$. Therefore, $|C^* \cap T^*| = 2$. If $\{f,g\} \subseteq C^*$, then $\lambda(T^* \cup \{f,g\}) = 2$. But $e \in \cl(T^* \cup \{f,g\})$, which contradicts the fact that $M / e$ is $3$-connected. So $\{e,f\} \subseteq C^*$ and $g \notin C^*$.

    \Cref{circuit_so_deletable} implies that $M \backslash g$ is $3$-connected, so, by \Cref{deletable_circuit_gives_cocircuit}, there is a $4$-element cocircuit $D^*$ of $M$ containing $g$ and either $e$ or $f$. Again, \cref{one_triad2,one_triad3} imply that $|D^* \cap T^*| = 2$. If $C^* \cap T^* = D^* \cap T^*$, then cocircuit elimination implies that $\{a_i,e,f,g\}$ is a cocircuit for some $i\in \{1, 2, 3\}$, a contradiction to \Cref{one_triad2}. Therefore, there is a unique element $a_i$ that is contained in both $C^* \cap T^*$ and $D^* \cap T^*$. Thus, $(a_i,T^*-\{a_i\},\{C^*-\{a_i\},D^*-\{a_i\}\})$ is a contraction certificate and $\lambda(T^* \cup \{e,f,g\}) = 2$, which contradicts \Cref{one_triad1}. Hence $M$ has a detachable pair.
\end{proof}

\begin{lemma} \label{one_triad}
    Let $M$ be a $3$-connected matroid with no triangles and precisely one triad, such that $|E(M)| \geq 12$. Then $M$ has a detachable pair.	
\end{lemma}

\begin{proof}
    Let $T^*$ be the unique triad of $M$, and suppose $M$ does not have a detachable pair. By \Cref{one_triad4}, for all $x \notin T^*$, we have that $M / x$ is not $3$-connected. Let $e \in E(M)-\cl(T^*)$. Then there is a vertical $3$-separation $(X,\{e\},Y)$ of $M$ such that $T^* \subseteq X$. Since $e \notin \cl(T^*)$, there exists an element $f \in X - T^*$. By \Cref{annoying_elements}, there is a $4$-element cocircuit $C^*$ of $M$ containing $\{e,f\}$ and exactly one element of $T^*$. Now $|C^* \cap X| \in \{2,3\}$. If $|C^* \cap X| = 3$, then $e \in \cl^*(X)$, a contradiction. So $|C^* \cap X| = 2$. But there is a unique element $g$ of $C^* \cap Y$, and $g \in \cl^*(X \cup \{e\})$, so $M / g$ is $3$-connected, a contradiction.  We deduce that $M$ has a detachable pair.
\end{proof}

\subsection*{Putting it together}
We now prove \cref{no_4_element_fans}.

\begin{proof}[Proof of \Cref{no_4_element_fans}]
    Suppose $M$ does not have a detachable pair. If $M$ has no triangles or triads whatsoever, then $M$ is a spike by \Cref{detpairs_notriangles}. If $M$ has exactly one triad and no triangles, or $M$ has exactly one triangle and no triads, then $M$ has a detachable pair by \cref{one_triad} or its dual. If $M$ has exactly one triangle and exactly one triad, then $M$ has a detachable pair by \Cref{triad_triangle}.
    Thus $M$ either has two distinct triangles, or two distinct triads.

    Suppose that $M$ has disjoint triads $T_1^*$ and $T_2^*$. If there exists an element $e \in E(M)$ such that $e$ is not contained in a triad and $M / e$ is $3$-connected, then, by \cref{quad_petal}, $M$ is a quasi-triad-paddle with a quad or near-quad petal. Otherwise, no such element $e$ exists, and thus, for all $x \in E(M)$, if $x$ is not contained in a triad, then $M / x$ is not $3$-connected. If $M$ has an element that is not contained in a triangle or a triad, then $M$ is a hinged triad-paddle by \Cref{not_quad_petal}. If every element of $M$ is contained in a triangle or a triad, then, by \Cref{k3m}, $M$ is either a triad-paddle or a tri-paddle-copaddle.
    Thus we may assume that $M$ has no disjoint triads and, dually, no disjoint triangles.
	
    We may also assume, up to duality, that $M$ has distinct triads $T_1^*$ and $T_2^*$.  Then $T_1^*$ meets $T_2^*$. By the dual of \cref{no_segment}, $|T^*_1\cap T^*_2|=1$ and, by \Cref{intersecting_no_other_triads}, there are no other triads of $M$.  Thus, $M$ has exactly five elements that are contained in a triad. Since $M$ has no pair of disjoint triangles, \cref{no_segment} and the dual of \Cref{intersecting_no_other_triads} imply that there are at most five elements of $M$ that are contained in a triangle.  Moreover, \cref{intersecting_plus_contractable} implies that, for all $x \in E(M)-(T^*_1 \cup T^*_2)$, the matroid $M / x$ is not $3$-connected. Hence, by \cref{intersecting_two_elements}, $M$ has at most one element that is not contained in a triangle or a triad. But now $|E(M)| \leq 11$. This contradiction completes the proof of the theorem.
\end{proof}

\Cref{detachable_main} now follows by combining \cref{disjoint_fans_with_like_ends,intersecting_fans_detachable,single_fans,no_4_element_fans}.

\begin{proof}[Proof of \cref{detachable_main}]
    If $M$ has disjoint maximal fans $F_1$ and $F_2$ with like ends, where $|F_1| \ge 4$ and $|F_2| \ge 3$, then, by \cref{disjoint_fans_with_like_ends}, either (i), (iv), (vi), or (viii)(a) holds.
    Otherwise, $M$ has no disjoint maximal fans $F_1$ and $F_2$ with like ends, where $|F_1| \ge 4$ and $|F_2| \ge 3$.
    Suppose that $M$ has distinct maximal fans $F_1$ and $F_2$ with $|F_1| \ge 4$ and $|F_2| \ge 3$, where $F_1 \cap F_2 \neq \emptyset$. By \cref{intersecting_fans_detachable}, either (i), (iii), (iv), or (v) holds.
    Now we may assume that if $M$ has distinct maximal fans $F_1$ and $F_2$ with $|F_1| \ge 4$ and $|F_2| \ge 3$, then $F_1$ and $F_2$ are disjoint and, up to duality, both ends of $F_1$ are triangles, and both ends of $F_2$ are triads.
    Then, if $M$ has a maximal fan with length at least four, \cref{single_fans} implies that either (i), (ii), (iv), or (viii)(b) holds.
    Finally, we may assume that $M$ has no $4$-element fans.
    Then, by \cref{no_4_element_fans}, either (i), (iv), (vi), (vii), (viii)(c), or (viii)(d) holds.
\end{proof}

\section{Proof of \texorpdfstring{\cref{detachable_graph}}{Theorem 1.3}}
\label{graph_proof}

It remains only to prove \cref{detachable_graph}. The following lemma, whose proof is straightforward and omitted, will be useful.

\begin{lemma} \label{build_fan_graph}
    Let $G$ be a simple $3$-connected graph such that $M(G)$ has a fan $F = (e_1,e_2,\ldots,e_{|F|})$ where $|F| \geq 4$ and $\{e_{|F|-2},e_{|F|-1},e_{|F|}\}$ is a triad.
    Let $G' = G /e_{|F|-2}\backslash e_{|F|-1}$, and let $h$ be the vertex of $G'$ that is incident to $e_{|F|-3}$ but not $e_{|F|}$.
    Then $G$ can be constructed from $G'$ by subdividing the edge $e_{|F|}$ to introduce a vertex $x$, and adding an edge incident with $x$ and $h$.
\end{lemma}

\noindent We also remind the reader that a matroid is graphic if and only if it has no minor isomorphic to $U_{2, 4}$, $F_7$, $F^*_7$, $M^*(K_5)$, and $M^*(K_{3, 3})$~\cite{Tutte1959}.

\begin{proof}[Proof of \cref{detachable_graph}]
    Let $G$ be a simple $3$-connected graph with no detachable pairs such that $|E(G)| \geq 13$. Then $M(G)$ is a $3$-connected matroid with no detachable pairs, and thus $M(G)$ is one of the matroids listed in \Cref{detachable_main}. 
    If $M(G)$ is a wheel, then $G$ is a wheel, whereas $M(G)$ is not a whirl, as a whirl is not graphic, as it has a $U_{2,4}$-minor.

    Next suppose that $M(G)$ is an accordion. Then there is a partition $(L,F,R)$ of $E(M)$ such that $(e_1,e_2,\ldots,e_{|F|})$ is a fan ordering of $F$, where $F$ is even, $|F| \ge 4$, and $\{e_1,e_2,e_3\}$ a triangle. We will show that in this case $G$ is a mutant wheel. After contracting $e_{|F|-2}$ and deleting $e_{|F|-1}$ in $G$, the set $F-\{e_{|F|-2},e_{|F|-1}\}$ is a fan of length $|F|-2$ in the cycle matroid. Repeating in this way, let
$$G' = G / \{e_{|F|-2},e_{|F|-4},\dotsc,e_2\} \backslash \{e_{|F|-1},e_{|F|-3},\dotsc,e_3\}.$$

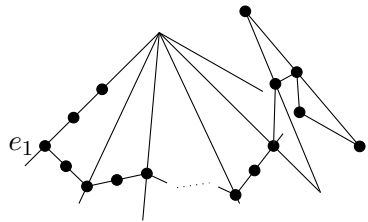
First assume $L=\{g_2, g_3\}$ is a left-hand triangle-type end of $F$. If $R$ is a right-hand fan-type end, then there is a labelling $R=\{h_2, h_3, h_4, h_5\}$ such that $(e_{|F|}, h_2, h_3, h_4, h_5)$ is a fan ordering of $R\cup \{e_{|F|}\}$.  Then $(M(G')/h_4)|\{e_1, g_2, g_3, h_5\}\cong U_{2, 4}$, by the dual of \cref{accordlhfan}, so $M(G')$, and therefore $M(G)$, is not graphic. So $R$ is not a right-hand fan-type end. If $R$ is a right-hand quad-type end, then, by the dual of \cref{accordlhquad}, there is a labelling $R=\{c_1, c_2, d_1, d_2\}$ such that $\sqcap(\{d_1, d_2\}, L)=1$. Then $(M(G')/d_2)|\{e_1, g_2, g_3, d_1\}\cong U_{2, 4}$, so $M(G')$, and thus $M(G)$, is not graphic. If $R=\{h_2, h_3\}$ is a right-hand triad-end of $F$, then $(M(G')/h_3)|\{e_1, g_2, g_3, h_2\}\cong U_{2, 4}$, by the dual of \Cref{accordlhtri}, and so again $M(G)$ is not graphic. Hence $L$ is not a left-hand triangle-end of $M(G)$ and, dually, $R$ is not a right-hand triad-type end of $F$.

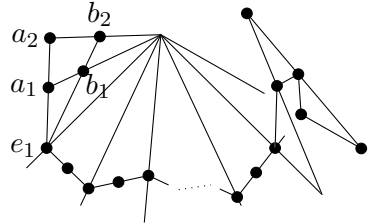
Now assume $L$ is a left-hand quad-type end of $F$. Suppose $R$ is a right-hand fan-type end with labelling $R=\{h_2, h_3, h_4, h_5\}$ such that $(e_{|F|}, h_2, h_3, h_4, h_5)$ is a fan ordering of $R\cup \{e_{|F|}\}$. Note that the restriction of $M(G') / h_5$ to $L \cup \{e_1, h_4\}$ is isomorphic to $M(K_4)$, where $e_1$ and $h_4$ correspond to non-adjacent edges of the $K_4$. By contracting $h_2$ from $M(G') / h_5$, the resulting matroid retains the $M(K_4)$ restriction, and has a triangle $\{e_1,e_{|F|},h_4\}$. Thus, $M(G')$ has a minor isomorphic to the Fano matroid $F_7$, the non-Fano matroid $F^-_7$, or the matroid $F^=_7$ obtained from $F^-_7$ by relaxing a circuit-hyperplane. On the other hand, if $R$ is a right-hand quad-type end, then, by \cref{accordlhquad}, there is a labelling $R=\{c_1, c_2, d_1, d_2\}$ such that $\sqcap(\{c_1,c_2\}, L)=\sqcap(\{d_1,d_2\}, L)=1$, and $M(G')/\{d_2, e_{|F|}\}$ also has a minor isomorphic to one of $F_7$, $F^-_7$, and $F^=_7$. Since $F_7$ is neither graphic nor cographic, and each of $F^-_7$ and $F^=_7$ has a $U_{2, 4}$ minor, it follows that $L$ is not a left-hand quad-type end of $F$ and, dually, $R$ is not a right-hand quad-type end of $F$.

Lastly, assume $L=\{g_2, g_3, g_4, g_5\}$ is a left-hand fan-type end of $F$ and $(e_1, g_2, g_3, g_4, g_5)$ is fan ordering of $L\cup \{e_1\}$, and $R=\{h_2, h_3, h_4, h_5\}$ is a right-hand fan-type end and $(e_{|F|}, h_2, h_3, h_4, h_5)$ is a fan ordering of $R\cup \{e_{|F|}\}$. If $r_{M(G')}(\{g_3, g_4, g_5, e_{|F|}, h_2\})=4$, then $(M(G')/\{h_2, h_4\})|\{e_1, e_{|F|}, g_5, h_5\}\cong U_{2, 4}$, so $M(G)$ is not graphic. Thus $r_{M(G')}(\{g_3, g_4, g_5, e_{|F|}, h_2\})=3$ and $G'$ is a mutant wheel with edges labelled as shown in \cref{mutant}. \Cref{build_fan_graph} now implies that $G$ is a mutant wheel.

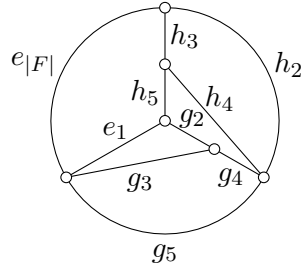
\begin{figure}
\begin{center}
\begin{tikzpicture}
\draw (160:1.5) arc (160:520:1.5);
\draw (0:0) -- (90:0.75) node [midway, left, xshift=2] {$h_5$};
\draw (90:0.75) -- (90:1.5) node [midway, right, xshift=-2] {$h_3$};
\draw (0:0) -- (-30:0.75) node [midway, above, xshift=1, yshift=-1] {$g_2$};
\draw (-30:0.75) -- (-30:1.5) node [midway, below, yshift=2, xshift=-3] {$g_4$};
\draw (0:0) -- (-150:1.5) node [midway, above] {$e_1$};

\draw (-210:1.5) node[left] {$e_{|F|}$};
\draw (-90:1.5) node[below] {$g_5$};
\draw (30:1.5) node[right] {$h_2$};

\draw (90:0.75) -- (-30:1.5) node [midway, above, xshift=2] {$h_4$};
\draw (-30:0.75) -- (-150:1.5) node [midway, below] {$g_3$};

\draw[fill=white] (0:0) circle (2pt);
\draw[fill=white] (90:1.5) circle (2pt);
\draw[fill=white] (-30:1.5) circle (2pt);
\draw[fill=white] (-150:1.5) circle (2pt);

\draw[fill=white] (90:0.75) circle (2pt);
\draw[fill=white] (-30:0.75) circle (2pt);
\end{tikzpicture}
\end{center}
\caption{The graph $G'$ if $L=\{g_2, g_3, g_4, g_5\}$ is a left-hand fan-type end, $R=\{h_2, h_3, h_4, h_5\}$ is a right-hand fan-type end, and $r_{M(G')}(\{g_3, g_4, g_5, e_{|F|}, h_2\})=3$.}
\label{mutant}
\end{figure}

Next suppose that $M(G)$ is an even-fan-spike (without a tip and cotip), with partition $\Phi$. Assume $M(G)$ is a non-degenerate even-fan-spike, so $\Phi = (P_1,P_2,\ldots,P_m)$, with $m \geq 3$, such that $P_i$ is an even fan of length at least two for all $i \in [m]$. Let $P_i$ have fan ordering $(p_1^i,p_2^i,\ldots,p_{|P_i|}^i)$ such that either $|P_i| = 2$ or $\{p_1^i,p_2^i,p_3^i\}$ is a triad. Observe that if $|P_i| > 2$, then $M(G) / p_3^i \backslash p_2^i$ is a non-degenerate even-fan-spike with partition $\Phi = (P_1,P_2,\ldots, P_i-\{p_3^i,p_2^i\},\ldots,P_m)$. Furthermore, if $|P_i| = 2$ and $m \geq 4$, then $M(G) / p_1^i \backslash p_2^i$ is a non-degenerate even-fan-spike with partition $\Phi = (P_1,P_2,\ldots,P_{i-1},P_{i+1},\ldots,P_m)$.,

Say $m = 3$. Since $|E(M)| \geq 13$, there exists $i \in [m]$ such that $|P_i| > 2$. Without loss of generality, assume that $|P_1| > 2$. It follows that $M(G)$ has a minor $N$ that is an even-fan-spike with partition $\Phi = (P_1,P_2,P_3)$ such that $|P_1| = 4$ and $|P_2| = |P_3| = 2$. But $N / p_1^1$ is isomorphic to a rank-$3$ spike with tip, which is either non-binary or isomorphic to $F_7$. Either case implies $M(G)$ is not graphic, a contradiction. So $m \geq 4$, in which case $M(G)$ has a minor isomorphic to a tipless rank-$4$ spike. Contracting any element of this rank-$4$ spike produces a rank-$3$ spike with tip, again a contradiction. Hence, $M(G)$ is a degenerate even-fan-spike.

Let $(P,Q)$ be the partition of the degenerate even-fan-spike, where $P = (p_1,p_2,\ldots,p_{|P|})$ and $Q = (q_1,q_2,\ldots,q_{|Q|})$ are even fans such that $\{p_1,p_2,p_3\}$ and $\{q_1,q_2,q_3\}$ are both triads. Let $G'=G/p_{|P|-1} \backslash p_{|P|-2} / p_{|P|-3}\backslash p_{|P|-4} \cdots / p_5 \backslash p_4$, and let $G'' = G' / q_{|Q|-1} \backslash q_{|Q|-2} / q_{|Q|-3} \backslash q_{|Q|-4} \cdots / q_5 \backslash q_4$. Since $M(G)$ does not have a $U_{2,4}$-minor, it follows that $G''$ is isomorphic to the rank-$4$ wheel. \Cref{build_fan_graph} now implies that $G$ is a warped wheel. To illustrate, a warped wheel with $|P|=6=|Q|$ is shown in \cref{warped}(A).

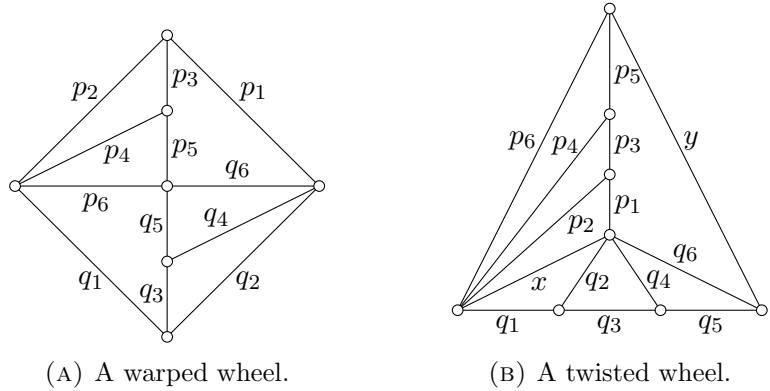
\begin{figure}
\center
	\begin{subfigure} {0.45\textwidth}
		\centering
		\begin{tikzpicture}
		\coordinate (a) at (-2,0);
		\coordinate (b) at (0,0);
		\coordinate (c) at (2,0);
		
		\coordinate (d1) at (0,2);
		\coordinate (d2) at (0,1.25);
		\coordinate (d3) at (0,1);
		\coordinate (d4) at (0,0.25);
		\coordinate (d5) at (0,1);
		
		\coordinate (e1) at (0,-2);
		\coordinate (e2) at (0,-1.25);
		\coordinate (e3) at (0,-1);
		\coordinate (e4) at (0,-0.25);
		\coordinate (e5) at (0,-1);
		
		\draw (a) -- (d1) node [midway, above, xshift=-2] {$p_2$};
		\draw (a) -- (e1) node [midway, below] {$q_1$};
		
		\draw (c) -- (d1) node [midway, above, xshift=4] {$p_1$};
		\draw (c) -- (e1) node [midway, below, xshift=2] {$q_2$};

		\draw (a) -- (b) node [midway, below, xshift=2, yshift=1] {$p_6$};
		\draw (c) -- (b) node [midway, above, xshift=-2, yshift=-1] {$q_6$};
		\draw (a) -- (d5) node [midway, below, xshift=10, yshift=5] {$p_4$};
		\draw (c) -- (e5) node [midway, above, xshift=-10, yshift=-5] {$q_4$};
		\draw (d1) -- (d5) node [midway, right, xshift=-2, yshift=-2] {$p_3$};
		\draw (d5) -- (b) node [midway, right, xshift=-2] {$p_5$};
		\draw (e1) -- (e5) node [midway, left, xshift=3, yshift=2] {$q_3$};
		\draw (e5) -- (b) node [midway, left, xshift=3] {$q_5$};
		
		\draw[fill=white] (a) circle (2pt);
		\draw[fill=white] (b) circle (2pt);
		\draw[fill=white] (c) circle (2pt);
		\draw[fill=white] (d1) circle (2pt);
		\draw[fill=white] (e1) circle (2pt);
		\draw[fill=white] (e5) circle (2pt);
		\draw[fill=white] (d5) circle (2pt);
		\end{tikzpicture}
	\subcaption{A warped wheel.}
	\end{subfigure}
	\begin{subfigure} {0.45\textwidth}
		\centering
		\begin{tikzpicture}
		\coordinate (a) at (-2,0);
		\coordinate (b) at (2,0);
		\coordinate (c) at (0,4);
		\coordinate (d) at (0,1);
		
		\coordinate (ac4) at (0,2.6);
		\coordinate (ac5) at (0,1.8);
		
		\coordinate (bd4) at ($(a)!0.333!(b)$);
		\coordinate (bd5) at ($(a)!0.667!(b)$);		
		
		\draw (d) -- (a) node [midway, below, xshift=2, yshift=2] {$x$};
		\draw (d) -- (b) node [midway, above] {$q_6$};
		
		\draw (a) -- (c) node [midway, above, xshift=-4] {$p_6$};
		\draw (c) -- (b) node [midway, above, xshift=2] {$y$};

		\draw (d) -- (bd4) node [midway, below, xshift=5, yshift=4] {$q_2$};
		\draw (d) -- (bd5) node [midway, above, xshift=9, yshift=-10] {$q_4$};
		\draw (a) -- (ac4) node [midway, above, xshift=12, yshift=18] {$p_4$};
		\draw (a) -- (ac5) node [midway, below, xshift=18, yshift=14] {$p_2$};
		\draw (c) -- (ac4) node [midway, right, xshift=-2, yshift=-5] {$p_5$};
		\draw (ac4) -- (ac5) node [midway, right, xshift=-2] {$p_3$};
		\draw (ac5) -- (d) node [midway, right, xshift=-2] {$p_1$};
		\draw (a) -- (bd4) node [midway, below, yshift=2] {$q_1$};
		\draw (bd4) -- (bd5) node [midway, below, yshift=2] {$q_3$};
		\draw (bd5) -- (b) node [midway, below, yshift=2] {$q_5$};
		
		\draw[fill=white] (a) circle (2pt);
		\draw[fill=white] (b) circle (2pt);
		\draw[fill=white] (c) circle (2pt);
		\draw[fill=white] (d) circle (2pt);
		
		\draw[fill=white] (ac4) circle (2pt);
		\draw[fill=white] (ac5) circle (2pt);
		\draw[fill=white] (bd4) circle (2pt);
		\draw[fill=white] (bd5) circle (2pt);
		\end{tikzpicture}
	\subcaption{A twisted wheel.}
	\end{subfigure}	
\caption{(A) A warped wheel. (B) A twisted wheel.}
\label{warped}
\end{figure}

Now suppose that $M(G)$ is an even-fan-spike with tip $x$ and cotip $y$. Then $M(G) \backslash x / y$ is an even-fan-spike. Therefore, $G \backslash x / y$ is a warped wheel. A routine check shows that $G$ is a twisted wheel. A twisted wheel with $|P|=6=|Q|$ is shown in \cref{warped}(B).

Next suppose that, for $M \in \{M(G),M^*(G)\}$, the matroid $M$ is an even-fan-paddle. First, assume $M$ is non-degenerate with partition $\Phi$. Then $\Phi=(P_1,P_2,\ldots,P_m)$, with $m \geq 3$, and there is an element $x \in P_m$ such that $P_i \cup \{x\}$ is an even fan of length at least four for all $i \in [m]$. It is easily checked that when $M=M(G)$, the graph $G$ is a multi-wheel and, furthermore, $M^*$ is not graphic, since $M|(P_1 \cup P_2 \cup P_m)$ has a $M(K_{3,3})$-minor.  So $M^*(G)$ is not an even-fan-paddle.
Now assume $M$ is degenerate with partition $(P_1, P_2, \{x, y\})$, where $P_1\cup \{x\}$ and $P_2\cup \{x\}$ are even fans of length at least four. If $M=M(G)$, then $G$ is a degenerate multi-wheel. If $M=M^*(G)$, then $G$ is a stretched wheel.

If $M(G)$ is a triad-paddle, then $G\cong K_{3, m}$, where $m\ge 5$ as $|E(M(G))|\ge 13$. Note that $M^*(K_{3,m})$ is not graphic, for $m\ge 3$, so $M^*(G)$ is not a triad-paddle. Now suppose that $M \in \{M(G),M^*(G)\}$ is a hinged triad-paddle with partition $(P_1, P_2, \ldots, P_m, \{x\})$, for some $m\ge 3$. Then, as $P_i$ is a triad but $P_i\cup \{x\}$ is not a $4$-element fan, for $i \in \{1,2\}$, it follows that $M$ has a $U_{2, 4}$-minor, a contradiction.

Suppose now that $M \in \{M(G),M^*(G)\}$ is a tri-paddle-copaddle with partition $(P_1, P_2, \ldots, P_s, Q_1, Q_2, \ldots, Q_t)$, for some $s,t\ge 2$. Then, by considering $M \backslash (P_3 \cup \dotsm \cup P_{s}) / (Q_3 \cup \dotsm \cup Q_{t})$, it is easily checked that $M$ has either a $U_{2, 4}$-minor or both a $M(K_{3, 3})$- and $M^*(K_{3, 3})$-minor, contradicting that $M$ is graphic or cographic.

Lastly, suppose that $M \in \{M(G),M^*(G)\}$ is a quasi-triad-paddle with an augmented-fan, co-augmented-fan, quad, or near-quad petal. Then $M$ has a $M(K_{3,3})$-minor, so $M = M(G)$.
If $M(G)$ has an augmented-fan petal or a co-augmented-fan petal, then it is easily seen that $G$ is isomorphic to $K^a_{3, m}$ or $K^b_{3, m}$, respectively.
It remains to consider when $M(G)$ has a quad or near-quad petal.
Let $(P_1, P_2, \ldots, P_m)$ be the partition of the quasi-triad-paddle, for $m\ge 3$.
Then, by considering $M(G)|(P_1\cup P_2\cup P_m)$, it is easily checked that $M(G)$ has a minor isomorphic to either $U_{2, 4}$ or $F_7$, contradicting that $M(G)$ is graphic. This completes the proof of \Cref{detachable_graph}.
\end{proof}

\section*{Acknowledgements}
We thank the referees for their thorough reading of the paper, and their helpful comments.

\bibliographystyle{plain}
\bibliography{detach.bib}

\begin{thebibliography}{10}

\bibitem{bixby}
R.E.\ Bixby.
\newblock A simple theorem on 3-connectivity.
\newblock {\em Linear Algebra and its Applications}, 45:123--126, 1982.

\bibitem{almostfrag1}
N.\ Brettell, B.\ Clark, J.\ Oxley, C.\ Semple, and G.\ Whittle.
\newblock Excluded minors are almost fragile.
\newblock {\em Journal of Combinatorial Theory, Series B}, 140:263--322, 2020.

\bibitem{almostfrag2}
N.\ Brettell, J.\ Oxley, C.\ Semple, and G.\ Whittle.
\newblock Excluded minors are almost fragile~{II}: Essential elements.
\newblock {\em Journal of Combinatorial Theory, Series B}, 163:272--307, 2023.

\bibitem{2regexminors}
N.\ Brettell, J.\ Oxley, C.\ Semple, and G.\ Whittle.
\newblock The excluded minors for 2- and 3-regular matroids.
\newblock {\em Journal of Combinatorial Theory, Series B}, 163:133--218, 2023.

\bibitem{BP23}
N.\ Brettell and R.\ Pendavingh.
\newblock Computing excluded minors for classes of matroids representable over
  partial fields.
\newblock {\em Electronic Journal of Combinatorics}, 31:P3.20, 2024.

\bibitem{BS14}
N.\ Brettell and C.\ Semple.
\newblock A splitter theorem relative to a fixed basis.
\newblock {\em Annals of Combinatorics}, 18:1--20, 2014.

\bibitem{detpairs1}
N.\ Brettell, G.\ Whittle, and A.\ Williams.
\newblock {$N$}-detachable pairs in 3-connected matroids~{I}: Unveiling {$X$}.
\newblock {\em Journal of Combinatorial Theory, Series B}, 141:295--342, 2020.

\bibitem{detpairs2}
N.\ Brettell, G.\ Whittle, and A.\ Williams.
\newblock {$N$}-detachable pairs in 3-connected matroids~{II}: Life in {$X$}.
\newblock {\em Journal of Combinatorial Theory, Series B}, 149:222--271, 2021.

\bibitem{detpairs3}
N.\ Brettell, G.\ Whittle, and A.\ Williams.
\newblock {$N$}-detachable pairs in 3-connected matroids~{III}: The theorem.
\newblock {\em Journal of Combinatorial Theory, Series B}, 153:223--290, 2022.

\bibitem{quaternary}
J.F.\ Geelen, A.M.H.\ Gerards, and A.\ Kapoor.
\newblock The excluded minors for {GF}(4)-representable matroids.
\newblock {\em Journal of Combinatorial Theory, Series B}, 79:247--299, 2000.

\bibitem{nearreg}
R.\ Hall, D.\ Mayhew, and S.H.M.\ {van Zwam}.
\newblock The excluded minors for near-regular matroids.
\newblock {\em European Journal of Combinatorics}, 32:802--830, 2011.

\bibitem{oxley!!!}
J.\ Oxley.
\newblock {\em Matroid Theory}.
\newblock Oxford University Press, second edition, 2011.

\bibitem{deltay}
J.\ Oxley, C.\ Semple, and D.\ Vertigan.
\newblock Generalized {$\Delta$-$Y$} exchange and $k$-regular matroids.
\newblock {\em Journal of Combinatorial Theory, Series B}, 79:1--65, 05 2000.

\bibitem{flowers}
J.\ Oxley, C.\ Semple, and G.\ Whittle.
\newblock The structure of the 3-separations of 3-connected matroids.
\newblock {\em Journal of Combinatorial Theory, Series B}, 92:257--293, 2004.

\bibitem{oxleywufans}
J.\ Oxley and H.\ Wu.
\newblock On the structure of 3-connected matroids and graphs.
\newblock {\em European Journal of Combinatorics}, 21:667--688, 2000.

\bibitem{seymoursplitter}
P.D.\ Seymour.
\newblock Decomposition of regular matroids.
\newblock {\em Journal of Combinatorial Theory, Series B}, 28:305--359, 1980.

\bibitem{tansplitter}
J.J.-M.\ Tan.
\newblock {\em Matroid $3$-connectivity}.
\newblock PhD thesis, Carleton University, 1981.

\bibitem{Tutte1959}
W.T.\ Tutte.
\newblock Matroids and graphs.
\newblock {\em Transactions of the American Mathematical Society}, 90:527--552,
  1959.

\bibitem{wheelsandwhirls}
W.T.\ Tutte.
\newblock Connectivity in matroids.
\newblock {\em Canadian Journal of Mathematics}, 18:1301–1324, 1966.

\bibitem{whittlestabs}
G.\ Whittle.
\newblock Stabilizers of classes of representable matroids.
\newblock {\em Journal of Combinatorial Theory, Series B}, 77:39--72, 1999.

\bibitem{williamsdetpairs}
A.\ Williams.
\newblock {\em Detachable Pairs in $3$-Connected Matroids}.
\newblock PhD thesis, Victoria University of Wellington, 2015.

\end{thebibliography}
\end{document}